\documentclass{amsart}
\usepackage[margin=26.6mm]{geometry}
\usepackage[parfill]{parskip}
\usepackage{setspace}
\usepackage{amsmath,amsfonts,amscd,amssymb,bm,amsthm, mathrsfs}
\usepackage{tikz-cd}
\usetikzlibrary{positioning,arrows,calc}
\usepackage{movie15}
\usepackage{graphicx}
\usepackage{comment}
\usepackage[pdfencoding=auto,colorlinks,
            linkcolor=red,
            anchorcolor=blue,
            citecolor=green]{hyperref}
\usepackage{bookmark}

\usepackage[backend=bibtex, maxnames=4]{biblatex}\bibliography{Pi1}

\usepackage[all]{xy}
\usepackage{tikz}
\definecolor{myred}{RGB}{183,18,52}
\definecolor{myyellow}{RGB}{254,213,1}
\definecolor{myblue}{RGB}{0,80,198}
\definecolor{mygreen}{RGB}{0,155,72}
\newcommand{\mM}{\mathcal{M}}
\newcommand{\mS}{\mathcal{S}}

\newcommand{\mJ}{\mathcal{J}}
\newcommand{\mT}{\mathcal{T}}
\newcommand{\mU}{\mathcal{U}}
\newcommand{\mC}{\mathcal{C}}
\newcommand{\mK}{\mathcal{K}}
\newcommand{\mH}{\mathcal{H}}
\newcommand{\mX}{\mathcal{X}}

\newcommand{\msC}{\mathscr{C}}
\newcommand{\mY}{\mathcal{Y}}
\newcommand{\mZ}{\mathcal{Z}}
\newcommand{\mO}{\mathcal{O}}

\newcommand{\mA}{\mathcal{A}}
\newcommand{\mG}{\mathcal{G}}

\newcommand{\mE}{\mathcal{E}}
\newcommand{\mQ}{\mathcal{Q}}
\newcommand{\mB}{\mathcal{B}}

\newcommand{\aA}{\mathbb{A}}

\newcommand{\FF}{\mathbb{F}}

\newcommand{\DD}{\mathbb{D}}
\newcommand{\EE}{\mathbb{E}}
\newcommand{\ZZ}{\mathbb{Z}}

\newcommand{\CC}{\mathbb{C}}
\newcommand{\QQ}{\mathbb{Q}}
\newcommand{\RR}{\mathbb{R}}

\newcommand{\PP}{\mathbb{P}}
\newcommand{\CP}{\mathbb{C}P}
\newcommand{\RP}{\mathbb{R}P}

\newcommand{\Conf}{\mbox{Conf}}
\newcommand{\Diff}{\mbox{Diff}}

\newcommand{\w}{\omega}

\newcommand{\ov}{\overline}

\newcommand{\wt}{\widetilde}

\newtheorem{thm}{Theorem}[section]
\newtheorem{dfn}[thm]{Definition}
\newtheorem{defn}[thm]{Definition}
\newtheorem{cor}[thm]{Corollary}
\newtheorem{lma}[thm]{Lemma}
\newtheorem{prp}[thm]{Proposition}

\newtheorem{rmk}[thm]{Remark}

\begin{document}
\title[Symplectic $(-2)$-spheres and the Symplectomorphism group]{Symplectic $(-2)$-spheres and the symplectomorphism group of small rational 4-manifolds, II}
\author{ Jun Li, Tian-Jun Li, Weiwei Wu}
\address{Department  of Mathematics\\  University of Michigan\\ Ann Arbor, MI 48109}
\email{lijungeo@umich.edu}
\address{School  of Mathematics\\  University of Minnesota\\ Minneapolis, MN 55455}
\email{tjli@math.umn.edu}
\address{School  of Mathematics\\  University of Georgia\\ Athens, GA 30602}
\email{weiwei@math.uga.edu }
\date{\today}
\begin{abstract}
  For $(\CC P^2  \# 5{\overline {\CC P^2}},\omega)$, let $N_{\w}$ be the number of $(-2)$-symplectic spherical homology classes.
  We completely determine  the Torelli symplectic mapping class group (Torelli SMCG):
   Torelli SMCG is trivial if  $N_{\w}>8$;  it is $\pi_0(\Diff^+(S^2,5))$ if $N_{\w}=0$ (by \cite{Sei08},\cite{Eva11}); it is $\pi_0(\Diff^+(S^2,4))$ in the remaining case. Further, we completely determine the rank of $\pi_1(Symp(\CC P^2  \# 5{\overline {\CC P^2}})$ for any given symplectic form. Our results can be uniformly presented regarding Dynkin diagrams of type $\aA$ and type $\DD$ Lie algebras. We also provide a solution to the smooth isotopy problem of rational 4-manifolds.
 \end{abstract}
\maketitle
\setcounter{tocdepth}{2}
\tableofcontents

\section{Introduction}\label{Intro}


The main theme of this paper is the symmetry of symplectic rational surfaces.  Based on Gromov and many other experts' works, it is now known that the topology of symplectomorphism groups exhibits various levels of similarity to the biholomorphisms for a K\"ahler manifold.  There are also even deeper relations from symplectomorphism groups to algebraic geometry, built on the study of moduli spaces of algebraic varieties and mirror symmetry, see \cite{DKK16},\cite{Ailsa14}.

In this paper, we focus on the classical feature of this similarity.  Consider a symplectic rational surface $X$ equipped with the monotone symplectic form, the homotopy type of $Symp(X, \w)$ are known to the work of Gromov, Lalonde-Pinsonnault, Seidel, and Evans.  In this class of examples, when $\chi(X)<8$, $Symp(X, \w)$ is homotopically equivalent to the biholomorphism group of their Fano cousin.  Surprisingly, when $\chi(X)\ge8$, $Symp(X, \w)$ exhibits a completely different nature, especially in its mapping class group $\pi_0(Symp(X, \w))$.  In \cite{Sei08}, Seidel observed the relation between $\pi_0Symp(X, \w)$ and the sphere braid group with five strands $Br_5(S^2)$ when $\chi(X)=8$ through monodromy of a universal bundle over configuration space of points.  Evans \cite{Eva11} eventually proved the two groups are isomorphic.

For the case when $\w$ is non-monotone, the homotopy type of $Symp(X,\w)$ is much more difficult to study.  Thanks to the works by Abreu \cite{Abr98}, Abreu-McDuff \cite{AM00}, Lalonde-Pinsonnault \cite{LP04}, as well as many other authors, much is known when $\chi(X)<8$.  In one of the recent works, Anjos-Pinsonnault \cite{AP13} computed the homotopy Lie algebra of $Symp(X,\w)$ when $\w$ is non-monotone.

However, the problem has been open for a long time when $\w$ is non-monotone and $\chi(X)\ge8$.  One of the main difficulty in studying the homotopy type of $Symp(X,\w)$ is the lack of understanding of the \textbf{symplectic mapping class group} $\pi_0Symp(X,\w)$ (SMCG for short).  $Symp(X,\w)$ has a subgroup $Symp_h(X,\w)$ that acts trivially on the homology of $X$ and its  mapping class group $\pi_0(Symp_h(X,\w))$ is called the \textbf{Torelli symplectic mapping class group} (Torelli SMCG for short) .  In short, we have the following short exact sequence
\begin{equation}\label{smc}
1 \to\pi_0(Symp_h(X,\w))\to \pi_0(Symp(X,\w)) \to   \Gamma(X,\w)\to 1.
\end{equation}
Since the homological action $\Gamma(X,\w)$ can be independently studied (see \cite{LW12}), the crux of the problem lies in the Torelli part of SMCG.

Torelli SMCG is also of many independent interests.  Donaldson raised the following question (cf.\cite{SS17}): is the Torelli SMCG group generated by squared Dehn twists in Lagrangian spheres?  A weaker version of this question is the following open problem: for a generic symplectic form on a rational surface, the Torelli SMCG is trivial.  Even for five blow-ups of $\CP^2$, this weaker conjecture is previously not known to be true or not.  In a slightly more general context, Lagrangian Dehn twists can be regarded as the monodromy of the coarse moduli of rational surfaces, and Donaldson's conjecture is asking for the triviality of the cokernel from $\pi_1(Conf_n(\CP^2))\to \pi_0Symp(X, \w)$.  Indeed, this is a natural question over any coarse moduli of a projective variety.  In a different direction,  Question 2.4 in \cite{Smith14} asks for this cokernel over the coarse moduli of a degree $d$ hypersurface in $\CP^N$.



As yet another motivation for studying Torelli SMCG as pointed out first in \cite{Coffey05} and later developed in \cite{LW12} and \cite{BLW12}, the understanding of $\pi_0(Symp(X, \w))$ gives insights to the problem of Lagrangian uniqueness.  This enabled one to re-prove Evans and Li-Wu's result on the uniqueness of homologous Lagrangian spheres when $\chi(X)<8$ using the result in \cite{LLW15}.  As a result of the lack of computations for Torelli SMCG, it was also unclear whether Lagrangian or symplectic $(-2)$-spheres are unique up to Hamiltonian isotopies for non-monotone rational surfaces.

In this paper, we compute the TSCM for non-monotone surfaces with $\chi(X)=8$, hence deduce a series of consequences that answer the questions of Lagrangian uniqueness above.   We also hope this would shed some light on general symplectic rational surfaces. Along the way, we give a new proof that $Symp_h(X) \subset \Diff_0(X)$ for any rational surface $X$ in Appendix \ref{t:MS}, solving Question 16 in the problem list of the book by McDuff-Salamon \cite{MS17}, which is of independent interest.  Note that this result was proved earlier in \cite{She10} using a completely different method.

Following \cite[Section 2]{LL16}, we recall in Section \ref{sec:lagrangian_systems_and_types_of_symplectic_forms} that, for any symplectic form $\omega$ on a rational surface with Euler number at most 11, the homology classes of Lagrangian $(-2)$-spheres form a root system $\Gamma(X,\w)$, called the \textbf{Lagrangian system}.  When $\chi(X)\le8$, $\Gamma_L(\w)$ is a sublattice of $\DD_5$, which has 32 possibilities (see Table \ref{5form}).  We call a sub-system \textbf{type $\aA$} if it is of type $\aA_1,$ $\aA_2,$ $\aA_3,$ $\aA_4,$ or their direct product, and \textbf{type $\DD$} if they are either $\DD_4$ or $\DD_5$.

\begin{defn}\label{d:TypesOfForm}
        We call a symplectic form $\w$ to be \textbf{type $\aA$} or \textbf{type $\DD$} if its corresponding Lagrangian system of of type $\aA$ or $\DD$, respectively.
\end{defn}

As is detailed in Section \ref{sec:lagrangian_systems_and_types_of_symplectic_forms},  in the reduced symplectic cone there are precisely two strata(which we also call open faces) of forms of type $\DD$ when $\chi(X)=8$, and all the rest of 30 possible strata of symplectic forms when $\chi(X)\le8$ are of type $\aA$.  Our main theorem concludes that the behavior of $\pi_0Symp(X,\w)$  is compatible with this combinatorial structure of the symplectic cone with explicit computations.


\begin{thm}[Main Theorem 1]\label{t:main}

Let $(X,\w)$ be a symplectic rational surface with $\chi(X)\le8$.

\begin{itemize}

\item When $\Gamma_L(\w)$ is of type  $\aA$, sequence \eqref{smc} reads
$$1 \to 1(\cong \pi_0Symp_h(X,\w))\to \pi_0(Symp(X,\w)) \to W(\Gamma_L(\w))\to 1,$$

\noindent where $W(\Gamma_L(\w))$ is the Weyl group of the root system  $\Gamma_L(\w)$. In other words, $\pi_0(Symp(X,\w))$ is isomorphic to $ W(\Gamma_L(\w))$;

\item When $\Gamma_L(\w)$ is of type  $\DD_n$, $n=4,5$,  sequence \eqref{smc} reads
 $$1 \to \pi_0(\Diff^+(S^2,n)) \to \pi_0(Symp(X,\w)) \to W(\Gamma_L(\w)) \to 1,$$
\noindent where $\pi_0(\Diff^+(S^2,n))$ is the mapping class group of n-punctured sphere.
\end{itemize}
\end{thm}

Note that this theorem was observed by McDuff \cite[Remark 1.11]{McD08}, where a sketch of deforming the Lagrangian Dehn twists to symplectic twists was given.  Our approach takes a slightly different form via ball-packings, see more details from the sketch of proof below.  From the point of view of braid groups, Theorem \ref{t:main} could be natural: one should think of the strands of the braid group as exceptional curves (which will be justified in the course of the proof).  As the class $\w$ becomes more generic through a path of deformation, some braidings disappear due to symplectic area reasons, and this leads to a \emph{strand-forgetting} phenomenon when a $\DD_5$-form $\w$ deforms to a $\DD_4$-form.  The more generic strata correspond to braid groups over $S^2$ with fewer than 4 strands, which are trivial.  Indeed, this phenomenon was suggested previously in \cite{McD08}.


Although our main goal is to understand the Torelli SMCG for $\chi(X)=8$, previously known cases for $\chi(X)<8$ also fits into our framework. This motivates the following rank equality.

 \begin{thm}[Main Theorem 2]\label{t:main2}
 Let $X$ be $\CC P^2\# 5\overline{\CC P^2}$ with any symplectic form $\omega$, then
 \begin{equation}\label{e:rankEquality}
      rank[\pi_1(Symp_h(X,\w))]={N_{\omega}}-5 + rank[\pi_0(Symp_h(X,\w))]
 \end{equation}
   Here $rank[\pi_0(Symp_h(X,\omega))]$ is the rank of the abelianization of $\pi_0(Symp_h(X,\omega))$ and $N_\w$ is the number of homology classes representable by symplectic $(-2)$-spheres.
\end{thm}

The above rank equality is first observed in \cite{LL16} and was proved for $\chi(X)\le7$.  We apply our computation of Torelli SMCG to extend it to $\chi(X)=8$, and we expect this equality to hold for all rational surfaces.

Finally, we combine the analysis of $\pi_1$ and $\pi_0$ of $Symp(\CC P^2  \# 5{\overline {\CC P^2}},\w)$ to obtain the following conclusion on $(-2)$-symplectic spheres:

\begin{cor}
Homologous $(-2)$-symplectic spheres in $\CC P^2  \# 5{\overline {\CC P^2}}$ are  symplectically  isotopic for any symplectic form.  For a type $\aA$-form $\w$, Lagrangian spheres in $(X,\w)$ are Hamiltonian isotopic to each other if they are homologous.
\end{cor}

\subsection*{The strategy}



Since the structure of the proof is somehow convoluted, we provide a roadmap for readers' convenience, as well as fix some notations here.

The general strategy follows what was described in \cite{LLW15}.  Choose an appropriate configuration of exceptional spheres $C$, as explored by Evans \cite{Eva}.  The following diagram of homotopy fibrations will play a fundamental role in our study.

\begin{equation} \label{summary}
\begin{CD}
Symp_c(U)@>^\sim>> Stab^1(C) @>>> Stab^0(C) @>>> Stab(C) @>>> Symp_h(X, \omega) \\
@. @. @VVV @VVV @VVV \\
@.@. \mG(C) @. Symp(C) @. \msC_0 \simeq \mJ_{open}
\end{CD}
\end{equation}

The terms in this homotopy sequence are defined as follows:

\begin{itemize}\setlength\itemsep{-3pt}
  \item $\msC_0$ is the space of  configurations which are symplectically isotopic to $C$; and $\mJ_{open}$ is the collection of almost complex structures which do not admit $J$-holomorphic spheres of $c_1\le0$;
  \item $Symp(C)$ is the symplectomorphism group of a fixed configuration $C$ which preserves each component of $C$;
  \item $Stab(C)$ is the subgroup of $Symp_h(X,\w)$ that preserve $C$ (or fix $C$ setwisely);
  \item $\mG(C)$ is the gauge group of the normal bundle of $C$;
  \item $Stab^0(C)$ is the subgroup of $Stab(C)$ that fix $C$ pointwisely;
  \item $Stab^1(C)$ is the subgroup of $Stab^0(C)$ which fix a neighborhood of $C$;
  \item $Symp_c(U)$ is the compactly supported symplectomorphism of the complement of $C$.
\end{itemize}


 This series of homotopy fibrations will be established in Proposition \ref{fib5}. 
 Most of then were established in Evans \cite{Eva11}.  Our focus is the right end of diagram \eqref{summary}:
 \begin{equation}\label{right}
Stab(C)\to Symp_h(X, \omega)\to \msC_0 \simeq \mJ_{open}.
\end{equation}

 The term $Symp(C)$, which is the product of the symplectomorphism group of each marked sphere component,  is homotopic to $ \Diff^+(S^2, 5)\times (S^1)^5$.
To deal with the Torelli SMCG, we consider the following portion of the homotopy exact sequence associated to \eqref{right}:
  \begin{equation}\label{fom}
  \pi_1(\msC_0) \overset{\phi}\to \pi_0(Stab(C)) \overset{\psi}\to \pi_0(Symp_h(X, \omega))\to 1.
  \end{equation}

Compared to the monotone case when $\msC_0$ is contractible (where the form is of type $\DD_5$), we fall short of computing the homotopy type of it directly: indeed, the topology of the open strata of almost complex structure can be very complicated even in much simpler manifolds, see \cite{AM00}.

We took a new approach here.  Starting from a class of \textit{standard $\RP^2$ packing symplectic forms} ($\RP^2$ forms for short, see Definition \ref{pacform}), we show that the map $\phi$ is indeed surjective when there is an $\RP^2$ packing in $X$.  This surjectivity is in turn related to another relative ball-packing problem and makes use of the \textit{ball-swapping symplectomorphism} constructed in \cite{Wu13}.  We then use a stability argument inspired by \cite{MCDacs}, paired with a Cremona equivalence computation, to relate a type $\aA$ form with a $\RP^2$ form.

The forms of type $\DD_4$ is more complicated.  We will construct a key commutative diagram \eqref{e:key} (compare \cite{Sei08}).  The punchline is to remove those strata of almost complex structures which allows more than one $(-2)$-sphere, or spheres with self-intersection no greater than $(-3)$ from the space of $\omega$-compatible almost complex structure.  This yields a $2$-connected space.  Such a space is not homeomorphic to $\msC_0$, but captures $\pi_i(\msC_0)$ for $i=0,1,2$, which suffices for the study of $\pi_0$ and $\pi_1$ of $Symp(X,\w)$.  An extensive study of diagram \eqref{e:key} enables one to compare the induced homotopy sequence in the lowest degrees with the strand-forgetting sequence
$$1\to \pi_1(S^2- \hbox{4 points}) \to \Diff^+(S^2, 5)\overset{f_1}\to \Diff^+(S^2, 4)\to 1,$$

which eventually deduces our main theorem for $\DD_4$ using the Hopfian property of braid groups.

\begin{rmk}\label{rem:}
  After the first draft of this manuscript was posted, Silvia Anjos informed us about her work with Sinan Eden (\cite{Anjos}, \cite{AE17} ), in which they independently obtain similar results in some toric cases for the 4-fold blow-up of ${\CC P^2}$, including the generic case and the case where $\lambda=1$ in the Table \ref{5form}.  Moreover, they have a result to show that the generators of $\pi_1(Ham(X,\w))$ also generate the homotopy Lie algebra of $Ham(X,\w)$, using similar ideas from \cite{AP13}.
\end{rmk}

{\bf Acknowledgements:} The authors are supported by  NSF Grants. The first author would like to thank Professor Daniel Juan Pineda for pointing out the reference \cite{GG13}.   We appreciate useful discussions with S{\'{\i}}lvia Anjos, Olguta Buse, Richard Hind,  Martin Pinsonnault, and Weiyi Zhang.

\section{Lagrangian systems, symplectic cone, and stability of \texorpdfstring{$Symp(X,\w)$}{Symp}}
\label{sec:lagrangian_systems_and_types_of_symplectic_forms}

The goal of this section is two-fold. First, we review some basic facts about Langrangian/symplectic sphere classes,  which will be repeatedly used in our argument.  The definitions and results in this section are taken mostly from \cite{LL16} without proofs, and interested readers are referred there for more details. Secondly, we prove a stability result of $Symp(X,\w)$ using the approach in \cite{MCDacs}, which will be useful in the proof of main results of this paper.

\subsection{Reduced forms and Lagrangian root system}
\label{sec:rootsystem}

We review the definition of reduced forms and Lagrangian root systems in this section, which provides a natural stratification for symplectic classes of rational surfaces.  Most of the proofs can be found in \cite{LL16} and we will not reproduce here.

Let $X$ be $ {\CC P^2}\# n\overline{\CC P^2}$ with a standard basis $H, E_1, E_2, \cdots, E_{n}$  of $H_2(X;\ZZ)$.  Given a symplectic form $\w$, its class is determined by the $\w$-area on each class $H,E_1, \cdots , E_n$, denoted as $\nu, c_1, \cdots, c_{n}$.  In this case, we will often use the notation $[\w]=(\nu|c_1,\cdots,c_n)$ in the rest of the paper.  In many cases, we normalize the form so that $[\w]=(1|c_1,\cdots,c_n)$.

\begin{dfn}
 $\w$ is called {\bf reduced} (with respect to the basis) if
\begin{equation}\label{reducedHE}
 \nu >  c_1\geq c_2 \geq \cdots \geq c_n>0
\quad \text{and} \quad \nu\geq c_i+c_j+c_k.
\end{equation}

\end{dfn}

We will also frequently refer to the following change of basis in $H^2(X,\ZZ)$.  Note that $X=S^2\times S^2 \# k\overline{\CC P^2}, k\geq1 $ is symplectomorphic to $ {\CC P^2}\# (k+1)\overline{\CC P^2}$.
When $X$ is regarded as a blow-up of $S^2\times S^2$, $H_2(X)$ can be endowed with a choice of basis $B, F, E'_1,\cdots, E'_k$, where $B$, $F$ are the classes of the $S^2$-factors and $E_i'$ are the exceptional classes; while when it is regarded as a blow-up of $\CP^2$, $H_2(X)$ has the basis $H, E_1,\cdots, E_k, E_{k+1}$, where $H$ is the line class, and $E_i$ are the exceptional divisors.  The two bases satisfy the following relations:
 \begin{align}\label{BH}
 B=H-E_2,\nonumber \\F=H-E_1, \nonumber \\E'_1=H-E_1-E_2,\\E'_i=E_{i+1},\forall i\geq 2\nonumber,
 \end{align}
The inverse transition will also be useful:
  \begin{align}\label{HB}
 H=B+F-E'_1,\nonumber \\E_1=B-E'_1, \nonumber \\E_2=F-E'_1,\\E_j=E'_{j-1},\forall j>2\nonumber .
 \end{align}

A more explicit form of base change for a class is given by the following

\begin{lma}\label{l:}
     Under the above base change formula,  $\nu H-c_1E_1 - c_2E_2 -\cdots -c_kE_k =\mu B + F -a_1E'_1 - a_2E'_2 -\cdots -a_{k-1}E'_{k-1}$ if and only if
\begin{equation}\label{ctoa}
\mu = (\nu- c_2)/(\nu -c_1), a_1 = (\nu -c_1 - c_2)/(\nu -c_1),
a_2 = c_3/(\nu - c_1), \cdots, a_{k-1} = c_{k}/(\nu - c_1).
\end{equation}

 \end{lma}

The significance of reduced classes lies in the following result \cite{GaoHZ, LL01, KK17}:

\begin{thm} \label{redtran}
For a rational surface $X= {\CC P^2}\# k\overline{\CC P^2}$,  every class with positive square in $H^2(X;\RR)$ is equivalent to a reduced class under the action of $\rm{Diff}^+(X)$. Further, any symplectic form on a rational surface is diffeomorphic to a reduced one.

If a symplectic form $\omega$ on $X$
is reduced, then its canonical class is  $ K_{\w}=-3H +\sum^{k}_{i=1} E_i.$

When  $3\leq k\leq 8$, any reduced class is represented by a symplectic form.
When $k\leq 2$, any reduced class with $\nu>c_1+c_2$ is represented by a symplectic form.

  \end{thm}

\subsubsection{The normalized reduced cone \texorpdfstring{$P(X_k)$}{P(Xk)}  for \texorpdfstring{$3\leq k \leq 8$}{3<k<8}}

Recall from \cite{LL16} that
\begin{dfn}
Let $X_k= {\CC P^2}\# k\overline{\CC P^2}$.  Its normalized reduced symplectic cone $P_k=P(X_k)$  is defined as the space of reduced symplectic classes having area 1 on $H$.  We represent such a class
 by $(1|c_1, \cdots, c_k)$,     or $(c_1, \cdots, c_k)\in \RR^k$
 \end{dfn}

 When $ k\leq8,$, we call $M_k= (1|\frac13, \cdots, \frac13)$ or $(\frac13, \cdots, \frac13)\in P_k,$   the \emph{(normalized) monotone class}.
 When $3\leq k\leq8,$ $P_k$ has an explicit description.  Consider the following  $k$ (spherical) classes of square $-2$:
$$l_1= H-E_1-E_2-E_3,  \quad l_2=E_1-E_2,\quad  ... \quad, \quad l_k=E_{k-1}-E_k.$$

\begin{prp}\label{nrsc}
For $X_k=\CC P^2 \# k{\overline {\CC P^2}}, 3\leq k\leq8,$  the  normalized reduced symplectic cone $P_k$ is a convex  polyhedron in $\mathbb R^k$ with $k+1$ vertices: one of the vertices is $M_k$, and $k$ other vertices in the hyperplane $c_k=0$ located at
$$G_1=(0,...,0),
 G_2=(1, 0,..., 0),
G_3=(\frac{1}{2}, \frac{1}{2}, 0, ...,0),$$
$$G_4=( \frac{1}{3}, \frac{1}{3}, \frac{1}{3}, 0,...,0),
  ... ,
G_k=(\frac{1}{3},..., \frac{1}{3}, 0).
$$
The edges $M_kG_i$ are characterized as pairing trivially with  $l_j$ for any $j\ne i$ and positively with $l_i$.
 Consequently, the reduced symplectic classes are characterized as the symplectic classes which pair positively on each $E_i$ and non-negatively on each $l_i$.
\end{prp}

Further, we highlight the combinatorial structure of the reduced cone.

\begin{dfn} \label{openface}
  A $p-$dimensional {\bf open face} of $P_k$ is defined as the {\bf interior} of the convex hull of $M_k$ together with $p\leq k$ points in the set $\{G_i\}.$  $P_k$ has
$2^k$ open faces in total:
a unique zero dimensional open face $M_k$; $k$ one dimensional open faces,
and generally, $\binom{k}{p}$ open faces of dimension $p$.

Our convention is to denote an open face with  vertices $v_1, v_2, \cdots, v_l$ simply by $v_1v_2\cdots v_l$.
\end{dfn}

\subsubsection{Lagrangian root systems for \texorpdfstring{$3\leq k\leq 8$}{3<k<8}}\label{s:cone}

We  slightly reformulate a result from \cite{Man86} (see also \cite{LZ14}).  For    $X_k$  with $3 \leq k\leq 8$, define the set
\begin{equation}\label{smooth root system}
R(X_k)=R_k :=  \{ A \in H_2(X_k,\ZZ)   \mid \left<A,K_k\right> = 0, \quad \left<A,A\right> = -2 \},
\end{equation}
where $K_k=-(3H-E_1-...-E_k)$.
It is straightforward to check $R_k$ is a root system described in the table below,

\[ \begin{array}{c|cccccc}
k & 3 & 4 & 5 & 6 & 7 & 8 \\
\hline
R(X_k)&\aA_1\times \aA_2 &\aA_4 & \DD_5 & \EE_6 & \EE_7 & \EE_8\\
|R(X_k)| & 8 & 20 & 40 & 72 & 126 & 240 \\
\end{array} \]
The classes $\{l_i\}$ provide a canonical choice of \textbf{simple roots} of $R_k$, which describe the vertices of the Dynkin diagram.
One may correspond these simple roots $l_i$  to the edges $M_kG_i$ of $P_k$, which represents those reduced symplectic classes which pairs positively with $l_i$ and trivially with all other $l_j$.

Given a symplectic form $\w$ on $X_k$, one may then define the {\bf Lagrangian root system} $\Gamma_L(\omega):=\{A\in R_k: \w(A)=0\}$.  From Theorem 1.4 of \cite{LW12}, $\Gamma_L(\w)$ are those classes representable by embedded Lagrangian spheres.  The following proposition about $\Gamma_L(\w)$ is proved in \cite{LL16}.

\begin{prp}[\cite{LL16} Proposition 2.24] \label{MLS}
Given a reduced symplectic form $\omega$ on $X_k$.

\begin{enumerate}
  \item If $\omega_{mon}$ is a monotone symplectic form on $X_k$ with $3\leq k\leq 8$, then $\Gamma_L(\omega_{mon})=R_k$.
  \item $\Gamma_L(X_k, \omega)$ is a sub-root system of  $R_k$, equipped with  a canonical choice of simple roots consisting of those $l_i$ in $\Gamma_L(X_k, \omega)$.
  \item There is a canonical choice of positive roots characterized positive pairing with $[\omega]$, given by the non-negative linear combinations of the simple roots $l_i\in \Gamma_L(X_k, \omega)$.
\end{enumerate}

\end{prp}

Let  $N_{\w}$ be the number of $\omega$-symplectic $(-2)$-sphere classes. Note that $N_{\w}$ and $\Gamma_L(\w)$ are both invariant in any given open face of the reduced cone (Definition \ref{openface}).  Let  $N_L$ be the number of  $\omega$-Lagrangian sphere classes up to sign.
Again from \cite{LW12} Theorem 1.4, any positive root defined above can be either represented by a smooth $\w$-symplectic $(-2)$-sphere or a $\w$-Lagrangian sphere.  Therefore, we have
\begin{equation}\label{SL}
    N_{\w}+N_L= |R^+(X_k)|=\frac{1}{2} |R_k|.
\end{equation}

Using the correspondence between $l_i$ and the edge $MG_i$, sometimes we label the faces of $P$ by these roots.  The general case is more complicated than we would like to reproduce here, and we will give a very explicit description in Section \ref{sec:SetupTypeA}.  The readers are referred to \cite{LL16} for the general case.

\subsection{Negative square classes and the stratification of \texorpdfstring{ $\mJ_{\w}$}{Jw}}

Let $\mJ_\w(X)$ be the space of $\w$-tamed almost complex structures on a manifold $X$, and we omit the reference to $X$ in the notation when no confusion is possible. In this section, we recall several results about $\mJ_\w$ of a rational 4-manifold. Note that all the statement holds true if we replace  $\mJ_\w$ by  $\mJ^c_\w$, the space of  $\w$-compatible almost complex structures, which will be useful in Sections \ref{s:freeact} and \ref{s:surj}.

In \cite{LL16}, we decompose $\mJ_{\w}(X)$ when $X$ is a rational 4-manifold with Euler number no larger than 12 into prime submanifolds labeled by negative square spherical classes. Let $\mathcal S_{\omega}$ denote the set of homology classes  of
  embedded $\omega$-symplectic spheres.  For any integer $k$, let $$\mathcal S_{\omega}^{\geq k},  \quad \mathcal S_{\omega}^{>k}, \quad   \mathcal S_{\omega}^{k}, \quad  \mathcal S_{\omega}^{\leq k},\quad  \mathcal S_{\omega}^{< k}$$
 be the subsets of $\mathcal S_{\omega}$ consisting of classes with square $\geq k, >k, =k, \leq k, <k$ respectively. Recall the following very useful Lemma \cite[Proposition 2.14]{LL16}.

 \begin{lma}\label{prime}
   Let $X$ be  a rational 4-manifold such that $\chi(X) \leq 12$.
 Given a  finite subset   $\mC\subset \mathcal S_{\omega}^{<0}$,
  $$\mC=\{A_1, \cdots, A_i,\cdots ,A_n |  A_i\cdot A_j \geq 0  \hbox{ if $ i\neq j$}\},$$
   we have the following {\bf prime submanifolds}
   $$\mJ_{\mC}:=\{ J\in \mJ_{\omega}|  A\in \mathcal S_{\w} \hbox{ admits a smooth embedded $J-$hol representative {\bf iff} } A\in \mC\},$$
which is a submanifold of codimension
   $cod_{\mC}= \sum_{A_i\in \mC } cod_{A_i}$ in $\mJ_\w$. Also denote $\mX_{2n}=\cup_{cod(\mC)\geq2n} \mJ_{\mC}.$
 \end{lma}

\begin{lma}\label{primeaction}
There is an action of  $Symp_h$ on each prime submanifolds in Lemma \ref{prime} \end{lma}
 \begin{proof}
 This follows from the fact that the action of $Symp_h$ on $\mJ_{\w}$ preserves the homology class of a $J$-holomorphic curve.

 \end{proof}

Note that we have the disjoint decomposition: $\mJ_{\omega} =\amalg_{\mC} J_{\mC},$ which is indeed a stratification at certain level, as follows:

\begin{thm}\label{rational}
For a symplectic rational 4 manifold with Euler number $\chi(X)\leq 8$ and  any symplectic form,
$\mX_4=\cup_{cod(\mC)\geq4} \mJ_{\mC}$ and $\mX_2=\cup_{cod(\mC)\geq2} \mJ_{\mC}$ are closed subsets  in $\mX_0=\mJ_{\w}$. Consequently,\\
(i).  $\mX_0 -\mX_4$ is a manifold.\\
(ii).  $\mX_2 - \mX_4$  is a closed codim-2 submanifold in  $\mX_0 - \mX_4$.

\end{thm}

This allows us to apply the following relative version of  Alexander-Pontrjagin duality in \cite{Eells61}:

\begin{lma}[Theorem 3.13 in \cite{LL16}] \label{relalex}
Let $\mX$ be a Hausdorff space,  $ \mZ \subset \mY $  a closed subset of
$\mX$ such that $\mX-\mZ, \mY-\mZ$ are paracompact manifolds locally modeled by topological linear spaces.  Suppose $\mY-\mZ$ is a closed co-oriented submanifold of $\mX-\mZ$ of  codimension $p$, then  we have an isomorphism of cohomology $H^i(\mX-\mZ,\mX-\mY; G) \cong H^{i-p}(\mY-\mZ; G)$ for any abelian group $G$.
\end{lma}

By taking $\mX=\mX_0$, $\mY=\mX_2$ and $\mZ=\mX_4$ in Lemma \ref{relalex}, we have the following conclusion on $\mJ_{open}:=\mX_0 -\mX_2$:

\begin{lma}[Corollary 3.14 in \cite{LL16}] \label{h1open}
  For  a symplectic rational surface $(X,\w)$ with  $\chi(X)\leq 8$ and any abelian group $G$,
  $H^1(\mJ_{open}; G)= \oplus_{A_i \in \mathcal S_{\omega}^{-2}} H^0(\mJ_{A_i})$.

If we further assume that $\chi(X)\leq 7$, then for  each  $ {A_i \in \mathcal S_{\omega}^{-2}}  $,  $\mJ_{A_i}$ is path connected and hence $H^1(\mJ_{open};G)=G^{N_{w}}$, where $N_{\omega}$ is the cardinality of $\mathcal S_{\omega}^{-2}$.
It follows from the universal coefficient theorem that   $H_1(\mJ_{open}; \ZZ)=\ZZ^{N_{\w}}$.
\end{lma}

\begin{rmk}\label{tc}
  Note that if we consider $\mJ^c_\w$, the space of  $\w$-compatible almost complex structures, we can define $\mX^c_{2n}$'s similarly， and  Lemma \ref{prime}, \ref{rational},and \ref{primeaction} still hold true.
\end{rmk}

Next, we recall some technical lemmata from \cite{LL16,BLW12} about curves in rational surfaces for later use:

\begin{lma}\label{tran}
 For a rational 4-manifold $X$ with any symplectic form $\w$,
  the group $Symp_h(X,\w)$ acts
 transitively on the space of homologous $(-2)$-symplectic spheres.
 \end{lma}

\begin{prp}[\cite{LL16}, Proposition 3.4]\label{sphere}
Let $X=S^2\times S^2 \# n\overline{\CC P^2}, n \leq4$ with a reduced symplectic form.
Suppose a  class $A=pB+qF-\sum r_iE_i \in H_2(X;\mathbb{Z})$
has a simple $J$-holomorphic spherical representative for some $J\in\mJ_\omega$. Then $p\in\{0,1\}$.

The spherical classes with negative squares has one of the following forms:

\begin{itemize}
  \item $B-kF-\sum r_i E_i, k\geq-1, r_i\in\{0,1\};$
  \item $F-\sum r_i E_i, r_i\in\{0,1\}$;
  \item $\mE=\{E_j-\sum r_i E_i, j<i, r_i\in\{0,1\}.$
\end{itemize}

  Under the base change \eqref{BH},  the three type of classes above can be written as

  \begin{itemize}
    \item $(k+1)E_1-kH-\sum r_i E_i, k\geq-2, r_i\in\{0,1\}$
    \item $E_i -\sum_{j>i} r_j E_j, i\geq2, r_j\in\{0,1\} $
  \end{itemize}

\end{prp}

\begin{prp}[\cite{LL16},Proposition 3.6]\label{decp}
Let  $X=(S^2\times S^2 \# k\overline{\CC P^2},\w), k\leq4$ be a symplectic rational surface.
Let $A$ be a K-nef class which has an embedded representative for some $J$.
Then for any simple $J'-$holomorphic representative of $A$ for some  $J'$, there is no component whose class has a positive square.  Moreover, if the symplectic form is reduced,\\
$\bullet$ any square zero class in the decomposition is of the form $B$ or $kF, k\in \ZZ^+$,\\
$\bullet$ any negative square class is a class of an embedded symplectic sphere as listed in Proposition \ref{sphere}.

\end{prp}

The following important result, first due to Pinsonnault, will be the key of our analysis on curve configurations.

\begin{thm}[\cite{Pin08} Lemma 1.2]\label{t:Pin}
  For a symplectic 4-manifold not diffeomorphic to ${\CC P^2}\# \overline{\CC P^2} $, any exceptional class with minimal symplectic area has an embedded $J$-holomorphic representative for any $J \in \mJ_{\w}$.
\end{thm}

 We'll use a slightly different version of it for rational 4-manifolds.

\begin{lma} [\cite{LL16} Lemma 2.19]\label{minemb}
Let $X$ be ${\CC P^2}\# n\overline{\CC P^2}$ with a reduced symplectic form $\w$, and $\w$ is represented using a vector
$(1|c_1, c_2 ,\cdots , c_n)$.
Then  $E_n$
has the smallest area among all exceptional sphere classes in $X$, and hence have an embedded $J$-holomorphic representative for any $J\in \mJ_{\w}$.
\end{lma}

\subsection{An inflation Theorem}

We are now ready for a stability result of the symplectomorphism group, and here we only state and prove a weaker result on $\pi_0$ and $\pi_1$ of $Symp(X_5,
w)$ that is sufficient for our purpose. A general statement for any $\pi_i$ of any $Symp(X,\w)$ of $X$ being a rational 4-manifold with $\chi(X)\leq 12$ is
proved in \cite{ALLP}.

Recall the definition of  $J$-tame cone $$\mK_{J}^{t}:=\{[\omega]\in H^2(M;\RR)|\omega \hbox{ tames } J\},$$
and $J$-compatible cone (also called the almost K\"ahler cone)
$$
\mK_{J}^{c}:=\{[\omega]\in H^2(M;\RR)|\hbox{$\omega$  is compatible with
$J$}\}.
$$

$\mK_J^{t}$and $\mK_J^{c}$ are both convex cones
  in the positive cone $\mathcal P=\{c\in H^2(M;\RR)|c\cdot c >0\}.$

Note that we have the tamed Nakai-Moishezon theorem for rational surfaces when Euler number is small:

\begin{thm}[Theorem 1.6 in \cite{Zha17}]\label{ccinf}
  Suppose $M=S^2\times S^2$ or $\CC P^2  \# k{\overline {\CC P^2}}, k\leq 9$, and let $C_J^{>0}:=\{c\in H^2(M;\RR)| [\Sigma]=c \text{ for some some $J$-holomorphic subvariety } \Sigma \}$ be the curve cone of $J$.
For an almost K\"ahler $J$ on $M$,
$C_{J}^{\vee,>0}(M)=K_{J}^{c}(M),$ that is, the almost K\"ahler cone is the dual cone of the curve cone.

\end{thm}

Although the result is stated for an almost Kahler $J$ and the almost K\"ahler cone, Zhang's argument works for a tamed $J$.
 An important ingredient for Theorem \ref{ccinf} is the tamed $J-$inflation by Lalonde, McDuff \cite{LM96, MCDacs} and Buse \cite{Buse11}.  Note that in Zhang's proof of Theorem \ref{ccinf}, he realizes all the extremal rays of the symplectic cone as a sum of embedded curves in 0-square and negative square homology classes.   This is to say, only Lemma 3.1 in \cite{MCDacs} and Theorem 1.1 in \cite{Buse11} are used, and no inflation along positive self-intersection curves is needed.

  Let's recall the framework in \cite{MCDacs}: Let $\mT_w$ be the space of symplectic forms in the class $w$, and $\mA_w:=\cup_{\w\in\mT_w}\mJ_\w$.  If $\w, \w'\in \mT_w$, then one can show that they are isotopic, and hence the   symplectomorphism groups $Symp(M, \w)$ and $Symp(M, \w')$ are homeomorphic.
We also have the fibration $$ Symp(X,\w)\cap  \Diff_0(X)\to \Diff_0(X) \to \mT_{w},$$ where $ \Diff_0(X) $ is the identity component of the diffeomorphism group.

Let $P_{w}$ be the space of pairs
$$P_{w} = \{(\w, J) | \mT_{w} \times \mA_{w}: J\in\mJ_\w\},$$
Consider  the projection  $P_{w} \to  \mA_{w}$. It is a homotopy fibration, of which the fiber at $J$ is the space of $J-$tame symplectic form.  This projection inducdes a homotopy equivalence since the fiber is convex.
The projection
$P_{w} \to  \mT_{w}$ is also a homotopy equivalence: its fiber over $\w$ is the contractible set of $\w$-tame almost complex structures.
Hence $T_w$ and $\mA_w$ are homotopy equivalent.

 Via the homotopy equivalence, we have the following fibration, well defined up to homotopy.
\begin{equation} \label{homotopy fibration}
 Symp(X,\w)\cap  \Diff_0(X)\to \Diff_0(X) \to \mA_{w}.
 \end{equation}

 Let $\mathcal S_{w}$ denote the set of homology classes  that are represented by an
  embedded $\omega$-symplectic sphere for some $\w\in \mT_w$.  Note that  $\mS_{\w_1}=\mS_{\w_2}$ for any $\w_1,\w_2\in \mT_w$.  For the applicability of Proposition \ref{nonbalstab} and \ref{-2stable}, we have the following Theorem, which plays a key role in our study and is of independent interests.  This is the main result (Corollary 2.i) in \cite{She10}.  We give a different approach in Appendix A, which contains a short proof for $X_5$ and a general proof for $X_n$.

  \begin{thm}\label{smoothisotopy} For any symplectic form $\w$ on $X_n$
$Symp_h(X_n, \w) \subset \Diff_0(X_n)$ .
\end{thm}

The following proposition is the main result of this section.

\begin{prp}\label{-2stable}
    For $i=1,2$, let $\w_i\in\mT_{w_i}$ are two symplectic forms on $X_5$. If $\mS_{w_2}\subset \mS_{w_1}$ and  $\mS^{-2}_{\w_1}=\mS^{-2}_{\w_2}$, then $\pi_i Symp(X_5,\w_1)= \pi_i Symp(X_5,\w_2)$ for $i=0,1$.
\end{prp}

\begin{proof}

Firstly, Theorem \ref{ccinf} implies $\mA_{w_2}\subset \mA_{w_1}$. To see this, take any $J\in \mA_{w_2}$, then $J$ is tamed by some $\w\in \mT_{w_2}$, and the {\bf only} $J$-holomorphic curves are in the classes of $\mS_{\w_2}$. Since  $\mS_{\w_2}\subset\mS_{\w_1},$ we know  $[\w_1]$ pairs positively with every class in $\mS_{\w_2}$, and hence by Theorem \ref{ccinf},  $[\w_1]=w_1$ is in the tame cone of $J$, meaning that $J$ tames some symplectic form in the class  $[\w_1]$. Then we have $J\in \mA_{w_1}.$

Therefore, there is an induced map $Symp(X,\w_1) \to  Symp(X,\w_2)$, which is well defined up to homotopy and makes the following diagram on \eqref{homotopy fibration} for $\w_1$ and $\w_2$ commute up to homotopy.

$$
\begin{array}{ccccccc}
& & Symp_h(X,\w_1) &\to & \Diff_0(M) & \to & \mA_{\w_1}\\
& & \downarrow & & \;\downarrow = & & \downarrow\\
  & & Symp_h(X,\w_2) &\to & \Diff_0(M) & \to & \mA_{\w_2}.\\
& & & & & &
\end{array}
$$

Here we replace the fiber of sequence \ref{homotopy fibration} by $Symp_h(X,\w)$ because of Theorem \ref{smoothisotopy}.
  The complement of  $\mA_{w_2}\subset \mA_{w_1}$ has codimension $4$, since  $S^{-2}_{\w_1}=S^{-2}_{\w_2}$ they have the same prime submanifolds of codimension $0$ and $2$.   Then the inclusion induce an isomorphism  $\pi_i(\mA_{\w_1})\to \pi_i(\mA_{\w_2})$  for $i=0,1,2.$
Therefore,
from the homotopy commuting diagram and  the associated diagram of long exact homotopy sequences of homotopy groups,
 the induced homomorphisms  $\pi_i (Symp_h(X,\w_1))\to \pi_i(Symp_h(X,\w_2)))$ are  isomorphisms for $i=0,1.$   Notice that by the smooth isotopy theorem \ref{smoothisotopy}, the fibers of the sequences are $Symp (X_5, \w_k)$.
\end{proof}

\begin{prp}\label{nonbalstab}
  Given any open line segment $L$ starting from the vertex $A$ of the reduced cone and two symplectic forms $\w_i$, $i=0,1$, where $[\w_i]\in L$, we have $\pi_j(Symp(X,\w_0))\cong\pi_j(Symp(X,\w_1))$ for $j=0,1$.

  Equivalently, $\pi_0, \pi_1$ of $Symp(X,\w)$ is invariant under the following type of deformation of symplectic form: 

  $ [\w_1]= (1|c_1,\cdots, c_5)$ and $[\w_t]=(1| (c_1-1)t+1 , t c_2, tc_3, tc_4, tc_5 )$, for $0< t<\frac{1}{c_2+1-c_1}$.
\end{prp}

\begin{proof}
  Firstly, note that the cohomology class of $\w$'s are points in the polyhedron cone lying in $\RR^5$, by Proposition \ref{nrsc}. The point $A:=[\w_0]$ is a limiting point on the cone, with coordinate $ A= (1,0,0,0,0).$ Then the line starting from $A$ passing through $[\w_1]:=(c_1,\cdots, c_5)$ has parametric equation $ L(t)= ((c_1-1)t+1 , t c_2, tc_3, tc_4, tc_5 )$. Note that this line connecting $A=[\w_0]$ and $ [\w_1]$ will stay in the reduced cone if $0< t<\frac{1}{c_2+1-c_1}.$
  From Proposition \ref{nrsc}, those rays passing through point $A$ will intersect the open face of the reduced cone sitting opposite to point $A$, and the open face is defined by $c_1=c_2$.  Before the line intersect the open face defined by $c_1=c_2,$ the line stay in the interior of the reduced cone. And this means we always have $c_1>c_2,$  which means $(c_1-1)t+1 > t c_2$ and is equivalent to $t<\frac{1}{c_2+1-c_1}.$

We shall check that $\w_1, \w_2$  satisfy the assumptions, i.e. $S_{\w_2}\subset S_{\w_1}$ and  $S^{-2}_{\w_1}=S^{-2}_{\w_2}$.

  By the classification Lemma \ref{sphere},  the coefficients of $E_2, \cdots, E_5$ are negative and $E_1$ non-negative for the class of a negative sphere, except for $2H-E_1-\cdots E_5$ or $H-E_1-\sum_j E_j, 2\leq j \leq i$. First it's straightforward to check  the positivity (or negativity) of the area of $2H-E_1-\cdots E_5$ and $H-E_1-\sum_j E_j$ remains unchanged when $[\w_t]$ moves inside the interval $0\leq t<\frac{1}{c_2+1-c_1}$. For other classses, note that  as $t$ increases, this deformation increases the area of $E_2, \cdots E_5$, and decreases $E_1$.  Therefore, from a case-by-case checking in Lemma \ref{sphere}, $\w_1(A)>\w_2(A)$ for any $A\in S^{<0}$. Hence we always have $S_{\w_2}\subset S_{\w_1}$.

  Also, we can directly check $S^{-2}_{\w_1}=S^{-2}_{\w_2}$ by the classification. In the case of $X_5$, all possible symplectic $-2$ classes are $E_i-E_j$, $H-E_p-E_q-E_r$.  Therefore, the statement is a corollary of Proposition \ref{-2stable}.

\end{proof}


\section{Connectedness of \texorpdfstring{ $Symp_h(X,\w)$}{Symp} for type \texorpdfstring{$\aA$}{A} forms}\label{5}

\subsection{Pure braid group on spheres}

In this section, we recall some standard facts regarding pure braid groups on spheres and disks.  For more details, the readers may refer to \cite{Bir69} and \cite{KT08}, etc.

Recall the \textit{braid group of $n$ strands on a sphere} is $\pi_1(Conf(S^2,n))$, while the pure braids are those in $\pi_1(Conf^{ord}(S^2,n))$.  We have the following basic isomorphisms.

\begin{lma}
  $$ \pi_0(\Diff^{+}(S^2,5)) \cong  PB_5(S^2)/ \langle\tau\rangle  \cong PB_{2}(S^2-\{x_1,x_2,x_3\}),$$
     where $PB_5(S^2)$  and $PB_{2}(S^2-\{x_1,x_2,x_3\})$ are
the pure braid groups of 5 strings on $S^2$, and  2 strings on $S^2-\{x_1,x_2,x_3\}$ respectively.
$\langle\tau\rangle =  \ZZ_2$ is the center of the pure braid group $PBr_5(S^2)$ generated by the  full twist $\tau$ of order 2.

It follows that $Ab( \pi_0(\rm{Diff}(5, S^2)))=\ZZ^5$.
\end{lma}

\begin{proof}
Consider the following homotopy fibration

  $$\Diff^{+}(S^2, 5)\to \Diff^+(S^2)\to Conf^{ord}(S^2,5).$$

We have the isomorphism $\pi_0(\Diff^{+}(S^2, 5))\cong\pi_1(Conf^{ord}(S^2,5))/im(\pi_1[\Diff^+(S^2)] \longrightarrow \pi_1[Conf^{ord}(S^2,5)])$ from the associated homotopy exact sequence.  Note that $\pi_1[Conf^{ord}(S^2,5)]=PBr_5(S^2)$ by definition, and
$im(\pi_1[\Diff^+(S^2)] \longrightarrow \pi_1[Conf^{ord}(S^2,5)])\cong\ZZ_2$ is the full twist $\tau$ of order 2, given by rotation of $S^2$ .  This gives the first isomorphism.

The second isomorphism follows from  the direct sum decomposition (cf. the proof of Theorem 5 in  \cite{GG13}),
$$PB_n(S^2)\simeq PB_{n-3}(S^2-\{x_1,x_2,x_3\})\oplus  \langle\tau\rangle. $$

Now we have  $Ab(\pi_0(\Diff^+(S^2,5)))= \ZZ^5$ since $Ab(PB_{2}(S^2-\{x_1,x_2,x_3\}))=\ZZ^{5}$ (\cite{GG13} Theorem 5).
\end{proof}

We also have an explicit description of some generating sets of the braid group $Br_n(S^2)$ and pure braid group $PB_n(S^2)$ on the sphere following \cite{KT08} section 1.2 and 1.3.

\begin{figure}[ht]
  \centering
  \includegraphics[scale=0.6, trim=0 70 30 0]{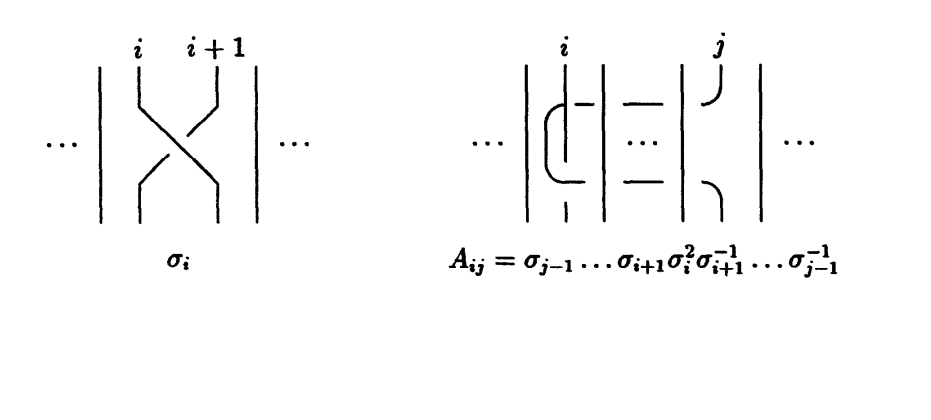}
  \caption{The Artin generator $\sigma_i$ and the standard  generator $A_{i,j}$}
  \label{Aij}
\end{figure}

Take the Artin generators $\{\sigma_1, \cdots, \sigma_{n-1}\}\subset Br_n(S^2)$, where $\sigma_i$ switches the ith point with (i+1)th point, then the standard generators of the pure braid group is given by $\{A_{ij}\}_{0\le i,j\le n}$:

\begin{equation}
A_{ij}=\sigma_{j-1}\cdots \sigma_{i+1} \sigma^2_{i} \sigma_{i+1}^{-1} \cdots \sigma_{j-1}^{-1}
\end{equation}

One can think $A_{ij}$ as the twist of the  point $i$ with the  point $j$, which geometrically
 (see Figure \ref{Aij}) can be viewed as moving $i$  around $j$ through a loop separating  $j$  from all other points.

We next apply Theorem 5(e) in \cite{GG13} and Proposition 7 in \cite{GG05}  to obtain a minimal generating set of $PB_5(S^2)/\ZZ_2$ suitable for our applications.  First recall

\begin{lma}[\cite{GG13}, Theorem 5(e)]\label{Fbraid}
\leavevmode
\begin{itemize}
  \item For  $ PB_{n-3}(S^2-\{x_1,x_2,x_3\})\simeq PB_n(S^2)/\ZZ_2,$ the set $\{ A_{ij}, j\geq4, 2\leq i<j\}$ is a
generating set.
  \item For any given $j$, one has the following relation ensuring that we can further remove the generators $A_{1j}$.
\begin{equation}\label{e:re}
     (\Pi_{i=1}^{j-1} A_{ij})(\Pi_{k=j}^{n+1} A_{jk})=1,
\end{equation}

\end{itemize}
\end{lma}

For the case of $n=5$, we have

\begin{lma}\label{Pbraid}
In $\pi_0(\Diff^{+}(S^2, 5))=PB_5(S^2)/\ZZ_2$, the following two sets of elements are both generating:

  \begin{enumerate}
  \item  $\{ A_{12}, A_{13},A_{14},A_{23},A_{24} \} $,
  \item  $\{A_{13}, A_{14},A_{15},A_{23},A_{24},A_{25}\}.$
  \end{enumerate}

\end{lma}

\begin{proof}

Case (1) is simply a permutation of the indices from $ \{A_{2 4}, A_{25}, A_{3 4}, A_{3 5},  A_{4 5}\}$ given in Lemma \ref{Fbraid}.
For case (2), we may re-index the generators to get $\{A_{14}, A_{24},A_{34},A_{15},A_{25},A_{35}\}$ by noting $A_{ij}=A_{ji}$.  By Lemma \ref{Fbraid}, the surface relation reads $$ (\Pi_{i=1}^{j-1} A_{ij})(\Pi_{k=j+1}^{5} A_{jk})=1.$$
Let $j=4$, we have $A_{14}A_{24}A_{34}A_{45}=1$. This means the above set generates the missing $A_{45}$ in Lemma \ref{Fbraid} hence a generating set.

\end{proof}

The last fact we'll need is the Hopfian property of the braid groups.

\begin{lma}\label{Hopfian}
The pure and full braid groups on disks or spheres are Hopfian, i.e. every self-epimorphism is an isomorphism.
\end{lma}

\begin{proof}
The disk case and sphere case can be dealt with in the same way although we only need the sphere case:

\begin{itemize}
 \item On disks:  Lawrence, Krammer \cite{LK00} and  Bigelow \cite{Big00}  showed that (full) braid groups on disks are linear.  By the result of Mal\'cev in \cite{Mal40}, finitely generated linear groups are residually finite; and finitely generated residually finite groups are Hopfian.  Note that the residual finiteness property is subgroup closed.  Therefore, the pure braid group on disks, as the subgroup of the full braid group, is residually finite. The pure braid group is also finitely generated, hence is Hopfian.
\item  On spheres: V. Bardakov \cite{Bardakov} shows that, the sphere full braid groups and the mapping class groups of the $n$-punctured sphere $MCG(S^2,n)$ are linear. The rest of the argument is the same as above. In particular,  $PB_4(S^2)/\ZZ_2=\pi_0(\Diff(S^2,4))$ is $MCG(S^2,4)$ and in the meanwhile $PB_4(S^2)/\ZZ_2$ is finitely generated, hence $PB_4(S^2)/\ZZ_2$ is Hopfian.
\end{itemize}
\end{proof}

Note that a group is Hopfian if and only if it is not isomorphic to any of its proper quotients. Later, we will make use of the following lemma:

\begin{lma}\label{hopfcomp}

Let $G$ be a sphere braid group, and $H$ is an arbitrary group.  If there are two surjective group homomorphisms     $ p: H \twoheadrightarrow G$ and $q: G \twoheadrightarrow H,$ then $G$ and $H$ are isomorphic.

\end{lma}

\begin{proof}
 Consider the composition $ q\circ p: G \overset{p}\twoheadrightarrow H \overset{q}\twoheadrightarrow G.$ It is a self-epimorphism of $G$ hence has to be an isomorphism by the Hopfian property. Therefore, the map $p: G \twoheadrightarrow H$ has to be injective.
\end{proof}

\subsection{Basic setup}\label{sec:SetupTypeA}

In this subsection, we quickly recap the symplectic cone and Lagrangian root system for $X_5$ in Table \ref{5form} and Figure \ref{5Dmon}. Then we provide a detailed explanation of the diagram \eqref{summary}.

\subsubsection{Reduced symplectic cone}

We make explicit the discussion in Section \ref{sec:rootsystem} for $k=5$.

 \begin{table}[ht]
\begin{center}
 \begin{tabular}{||c c c c ||}
 \hline\hline
$k$-face& $\Gamma_L$& $N_{\w}$  & $\omega:=(1|c_1,c_2,c_3 ,c_4,c_5)$  \\ [0.5ex]
 \hline\hline
 Point M &$\DD_5$ & 0&  monotone\\
 \hline
 MO &$\aA_4$&10 &  $\lambda<1;  c_1=c_2=c_3 =c_4=c_5 $\\
 \hline
 MA &$\DD_4$& 8&   $\lambda=1;c_1>c_2=c_3 =c_4=c_5   $\\
 \hline
 MB &$\aA_1\times \aA_3$& 13 &  $\lambda=1;    c_1=c_2>c_3 =c_4=c_5 $ \\
 \hline
 MC  &$\aA_2\times \aA_2$&15 &  $\lambda=1; c_1=c_2=c_3 >c_4=c_5$ \\
 \hline
 MD & $\aA_4$&10 & $ \lambda=1; c_1=c_2=c_3 =c_4>c_5  $ \\
\hline
 MOA  & $\aA_3$& 14& $\lambda<1; c_1>c_2=c_3 =c_4=c_5 $ \\
 \hline
 MOB  &$\aA_1\times \aA_2$& 16& $\lambda<1;  c_1=c_2>c_3 =c_4=c_5  $ \\
\hline
 MOC &$\aA_1\times \aA_2$& 16&    $\lambda<1;  c_1=c_2=c_3 >c_4=c_5$ \\
\hline
 MOD   & $\aA_3$& 14& $\lambda<1; c_1=c_2=c_3 =c_4>c_5 $ \\
 \hline
 MAB  & $\aA_3$& 14& $\lambda=1;  c_1>c_2>c_3 =c_4=c_5 $\\
 \hline
 MAC &$\aA_1\times \aA_1\times \aA_1$& 17&   $\lambda=1; c_1=c_2>c_3 >c_4=c_5$\\
 \hline
 MAD &$\aA_3$& $14$ & $\lambda=1$;  $ c_1>c_2=c_3 =c_4>c_5$   \\
 \hline
 MBC  &$\aA_1\times \aA_1\times \aA_1$& 17& $\lambda=1; c_1>c_2=c_3 >c_4=c_5$   \\
 \hline
 MBD &$\aA_1\times \aA_2$& $16$ & $\lambda=1$;  $ c_1=c_2>c_3 =c_4>c_5$   \\
 \hline
 MCD &$\aA_1\times \aA_2$& $16$ & $\lambda=1$;   $c_1=c_2=c_3 >c_4>c_5  $ \\
 \hline
 MOAB&$\aA_2$& $17$ & $\lambda<1$;  $c_1>c_2>c_3 =c_4=c_5$\\
 \hline
 MOAC&$\aA_1\times \aA_1$ & $18$ & $\lambda<1$;  $c_1 >c_2=c_3 >c_4=c_5$ \\
 \hline
 MOAD&$\aA_2$& $17$& $\lambda<1$;   $c_1> c_2=c_3=c_4 >c_5$\\
 \hline
 MOBC&$\aA_1\times \aA_1$ & $18$& $\lambda<1$;   $c_1 =c_2 >c_3>c_4=c_5$\\
 \hline
 MOBD&$\aA_1\times \aA_1$ & $18$ & $\lambda<1$;  $c_1=c_2 >c_3=c_4> c_5$\\
 \hline
 MOCD&$\aA_2$& $17$ & $\lambda<1$; $c_1=c_2=c_3 >c_4 >c_5$ \\
 \hline
 MABC&$\aA_1\times \aA_1$ & $18$ & $\lambda=1$; $c_1 >c_2> c_3 >c_4=c_5$\\
 \hline
 MABD&$\aA_2$& $17$&  $\lambda=1$; $c_1 >c_2 >c_3=c_4> c_5$\\
 \hline
 MACD &$\aA_1\times \aA_1$ & $18$ & $\lambda=1$; $c_1> c_2=c_3 >c_4> c_5$\\
 \hline
 MBCD &$\aA_1\times \aA_1$ & $18$ & $\lambda=1$; $c_1=c_2 >c_3 >c_4 >c_5$\\
 \hline
 MOABC &$\aA_1$ & 19 &$\lambda<1$;  $c_1> c_2> c_3 >c_4=c_5$\\
 \hline
 MOABD &$\aA_1$  & 19 & $\lambda<1$; $c_1 >c_2 >c_3=c_4 >c_5$ \\
 \hline
 MOACD &$\aA_1$  & 19 & $\lambda<1$;  $c_1 >c_2=c_3> c_4> c_5$\\
 \hline
 MOBCD  &$\aA_1$  & 19 & $\lambda<1$;  $c_1=c_2 >c_3 >c_4 >c_5$\\
 \hline
 MABCD  &$\aA_1$  & 19 &$\lambda=1$;  $c_1 >c_2> c_3> c_4> c_5$\\
 \hline
 MOABCD  & trivial & 20 & $\lambda<1$; $c_1 >c_2> c_3> c_4> c_5$ \\
 \hline\hline
\end{tabular}

\caption{Reduced symplectic form on $\CC P^2\# 5\overline{\CC P^2}$, note that here $\lambda=c_1+c_2+c_3$.}
\label{5form}
\end{center}
\end{table}

\begin{figure}[h]
\vspace{-4mm}
\begin{center}
  \begin{tikzpicture}

    \draw (0,0) -- (2,0);
    \draw (2,0) -- (3,-.5);
    \draw (2,0) -- (3,.5);

    \draw[fill=blue] (0,0) circle(.1);
    \draw[fill=blue] (1,0) circle(.1);
    \draw[fill=blue] (2,0) circle(.1);
    \draw[fill=blue] (3,-.5) circle(.1);
    \draw[fill=blue] (3,.5) circle(.1);
    \node at (3,.9) {$MO$};
    \node at(3,-.9) {$MD$};
    \node at(2,0.4) {$MC$};
     \node at(1,0.4) {$MB$};
       \node at(0,0.4) {$MA$};
    \node at (-1,0) {$\DD_5$};
  \end{tikzpicture}
\end{center}
  \caption{Lagrangian system for $\w_{mon}$}\label{5Dmon}
\end{figure}
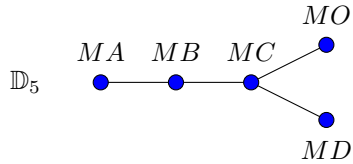

  The normalized reduced cone $P_5$ is the half-open convex cone spanned by $5$ vertices $M=(\frac13,\frac13,\frac13,\frac13,\frac13),$ $O=(0,0,0,0,0),$ $A=(1,0,0,0,0),$
$B=(\frac12,\frac12,0,0,0),$ $C=(\frac13,\frac13,\frac13,0,0),$ $D=(\frac13,\frac13,\frac13,\frac13,0)$, with the facet spanned by $OABCD$ removed.

Recall the rest of facets of the polytope $P_5$ correspond to classes with $\w(l_i)=0$ for some $i$.  We frequently use a ``co-notation'' to label the five edges $\{MO,MA,MB,MC,MD\}$ with the unique Lagrangian root $l_i$ such that $\w(l_i)>0$ if $\w$ lies on the corresponding edge.
For example, the correspondence reads
 $MO\to l_1=H-E_1-E_2-E_3,$ $MA\to l_2=E_1-E_2,$ $MB\to  l_3=E_2-E_3,$
$MC\to l_4=E_{3}-E_4,$ $MD\to l_5=
E_4-E_5$.

Similarly, we give a label to each face interior according to those $l_i$ which pair positively with forms in it.  All possible chambers are listed in Table \ref{5form}, and the corresponding Lagrangian roots can be read from its vertices by counting in all edges emanating from the vertex $M$.  For example, the Lagrangian labels of $MOAB$ is given by $MO\to l_1$, $MA\to l_2$, $MB\to l_3$ and $MC\to l_4$, and these four roots pair positively with $\w$ hence cannot be represented by Lagrangian spheres, but $MD$ pairs trivially with $\w$ and remains representable by a Lagrangian sphere.  $\Gamma_L(\w)$ can be read off by removing the labeling roots of $\w$ from the Dynkin diagram $\DD_5$.  For example, forms in the interior of the face $MOAC$ will have $\Gamma_L(\w)$ being the union of $MB$ and $MD$, which yields $\aA_1\times\aA_1$.  Table \ref{5form} allows us to divide types of symplectic forms in the following way.

\begin{defn}\label{d:Types}
  We call a symplectic form $\w$ on $X_5$ of \textbf{type $\aA$} if its Lagrangian system $\Gamma_{\w}$ is a product of $\aA_i$'s (when $N_w>8$ in table \ref{5form}) and \textbf{ type $\DD$} when  $N_\w\le 8$.
\end{defn}

\begin{rmk}\label{5r}

One might note that the removed facet $OABCD$ is itself the $P_4$, the normalized reduced cone of the ${\CC P^2} \# 4\overline{\CC P^2}$.
\end{rmk}

\subsubsection{ The digram of fibrations \eqref{summary}}
In this section, we prove the fibration property of the sequence \eqref{summary}.  We use the same choice of configuration $C$ as in \cite{Eva11}.   This is a configuration of six smooth exceptional spheres which are transverse and positive at every intersections.  The homology classes of these spheres and intersection patterns are shown in the following diagram.

\begin{figure}[ht]
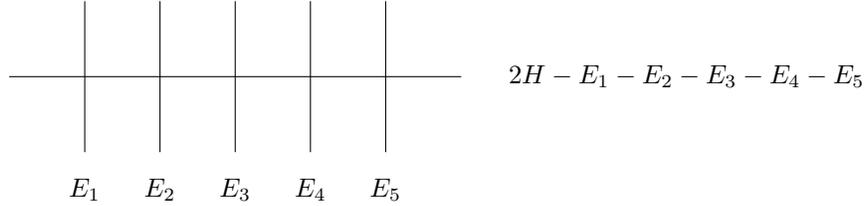

\begin{center}
\[
\xy
(0, -10)*{};(60, -10)* {}**\dir{-};
(90, -10)* {2H-E_1-E_2-E_3-E_4-E_5};
(10, 0)*{}; (10, -20)*{}**\dir{-};
(10, -25)*{E_1};
(20, 0)*{}; (20, -20)*{}**\dir{-};
(20, -25)*{E_2};
(30, 0)*{}; (30, -20)*{}**\dir{-};
(30, -25)*{E_3};
(40, 0)*{}; (40, -20)*{}**\dir{-};
(40, -25)*{E_4};
(50, 0)*{}; (50, -20)*{}**\dir[red, ultra thick, domain=0:6]{-};
(50, -25)*{E_5};
\endxy
\]
  \caption{Configuration of $X_5$}\label{conf5}
\end{center}
\end{figure}

\begin{prp}\label{fib5}
  Given a configuration $\CC P^2  \# 5{\overline {\CC P^2}}$ with any symplectic form $\omega$ and the above choice of $C$, the diagram \eqref{summary} is a fibration.  Further, $Symp_c(U)$ is weakly homotopic equivalent to $Symp_c(T^*\RR P^2, \w_{std})$, where $U$ is the complement of $C$.
\end{prp}

\begin{lma}[\cite{Eva11} section 4.2]\label{symgau}

Assume $C$ has $k$ irreducible spherical components, i.e.$C=\cup_{j=1}^k  C_j$, and each $C_j$ has $r_j$ intersection points with others.
\begin{equation}\label{sympc}
 Symp(C)= \prod_{j=1}^k Symp(S^2,r_j).
\end{equation}
  \begin{equation} \label{sympk}
Symp(S^2,1)\sim Symp(S^2,2)\sim S^1;\hspace{5mm}
Symp(S^2,3)\sim {\star};\hspace{5mm}
  Symp(S^2,n)\sim \Diff^+(S^2,n).
\end{equation}
\begin{equation}\label{mG}
 \mG(C) \cong \oplus_{j=1}^k \mG_{r_j}(S^2)\cong \oplus_{j=1}^k  \mathbb Z^{r_j-1}.
\end{equation}

  In particular, for the above $C$ in Figure \ref{conf5},  $Symp(C)\sim(S^1)^5\times  \Diff^+(S^2,5)$, and $\mG(C)\sim\ZZ^6.$
\end{lma}

We will denote the space of transverse symplectic configurations of the above type $\msC$.  To apply further techniques to the configurations of curves, we usually require the different sphere components to intersect in an $\w$-orthogonal way.  Denote $\msC_0$ as the subspace of $\msC$ consisting of such $\w$-orthogonal configurations.  The following weak homotopy equivalence between the space of curve configurations and space of almost complex structures follows from Gompf isotopy:

 \begin{lma}[\cite{Ev11} Lemma 26, or \cite{LL16} Lemma 4.3]\label{CeqJ}
Given  $({\CC P^2} \# 5{\overline {\CC P^2}},\w)$, the inclusion
$\msC_0\hookrightarrow\msC$ induces a weak homotopy equivalence.

Denote by $\mJ_{C}$ the set of almost complex structures $J$ which allows a $J$-holomorphic configuration $C$, then $\msC$ is weakly homotopic to $\mJ_{C}$.
 \end{lma}
We remind the reader that although all classes in the above configuration admits a $J$-holomorphic representative for all $\w$-tame $J$, we require in $\mJ_C$ that such representatives must be \textit{smooth}.

The next lemma gives $\msC\sim \mJ_{open}$ (see definitions below equation \eqref{summary}).

\begin{lma}\label{-1open}
 Let $\CP^2\# 5{\overline {\CC P^2}}$ be equipped with a reduced symplectic form, and configuration $C$ of exceptional spheres in the classes given as above.
  Then we know the space  $\mJ_{C}$ is $\mJ_{open}$. Moreover, the space $\msC$ is homotopic to $\mJ_{open}$.
\end{lma}

\begin{proof}
  The first statement is by Lemma 3.17 in \cite{LL16} the last case. The second statement is by the above Lemma \ref{CeqJ}.
\end{proof}

To prove the fibration property, we will need the following construction called the {\bf ball-swapping} following \cite{Wu13}:

\begin{dfn}\label{d:ball-swapping}
Suppose $X$ is a symplectic manifold. And $\widetilde{X}$ a blow up of $X$ at a packing of  $n$ balls $B_i$.
Consider the ball packing in $X$  $\iota_{0}:\coprod_{i=1}^n B(i)\rightarrow X,$ with image $K$.
Suppose there is a Hamiltonian isotopy  $\iota_{t}$ of $X$ acting on this ball backing $K$ such that $\iota_{1}(K)=K$,
then $\iota_{1}$ defines a symplectomorphism on the complement of $K$ in $X$.
 From the interpretation of blow-ups
in the symplectic category \cite{MP94}, the blow-ups can be represented as
$$\widetilde{X}=(X\backslash\iota_j(\coprod_{i=1}^n B_i))/\sim,\text{ for }j=0,1.$$
Here the equivalence relation $\sim$ collapses the natural $S^1$-action on $\partial B_i=S^3$. Hence $\iota_1$ as
symplectomorphism on the complement descends to a symplectomorphism  $\wt\iota: \widetilde{X}\to \wt X$.
\end{dfn}


The following fact is well-known.

\begin{lma}\label{relflux}
Let $Symp(S^2,n)$ denote the group of symplectomorphisms of the $n$-punctured sphere,
and $Symp(S^2, \coprod_{i=1}^n D_i)$ denote the group of symplectomorphisms of the complement of $n$ disjoint closed disks (each with a smooth boundary) which extend continuously to boundary in $S^2$. $Symp_0(S^2,n)$ and $Symp_0(S^2,\coprod_{i=1}^n D_i)$ are their identity components respectively. Then $Symp(S^2,n)$ is isomorphic to $Symp(S^2,\coprod_{i=1}^n D_i)$ and
 $$Symp(S^2,\coprod_{i=1}^n D_i)/Symp_0(S^2,\coprod_{i=1}^n D_i)= Symp(S^2,n)/  Symp_0(S^2,n),$$

where the right hand side is isomorphic to $\pi_0 Symp(S^2,n)= \pi_0 \Diff^+(S^2,n)$.

\end{lma}

\begin{proof}
The statement follows from a conjugation with the symplectomorphism $S^2-\{p_1,\dots,p_n\}\xrightarrow{\sim} S^2-\{\coprod_{i=1}^n D_i\}$, which yields an isomorphism between $Symp(S^2,\coprod_{i=1}^n D_i)$ and $Symp(S^2,n)$ which induces also an isomorphism $Symp_0(S^2,\coprod_{i=1}^n D_i)\cong Symp_0(S^2,n)$

\end{proof}


\begin{proof}[Proof of Proposition \ref{fib5}]
The rightmost term of diagram \ref{summary} was proved in \cite{Eva11} using Gompf isotopy and Banyaga extension Theorem.  We will focus on the rest of the diagram
\begin{equation}
\begin{CD}
Symp_c(U) @>^\sim>>  Stab^1(C) @>>> Stab^0(C) @>>> Stab(C) \\
@.@. @VVV @VVV  \\
@.@. \mG(C) @. Symp(C).
\end{CD}
\end{equation}

We first show that the restriction map $Stab(C)\to Symp(C)$ is surjective. Note that $Symp(C)=\prod_i Symp(e_i,p_i) \times Symp(Q,5),$ where $e_i$ are the curve components in class $E_i$, $Q$ is the curve component in class $2H-E_1-\cdots E_5$, and $p_i$ are the intersections $e_i\cap Q$.

For any path  $\psi_t \subset Symp(C ) , t\in[0,1],$ it can be extended into an ambent isotopy in $Stab(C)$.  This is because $H_2(X)$ is integrally spanned by homology classes of curves in $C$, which implies that $H_2(M, C;\RR)=0,$ and hence Banyaga extension applies.

Therefore, to prove the surjectivity of $Stab(C)\to Symp(C)$, it suffices to lift an arbitrary choice of 2-dimensional mapping $h^{(2)}\in Symp(C)$ in each connected component of $Symp(C)$ and extend it to a 4-dimensional mapping $h^{(2)}$ to $Stab(C)$, then compose it with a Banyaga extension given above.  Since $Symp(e_i,p_i)$ are connected, these connected components are identified with connected components of $Symp(Q,5)$, or rather, elements in $\pi_0(\text{Diff}^+_5(Q))$.

To this end, we blow down the exceptional spheres $e_1, \cdots, e_5$, and obtain a pair $(\CC P^2, \coprod_{i=1}^5 B(i))$ with a conic $\ov Q\subset \CP^2$ as the proper transform of $Q$, and the five disjoint balls  $\coprod_{i=1}^5 B(i)$ intersect \textit{nicely} with $\ov Q$.  This means, when $B(i)$ is regarded as the image of an embedding $\psi_i:B(c_i)\hookrightarrow \CP^2$, where $B(c_i)\subset (\CC^2,\w_{std})$ is the standard ball of radius $c_i$, then $\psi_i^{-1}(Q)\subset \{z_2=0\}\subset \CC^2$.

Note that by the above identification in Lemma \ref{relflux}, this blow-down process sends any $h^{(2)}$ in $Symp(Q,5)$ to a unique $\overline {h^{(2)}} $ in $Symp(S^2,\coprod_{i=1}^5 D_i)$, where $D_i=Q\cap B(i)$.  We claim that there exists $\overline {h^{(4)}}\in Symp(\CP^2)$ whose restriction is $\overline {h^{(2)}}$, and it setwise fixes the image of each ball $\coprod_{i=1}^5 B(i)$, as well as $\ov Q$.  This two-step construction follows that in \cite{Wu13}: we first regard $h^{(2)}$ to be a Hamiltonian diffeomorphism on $Q$, and extend it to $f^{(4)}\in Ham(\CP^2)$ which has support near $\ov Q$.  The second step is done by the connectedness of  ball packing relative to a divisor (the conic $\ov Q$ in our case). Namely, by Lemma 4.3 and Lemma 4.4 in \cite{Wu13}, there exists a symplectomorphism $g^{(4)}\in Symp(\CC P^2,\w)$ which sends $f^{(4)}\circ \psi_i(B(i))$ to $\psi_i(B(i))$ while fixing $\ov Q$ pointwise.  Therefore, the composition $\overline {h^{(4)}}=g^{(4)}\circ f^{(4)}$ is a symplectomorphism fixing the five balls, and induces a symplectomorphism $h^{(4)}\in Symp(\CP^2\#5\ov{\CP^2})$ through the ball-swapping construction.  Clearly, $h^{(4)}$ induces the same symplectomorphism $h^{(2)}$ on $Q$ by the correspondence in \ref{relflux}, hence has the desired properties.



Next, we want to see that $Stab(C)\to Symp(C)$ is a fibration using Theorem A in \cite{Mei02} and let's recall the theorem here.  Recall that a \textbf{vanishing $p$-cycle} is a fibred map $f : S^p\times [0, 1]\to E$ such that, for each $t>0$, the map $f_t$ is null-homotopic in its fibre. Call $f$ \textbf{trivial} if $f_0$ is also null-homotopic in its fibre.  An \textbf{emerging $p$-cycle} is a fibred map
$f : S^p\times (0, 1]\to E$
such that $f(\bullet, t)$ has a limit for $t \to 0$ (recall that $\bullet$ denotes the base point in $S^p$).
Call it \textbf{trivial} if there exist $\epsilon>0$ and a fibred map
$f' : S^p \times [0, \epsilon)\to E$
such that for each $0 <t<\epsilon$ one has $f'(\bullet, t) = f(\bullet, t)$, and such that the maps $f_t$ and $f_t'$ are homotopic to each other in their common fibre, relatively to the basepoint
$f(\bullet, t)$.

  \begin{thm}[\cite{Mei02} Theorem 1.1] \label{fib3}
  A surjective map is a fibration if and only if it satisfies the following
three conditions:  1) the map is a submersion; 2)  all vanishing cycles of all dimensions are trivial; 3) all emerging cycles are trivial.
  \end{thm}

  We now check these 3 properties: For 1), at any given point $p$ in $Symp(C)$ and $\bar p\in \pi^{-1}(p)$, consider a tangent vector $\vec{v}\in T_p(Symp(C))$ represented by a path $\gamma_t$.  We can lift $\gamma_t$ into a path $\Gamma_t\subset Stab(C)$ such that $ \Gamma_0 =\bar p$  and  $\Gamma_t|_C= \gamma_t \circ \Gamma_0|_C$ for all $t$: this is easy to see when $p$ and $\bar p$ are identities, and such a lift can be obtained by conjugating $\bar p$ in the general case.  Then $\Gamma_t'(0)$ is a global Hamiltonian vector field that restricts to $\gamma_t'(0)$ on $C$, implying the projection is a submersion.

  For 2) and 3), consider any fiberwise continuous map $  S^p \times [0,1]\to Stab(C)$ in the definition of vanishing and emerging cycles, we have a path $ \psi_t:=f(\bullet,t), t\in [0,1]$ in $Stab(C)$ given by the base point over a the path $\gamma_t$ on the base.  Along the path $\psi_t, t\in [0,1],$ the fibers can be identified continuously to the fiber over $\pi(\psi_0)$ by the left multiplication of $\psi_0\cdot \psi_t^{-1}$. This creates the needed isotopy for vanishing cycles.  For emerging cycles, this path of identification gives a well-defined fiberwise homotopy class in each fiber.  One may take a map $S^p\to F_{\psi_0}$ with this homotopy class with base point on $\psi_0$, and propagate to a small neighborhood of $t\in[0,\epsilon)$ again by multiplications of the base point elements.  By definition, this yields the desired extension.  This concludes $Stab(C)\to Symp(C)$ is a fibration.

Then rest parts of the diagram being a fibration follows exactly the same arguments in \cite{Eva11} Proposition 34 and  \cite{LLW15} Lemma 2.4, which we will not replicate here.

Our last task is to show the weak homotopy equivalence $Symp_c(U) \sim Symp_c(T^*\RR P^2)$, by a very similar argument as Lemma 3.3 in \cite{LLW15}.
It's known that $U$ is diffeomorphic to $T^{*} \RR P^2$ by section 6.5 in \cite{Eva11}.  When $[ \omega] \in H^2(X;\QQ)$, up to rescaling we can write  $PD([ l\omega])=a H-b_1 E_1-b_2 E_2-b_3 E_3-b_4 E_4-b_5 E_5$ with
$ a, b_i \in \ZZ^{\geq 0}$. Further, we assume $ b_1\geq \cdots \geq b_5$. Then we can represent $PD([l\omega])$  as a
positive integral combination of all elements in the set $\{ 2H-E_1-E_2-E_3-E_4-E_5, E_1,E_2,E_3, E_4, E_5 \}$, which is the set of homology classes  of the components in $C$. This means that for a rational form $\w_{\QQ}$, $(U,\w_{\QQ})$ is a Stein domain with the same Stein completion as the complement of the  monotone case, which is a  $(T^*\RR P^2, \w_{std})$.  Hence we have the desired weak homotopy euqivalence for rational symplectic forms.  When $\w'$ is not rational, the same statement as Claim 3.4 in \cite{LLW15} holds and hence we can isotope any map  $S^n\to Symp_c(U,\w')$ into $S^n\to Symp_c(U,\w_{\QQ})$. Therefore, we have the desired weak homotopy equivalence for any symplectic form.

\end{proof}

\subsection{\texorpdfstring{$\RR P^2$}{RP2} packing symplectic form}



In this section, we will prove the vanishing of $\pi_0(Symp_h(X_5))$ in an open subset of the reduced cone.  Through ball-swappings, we will construct an explicit set of representatives of the image of
\begin{equation}\label{e:goal}
       \psi: \pi_0(Stab(C))\to \pi_0(Symp_h(X,\w))
\end{equation}

in \eqref{fom} and prove that they are Hamiltonian isotopic to identity.

We start by recalling a result of relative ball packing in $\CC P^2$:

\begin{lma}\label{relpack}
Given a symplectic $\CP^2$ with $\bar\w(H)=1$, a sequence $\{c_i\}_{i\le5}$ such that $max\{c_i\}\leq 1/2$ and a Lagrangian $\RP^2\subset (\CP^2,\bar\w)$. 
  Then there is a ball packing of $\iota:\coprod_{i=1}^5 B(c_i)\rightarrow (\CC P^2-\RP^2,\bar\w)$.  As a result, we have an embedded Lagrangian $\RP^2$ in $\CP^2\#5\overline{\CP^2}$ with $[\w]=(1|c_1,\cdots,c_5)$ when $c_i<\frac{1}{2}$ and $\underset{1\leq i\leq 5} \sum  c_i <2$.
\end{lma}
\begin{proof}

By \cite{BLW12} Lemma 5.2,  it suffices to pack $5$ balls of given sizes $c_i$
    into $(S^2\times S^2,\Omega_{1,\frac{1}{2}})$, where $\Omega_{1,\frac{1}{2}}=\sigma\oplus \frac12 \sigma$ and $\sigma$ is the volume form with area 1 on $S^2$.
Without loss of generality we assume that  $ c_1\geq \cdots\geq c_5$.
    From equation \eqref{HB}, blowing up a ball of size $c_1$ (here by ball size we mean the area of the corresponding exceptional sphere)  in
    $(S^2\times S^2,\Omega_{1,\frac{1}{2}})$ is symplectomorphic to
    $(\CC P^{2}\#2 \overline{\CC P}^{2}, \w')$ with $\w'$ dual to the class
    $(\frac{3}{2}-c_1)H-(1-c_1)E_1-(\frac{1}{2}-c_1)E_2$.  Therefore, by Lemma 5.2 in \cite{BLW12}, it suffices to prove that
    the vector
    $$[(\frac{3}{2}-c_1)|(1-c_1),c_2,c_3,c_4,c_5,(\frac{1}{2}-c_1)]$$ denoting the class $$ [S]=(\frac{3}{2}-c_1)H - (1-c_1)E_{1} - (\frac{1}{2}-c_1)E_{6} - \sum_{i=2}^{5}c_i E_{i}$$
    is Poincar\'e dual to a symplectic form for $\CC P^{2}\#6 \overline{\CC P}^{2}$.\\

    The square of this class is $  [S] \cdot [S]=1-\sum c_i^2$, which is clearly positive under our assumptions.   From \cite[Theorem 4]{LL01},  in order to check $[S]$ is Poincare dual to symplectic class, one only needs to check that it pairs positively with all exceptional classes that have canonical class $3H-E_1-\cdots-E_5$:

\begin{itemize}

  \item Clearly, $[S]\cdot E_i>0, \quad \forall i.$

  \item By the reducedness condition \ref{reducedHE}, the minimal value of $[S]\cdot( H-E_i-E_j)$ is either

$$(\frac{3}{2}-c_1)-  c_2-c_3>0;$$
$$or \quad (\frac{3}{2}-c_1)-  c_2 -(1-c_1) >0;$$
    this  means $PD[S]$ is positive on each exceptional class $H-E_i-E_j$.

\item The  minimal value of $[S] \cdot ( 2H-   E_1 -\cdots - \check{E}_i- \cdots - E_6)$ is either

$$2(\frac{3}{2}-c_1)- (1-c_1)-c_2-c_3-c_4-(\frac{1}{2}-c_1)=2-c_2-c_3-c_4-(\frac{1}{2}-c_1)>0;$$

$$or \quad 2(\frac{3}{2}-c_1)- (1-c_1)-c_2-c_3-c_4-c_5= 2 -c_1-c_2-c_3-c_4-c_5>0;$$
this means means $PD[S]$ is positive on each  exceptional class $2H-   E_1 -\cdots - \check{E}_i- \cdots - E_6$, as desired.

\end{itemize}


\end{proof}

Since the class of forms in Lemma \ref{relpack} is the starting point of our proof, they deserve a name for future convenience.

\begin{dfn}\label{pacform}\label{paform}
  Given a symplectic form (not necessarily reduced) with $[\w]=(\nu|c_1, c_2, \cdots, c_5)$ on $\CC P^2  \# 5{\overline {\CC P^2}}$.
  It is called an {\bf $\RR P^2$ packing symplectic form} if
\begin{equation}\label{e:inequality}
     c_i<\nu/2,\quad \sum_{i=1}^5 c_i<2\nu. 
\end{equation}

 If $\nu=1$, $\w$ is called a {\bf standard $\RR P^2$ packing symplectic form}.
\end{dfn}


\begin{lma}\label{stab5}
If $\w $ is an $\RR P^2$ packing form, then $ Stab(C)  \simeq \Diff^+(S^2,5).$ And we have the exact sequence
\begin{equation}\label{e:ES5pt}
     1\rightarrow \pi_1(Symp_h(X,\w))\rightarrow \pi_1(\msC_0) \xrightarrow[]{\phi}  \pi_0(\Diff^{+}(S^2, 5)) \xrightarrow[]{\psi} \pi_0(Symp_h(X,\w)) \rightarrow 1
\end{equation}
\end{lma}

\begin{proof}
  We argue following \cite{Eva11}. Let's start from the left end of diagram \ref{summary}.  We always have $Stab^0(C)\sim Symp_c(U)$ by Moser argument. 
  From Lemma 19 in \cite{Eva11} and Lemma \ref{symgau}, $Stab^0(C)\sim\mG(C)\sim \ZZ^4$, and $\pi_1(Symp(C))=\ZZ^5$ surjects onto $\pi_0(Stab^0(C))=\ZZ^4$.

 Now we show that $\pi_1(Symp(C))$ also surjects onto $\pi_0(Symp_c(U))$.  Let $\mu$ be the moment map for the $SO(3)$-action on $T^*\RP^2$. Then
$||\mu||$ generates a Hamiltonian circle action on $T^*\RP^2\setminus\RP^2$ which commutes with the round cogeodesic flow. The symplectic cut along a
level set of $||\mu||$ gives $\CC P^{2}$ and the reduced locus is a conic. Pick
five points on the conic and $||\mu||$-equivariant balls
centered on them, with their volumes given by the symplectic form.  This is always possible since the form is a standard $\RR P^2$ packing form. $(\CC P^{2}\#5 \overline {\CC P^{2}}, \omega )$ is symplectomorphic to the blow up in these five balls
and the circle action preserves the exceptional locus. Hence by Lemma 36 in \cite{Eva11}, the  diagonal element
$(1,\ldots,1)\in \pi_1(Symp(C)) =\ZZ^5$ maps to the generator of the Dehn twist of the
zero section in $T^*\RP^2$, which is also the generator in $\pi_0(Symp_c(U))$.

The associated exact sequence yields an isomorphism $\pi_0(Stab(C))\xrightarrow{\sim}\pi_0(Symp(C))\cong \pi_0(\text{Diff}^+_5(S^2))$.  Indeed, we have a weak homotopy equivalence $ Stab(C)  \simeq \Diff^+(S^2,5)$ since all the higher order terms in the exact sequence vanish. \eqref{e:ES5pt} follows from the homotopy exact sequence associated to the rightmost fibration of \eqref{summary}.


\end{proof}

To understand the connecting map $\phi$ in \eqref{e:ES5pt} We now give a local toric model of ball-swapping in the complement of $\RR P^2$. From the Biran decomposition \cite{Bir01dec}, we know that $\CC P^2=\RR P^2\sqcup U$.  $U$ is a symplectic disk bundle over a sphere denoted as $Q$, 
with fiber area $1/2$ and base area $2$.  Later we'll see this bundle is total space of $\mO(4)$ over $Q$.

Given $5$ balls with sizes $a_1, a_2,\cdots, a_5$ satisfying \eqref{e:inequality} and $i,j\in\{1,2,3,4,5\}$.  Assume further that
\begin{equation}\label{e:size}
     a_i>a_j, \quad a_r>a_t>a_s.
\end{equation}

 Then there is a toric blowup as in Figure \ref{Swapij}.  By the correspondence in \cite{McD91}, this implies there is a symplectic packing of $\coprod B^4(a_l)$ in $U$ where
 $B^4(a_l)\cap Q$ is a large disk in $B^4(a_l)$.  Moreover, there is an ellipsoid $E_{ij}\subset U$, such that $B_i\cup B_j\subset E_{ij}$, and $E_{ij}$ is disjoint from the rest of the balls.
We call this an \textit{$(i,j)$-standard packing}.

\begin{rmk}\label{rem:toricpacking}
    We would like to remind the reader that the embeddings that these blow-ups represent do not have the exact images at the blow-ups, but in a small neighborhood of them.  In particular, they are \textit{not} invariant under the toric action.  This small perturbation does not affect properties we mentioned above and observed in the picture.

    Readers who are familiar with Karshon's theory \cite{Ka99} may visualize this subtlety by forgetting one of the circle action, and achieve the packing by $S^1$-equivariant blow-ups.  For example, instead of blowing up $B(a_j)$ in a toric way, one may blow it up in an $S^1$-equivariant way (with respect to the $S^1$-action represented by the positive $y$-axis in the picture).  This way, the image of $B(a_j)$ avoids the exceptional divisor from the blowup of $B(a_i)$ and hence gives a ball-packing of both $B(a_i)$ and $B(a_j)$.  This small perturbation does not affect the property that the packing is inside the ellipsoid $E_{ij}$ and also guarantee that $B^4(a_l)\cap Q$ are pairwise disjoint. All the above discussions applies to the blowup-packing correspondence of $B(a_s),B(a_t),B(a_r)$.
\end{rmk}

\begin{figure}[ht]
  \centering
  \includegraphics[scale=1]{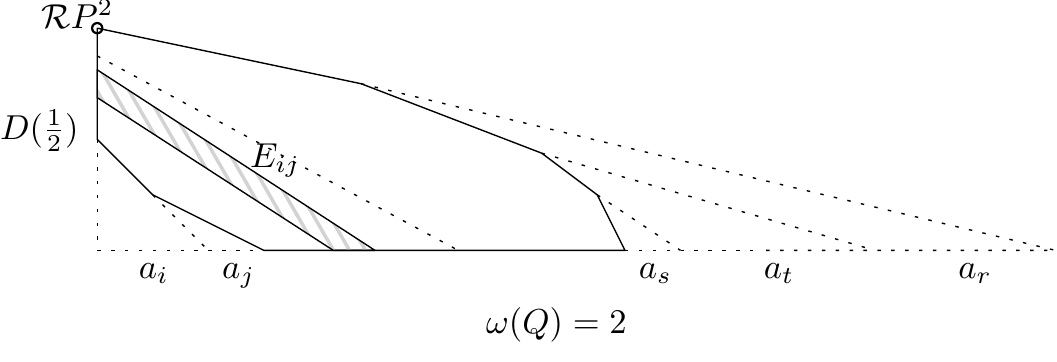}
  \caption{Standard toric packing and ball swapping in $\mO(4)$}
  \label{Swapij}
\end{figure}

Let us disregard $B(a_s), B(a_r)$ and $B(a_t)$ at the moment.  There is a natural circle action induced from the toric action on $U$, which rotates
the base curve $Q$ and fixes the center of $B^4(a_i)$.  The Hamiltonian of this circle action is $H(r_1,r_2)=|r_2|^2$, where $r_2$ is the vertical coordinate of the $\RR^2\cong \mathfrak{t}^*$. Clearly, this circle action runs the ball $B^4(a_j)$ around
 $B^4(a_i)$ exactly once, therefore, gives a ball-swapping.  When these two balls are blown-up, the corresponding ball-swapping that induces the pure braid generator $A_{ij}$ around the $(i,j)$-strands.

To put balls $B(a_s), B(a_r)$ and $B(a_t)$ back into consideration and invariant, we only need to make our construction above compactly supported.  Since the above Hamiltonian action is induced by $|r_2|^2$, we simply multiply a cut-off function $\eta(z_1,z_2)$  defined as following:
\begin{equation} \label{}
     \eta(z_1,z_2)=\left\{
        \begin{aligned}
          0, \hskip 3mm &x\in E_{ij} \setminus \{\mu^{-1}(r_1,r_2):\frac{r_1^2}{2-\epsilon-a_r-a_s-a_t}+
          \frac{r_2^2}{\dfrac{1}{2}-\epsilon}\leq 1/\pi\},\\
          1, \hskip 3mm &x\in \{\mu^{-1}(r_1,r_2):\frac{r_1^2}{a_i+a_j}+\frac{r_2^2}{\dfrac{1}{2}-2\epsilon}\leq 1/\pi\},
        \end{aligned}
     \right.
\end{equation}

where $\mu$ is the moment map from $U$ to $\RR^2$. The resulting Hamiltonian $\eta \circ H$ has a vanishing Hamiltonian vector field outside the ellipsoid in Figure \ref{Swapij}, and swap $B(a_i)$ and $B(a_j)$ as described above, hence descends to a ball-swapping as in Definition \ref{d:ball-swapping}.  We call such a symplectomorphism an $(i,j)$-\textbf{model ball-swapping} in $\mO(4)\#5\ov{\CC P^2}$ when $B_i$ and $B_j$ are swapped.  The following lemma is immediate from our construction.


\begin{lma}\label{l:BSproperties}
    The $(i,j)$-model ball-swapping is Hamiltonian isotopic to identity in the compactly supported symplectomorphism group of $\mO(4)\#5\ov{\CC P^2}$.  Moreover, it is an element in $Stab(C),$ which induces the generator $A_{ij}$ on $\pi_0(\Diff^{+}(S^2,5))$.
\end{lma}

\begin{rmk}\label{rem:equalsize}
    Note that when at least $2$ elements from $\{a_r,a_s,a_t\}:=
\{a_1, a_2,\cdots, a_5\}\setminus\{a_i, a_j\}$ coincide, toric packing as in Figure \ref{Swapij} doesn't exist.  Nonetheless, one could always slightly enlarge some of them to obtain distinct volumes satisfying equation \eqref{e:inequality}, then pack the original balls into the enlarged ones to obtain a standard packing.  Therefore, the above construction of model $(i,j)$-ball-swapping works as long as $a_i<1/2, \hskip 3mm \sum_i a_i<2$ holds.

The triviality of Hamiltonian isotopic class of $A_{ij}$ works equally well as long as we have $a_i>a_j$, since there is no isotopy needed outside the ellipsoid $E_{ij}$.
\end{rmk}

With these preparations, we can prove the vanishing of the Torelli symplectic mapping class group for a class of forms.

\begin{prp}\label{generator}
  Given $X=\CC P^2\# 5\overline{\CC P^2}$ with a {\bf standard $\RR P^2$ packing symplectic form} $\w \in [\w]=(1|c_1, c_2, \cdots, c_5)$ (i.e. \eqref{e:inequality} holds), and if either
\begin{itemize}
\item  there are at least 3 distinct values in $\{c_1, \cdots, c_5\};$ or
\item   there are 2 distinct values, and up to permutation of index, we have
$c_1=c_2>c_3=c_4=c_5$ or $c_1=c_2=c_3>c_4=c_5$. 
\end{itemize}
then $Symp_h( \CC P^2\# 5\overline{\CC P^2}, \omega)$ is connected, and rank of $\pi_1(Symp_h(X,\w))=N_{\w}-5$.
\end{prp}

\begin{proof}

Fix a configuration $C_{std}\in \mathcal{C}_0$ in $\CC P^2\# 5\overline{\CC P^2}$ with the given form $\omega$.  By looking at the sequence \eqref{e:ES5pt},
our goal is to find a generating set of $\pi_0\Diff^+(S^2,5)$ which has trivial image under the $\psi$-map.

From our assumption, there is a Lagrangian $\RR P^2$ away from  $C_{std}$.  Blowing down the exceptional curves, we have a ball packing $\iota: B_l=B(c_l)$ in the complement of $\RR P^2$.


Suppose $c_i> c_j$,  we have the semi-toric $(i,j)$-standard packing $\iota_s$ as defined above. One may further isotope $\iota$ to $\iota_s$,
by the connectedness of ball packing in \cite{BLW12} Theorem 1.1.    Clearly, from Lemma \ref{l:BSproperties}, the $(i,j)$-swapping is a Hamiltonian diffeomorphism, which fixes all exceptional divisors from all five ball-packing, and induces the pure braid generator $A_{ij}$ on $C$.  Also, Lemma \ref{l:BSproperties} and Remark \ref{rem:equalsize} implies the image of $A_{ij}$ under the $\psi$-map in \eqref{e:goal} is trivial.


Now we address the two cases in our Proposition.

\begin{itemize}
\item If there are at least 3 distinct values in $\{c_1, \cdots, c_5\}$, then a generating set as in Lemma
\ref{Pbraid} Case 1) has trivial images under $\psi$, hence
$Symp_h$ is connected.  To see this, we can do a permutation on $\{1,2,3,4,5\}$ to make $c_1>c_2>c_3\geq c_4$ or $c_1<c_2<c_3\le c_4$.
Then Lemma \ref{l:BSproperties} and Remark \ref{rem:equalsize} concludes that $\{ A_{12}, A_{13},A_{14},A_{23},A_{24} \} $ are trivial under $\psi$.

\item In the second case, up to permutation of indices, we have triviality of $\psi$-image of $\{A_{13}, A_{14},A_{15},A_{23},A_{24},A_{25}\}$, which is a generating set in Case 2) of Lemma \ref{Pbraid}.

\end{itemize}

\end{proof}

\subsection{Type \texorpdfstring{$\aA$ }{A} forms}

 We now extend the result of Proposition \ref{generator} to the whole type $\aA$ part the reduced symplectic cone.  The key technique is the Cremona transform and Proposition \ref{nonbalstab}.

Recall a (homological) \textbf{Cremona transform} of a rational surface $X$ is an automorphism of $H_2(X,\ZZ)$ defined by the reflection of a class $H-E_i-E_j-E_k$ for pairwise distinct $i,j,k$.  A Cremona transform is realized by a diffeomorphism on $X$, or it could be considered as a change of basis.  See more details from \cite{MS12}.  When there is no confusion, sometimes we use Cremona transform to refer to a diffeomorphism which induces a homological Cremona transform.

In this section, we'll first introduce the balanced symplectic forms, and then show that many type $\aA$ symplectic forms are diffeomorphic to a balanced form.
Finally, we apply Proposition \ref{nonbalstab} to extend these results to arbitrary type $\aA$ symplectic forms.

\subsubsection{ Balanced symplectic forms and \texorpdfstring{$\RR P^2$}{RP2}-relative ball packing}

Firstly we introduce the following definition

\begin{dfn}\label{balform}
We call a reduce form $(1|c_1, c_2, \cdots, c_n)$ on $\CC P^2  \# 5{\overline {\CC P^2}}$  {\bf balanced} if
$c_i<c_{i+1}+c_{i+2}$ for some $1\le i\le 3$.  In particular, all the non-balanced forms are contained in the open stratum $MOABCD$.

\end{dfn}

And we show that this is related to $\RR P^2$ packing:

\begin{lma}\label{Cremona}
Any balanced reduced form $\w_b= (1|c_1, c_2, \cdots, c_5)_b$
is diffeomorphic to a standard $\RR P^2$ packing symplectic form $\w_p = (1|c'_1, c'_2, \cdots, c'_5)_p$.

\end{lma}

\begin{proof}

Take any reduced form $\w_b =(1|c_1, c_2, \cdots, c_5)$ on $\CC P^2\# 5\overline{\CC P^2}$, then it satisfies $c_1\geq  c_2\geq c_3 \geq c_4 \geq c_5$. To obtain a packing form, we can simply start with any balanced condition $c_i<c_{i+1}+c_{i+2}$.  Perform a  Cremona transform along $H-E_i-E_{i+1}-E_{i+2}$.  This is captured by the matrix

 \begin{equation}\label{e:Cremona}
  \Phi_*\left( \begin{array}{c}
     H \\
     E_i \\
     E_{i+1} \\
     E_{i+2}\\
     E_j\\
     E_k
  \end{array} \right)=
  \left( \begin{array}{c}
     2H-E_i-E_{i+1}-E_{i+2}\\
     H-E_{i+1}-E_{i+2} \\
     H-E_{i}-E_{i+2}  \\
     H-E_i-E_{i+1}\\
     E_j \\
     E_k
  \end{array} \right):=\left( \begin{array}{c}
    h \\
     e_1\\
     e_2\\
     e_3\\
     e_4\\
     e_5
  \end{array} \right),
\end{equation}

 where $\{h,e_1,\cdots,e_5\}$ is regarded as a new standard basis. Now we need to check $\w(h)>2\w(e_i)= \Phi^*(\w(E_i))$, which will conclude the first statement.

 \begin{itemize}
\item $h-2e_1=2H-E_i-E_{i+1}-E_{i+2}-2H+2E_{i+1}+2E_{i+2}=
 E_{i+1}+E_{i+2}-E_i$, which has positive $\w$-area by the balanced assumption $c_i<c_{i+1}+c_{i+2}$;
\item For $h-2e_2=2H-E_i-E_{i+1}-E_{i+2}-2H+2E_{i}+2E_{i+2}=
  E_{i}+E_{i+2}-E_{i+1}$, and by the reducedness condition $c_i>c_{i+1}$ it has positive area;
\item $h-2e_3$, by the same reasoning as $h-2e_2$, it has positive symplectic area;
\item$h-2e_4= 2H-E_i-E_{i+1}-E_{i+2}-2E_j=( H-E_i-E_{i+1}-E_j) + H-E_{i+2}-E_j$. By reduced condition, $( H-E_i-E_{i+1}-E_j)$ has non-negative and $ H-E_{i+2}-E_j$ has positive area, and hence $h-2e_4$ has positive symplectic area;
\item For $h-2e_5$ we can apply the same argument as  $h-2e_4$.
\end{itemize}

Note that $\w(2h-e_1-\cdots-e_5)>0$ is automatic because $\w$ pairs positively with any exceptional classes.
\end{proof}

\subsubsection{Triviality of Torelli SMCG for an arbitrary type \texorpdfstring{$\aA$}{A} symplectic form}

Now we are ready to deal with an arbitrary type $\aA$ symplectic form via the above Cremona transform and stability of $Symp(X,\w)$.  For the duration of the following proof, we refer the readers to Table \ref{5form} for case checks.

\begin{prp}\label{p:balanced}

 Given $X=(\CC P^2\# 5\overline{\CC P^2}, \omega)$, let $\w$ be  a reduced form of type $\aA$,  then $Symp_h(X)$ is connected.
\end{prp}
\begin{proof}

 The type $\aA$ assumption allows us to restrict our attention to balanced forms on any $k-$face for $k\geq2$ and edges $MO,MB,MC$ and $MD$.

\begin{enumerate}

  \item \emph{$k-$faces, $k\geq3$, or $2-$faces or edges where $A$ is a vertex.}
    If $c_1<\frac12$, then it is automatically an $\RR P^2$-packing form with 3 distinct values. By Proposition \ref{generator}, we know $Symp_h$ is connected. Then  Proposition \ref{nonbalstab}, on each open face, allows us to remove the assumption $c_1<\frac12$. This is because the ray starting from vertex $A$ covers each open face with $A$ being a vertex, meanwhile,  on each ray there's a part with $c_1<\frac12$.

\item \emph{$k-$faces, $k\geq3$, or $2-$faces or edges where $A$ is not a vertex.}
 $A$ is not an open face means that $c_1=c_2<\frac12.$ These are still $\RP^2$ packing form with 3 distinct values. Then this is covered by Proposition \ref{generator}.

\item \emph{Other faces except for $MO$.} These faces will have $c_i>c_{i+1}$ for some $i=1,2,3,4$ and $\{c_i\}$ contain only two distinct values.
  By definition, these forms are clearly balanced.  Therefore, by Lemma \ref{Cremona} they are diffeomorphism to a $\RR P^2$ packing form.  And below we give an explicit Cremona transform and show these cases are covered by Proposition \ref{generator} case 1).

  Consider the Cremona transform with respect to either $H-E_1-E_{2}-E_{3}$ when $i=1, 2$; or $H-E_{3}-E_{4}-E_{5}$ when $i=3,4$.

\begin{itemize}
  \item If $i=1,2$, the resulting new basis reads
$$h=2H-E_1-E_2-E_3, \hskip 3mm e_1=H-E_2-E_3, \hskip 3mm e_2=H-E_1-E_3, $$
 $$e_3= H- E_1-E_2, \hskip 3mm e_4=E_4,  \hskip3mm e_5=E_5.$$

It is straightforward to check that $\w(h)>2\w(e_i)$ by the balancing and reducedness conditions, hence the pull-back form is an $\RP^2$ packing form.

If $i=1$, $c_1>c_2$, then $\w(H-E_2-E_3)>\omega(H-E_1-E_3)> \omega(E_5)$.  The last inequality holds because $\w$ is not of type $\DD_4$ hence $\lambda<1$.  This yields three different values in $\w(e_1)$, $\w(e_2)$ and $\w(e_5)$ hence Proposition
\ref{generator} applies.

If $i=2$, $\w(H-E_1-E_2)>\omega(H-E_1-E_3)\ge\omega(E_2)> \omega(E_5)$.  Therefore, $\w(e_3),\w(e_2)$ and $\w(e_5)$ again gives three distinct values and Proposition \ref{generator} applies.
  \item For $i=3,4$, the Cremona transform gives
$$h=2H-E_3-E_4-E_5, \hskip 3mm e_1=E_1, \hskip 3mm e_2=E_2, $$
 $$e_3= H- E_4-E_5, \hskip 3mm e_4=H-E_3-E_5,  \hskip3mm e_5=H-E_3-E_4.$$

If $i=3$, $\w(h)>2\w(e_i)$ is easy to check by the balance and reducedness of $\w$.  Also, $\w(E_1)\ge\w(H-E_2-E_3)>\w(H-E_3-E_4)>\w(H-E_4-E_5)$, yielding three distinct values $\w(e_1),\w(e_5)$ and $\w(e_3)$.

For $i=4$, we again can check this is a packing form, and then we have two possibilities.  If $\lambda<1$, $\omega(E_1)> \omega(H-E_4-E_5)> \omega(H-E_3-E_5)$, giving three distinct values $\w(e_1), \w(e_3)$ and $\w(e_4)$.  If $\lambda=1$, we would have $\w(e_3)=\w(e_4)>\w(e_5)=\w(e_1)=\w(e_2)$, which falls into the second case of Proposition \ref{generator}.  Either way we have the desired triviality of $Symp_h(X)$.

\end{itemize}

\item \emph{Case $MO$, where $c_1=\cdots=c_5$}. All such forms are balanced, hence by Lemma \ref{Cremona} they are diffeomorphism to a $\RR P^2$ packing form.  Applying a Cremona transform along $H-E_3-E_4-E_5$, one has the following configuration

\[
\xy
(-5, -10)*{};(85, -10)* {}**\dir{-};
(98, -12)* {=H-E_1-E_2};(98, -8)* {2h-e_1-\cdots- e_5};
(0, 0)*{}; (0, -20)*{}**\dir{-};
(0, -25)*{e_1=E_1};
(20, 0)*{}; (20, -20)*{}**\dir{-};
(20, -25)*{e_2=E_2};
(40, 0)*{}; (40, -20)*{}**\dir{-};
(40, 1)*{e_3=H-E_4-E_5};
(60, 0)*{}; (60, -20)*{}**\dir{-};
(60, -25)*{e_4=H-E_3-E_5};
(80, 0)*{}; (80, -20)*{}**\dir[red, ultra thick, domain=0:6]{-};
(80, 1)*{e_5=H-E_3-E_4};
\endxy
\]

It is again easy to check that this pull-back form is an $\RP^2$ packing form from reducedness.  We also have
$\w(e_3)=\w(e_4)=\w(e_5)>\w(e_1)=\w(e_2)$, which puts us in the second case in Proposition \ref{generator}.

\end{enumerate}

\end{proof}

\begin{rmk}
A direct consequence of our discussions on type $\aA$ forms is that, any square Lagrangian Dehn twist is isotopic to identity for these forms, because $Symp_h$ is connected. This fact has implications on the quantum cohomology of the given form on
$X=\CC P^2  \# 5{\overline {\CC P^2}}$.  For example, together with Corollary 2.8 in \cite{Sei08}, we know that $QH_*(X)/I_L$ is Frobenius for any Lagrangian $L$ for a given type $\aA$ form, where $I_L$ is the ideal of $QH_*(X)$ generated by the Lagrangian $L$.
\end{rmk}

\section{ Braiding for \texorpdfstring{ $\pi_0Symp_h(X,\w)$}{Symp} of type \texorpdfstring{ $\DD_4$ forms}{D4}}

In this section we focus on the remaining symplectic forms whose Lagrangian system $\Gamma_L $ is $\DD_4$. Without loss of generality, we can assume $\w$ is reduced. Then these are $\w \in MA$ in table \ref{5form}, where
\begin{equation}\label{wd}
c_1>c_2=c_3=c_4=c_5, \quad c_1+c_2+c_3=1.
\end{equation}
Note that all such forms are balanced.  By Lemma \ref{Cremona}, they are $\RR P^2$ packing symplectic forms.  Therefore, we always have the sequence \eqref{e:ES5pt} by Proposition \ref{Cremona}
\begin{equation}\label{imphi}
  1\rightarrow \pi_1(Symp_h(X,\w))\rightarrow \pi_1(\msC_0) \xrightarrow[]{\phi}  \pi_0(\Diff^{+}(S^2,5)) \xrightarrow{\psi}  \pi_0(Symp_h(X,\w))\rightarrow 1.
\end{equation}

Our goal is to examine \eqref{imphi} and prove the following:

\begin{thm}\label{p:8p4}
  Let $X=\CC P^2\# 5\overline{\CC P^2}$ with a reduced symplectic form $\w$ on $MA$, $\pi_0(Symp_h(X,\w))$ is $\pi_0(\Diff^+(S^2,4))=PB_4(S^2)/\ZZ_2$.  Moreover, the $\phi$-map in the sequence \eqref{e:ES5pt} has $Im(\phi)=\pi_1(S^2-\{\hbox{4 points}\})$.  Abstractly, $\pi_1(S^2-\{\hbox{4 points}\})\cong\FF_3$.

\end{thm}

From the Hopfian property of braid groups, it suffices to obtain surjections between $\pi_0(Symp_h(X,\w))$ and $PB_4(S^2)/\ZZ_2$ in both directions.  The following direction is easy.

\begin{lma}\label{gh}
For a given form $\w \in MA,$ then $\pi_0(Symp_h(X,\w))$ is a quotient of
$PB_4(S^2)/\ZZ_2$, i.e. $PB_4(S^2)/\ZZ_2 \twoheadrightarrow \pi_0(Symp_h(X,\w)).$
\end{lma}

 \begin{proof}


   Since $\w$  is balanced, by Lemma \ref{Cremona}, it is a $\RP^2$-packing form.   The isotopy constructed in Lemma \ref{l:BSproperties}, along with Remark \ref{rem:equalsize}, yields the triviality of $\psi$-image of $\{A_{12}, A_{13}, A_{14}, A_{15}\}$.   Therefore, the subgroup generated by these four elements, which is isomorphic to $\pi_1(S^2-\hbox{4 points})$, is a subgroup in the image of $\phi$ in sequence \eqref{imphi}.


From \cite{Bir69}, we have the short exact sequence of the forgetting one strand map:
\begin{equation}\label{forgetstrand}
0\rightarrow \pi_1(S^2-\hbox{4 points})\rightarrow PB_5(S^2)/\ZZ_2\rightarrow  PB_4(S^2)/\ZZ_2 \rightarrow 0.
\end{equation}

Therefore, one sees that $\pi_0(Symp_h(X,\w))$ is a quotient of $PB_4(S^2)/\ZZ_2$, and there is a surjective homomorphism $\psi: PB_4(S^2)/\ZZ_2 \to \pi_0(Symp_h(X,\w))$.

\end{proof}

The opposite direction of the surjective map is much more involved and will occupy most of the section.  The key ingredient of  section \ref{s:freeact} and section \ref{s:surj}. is to establish the following commutative diagram for $\w \in MA$ with rational periods:




\begin{equation}\label{e:key}
     \begin{tikzpicture}
\node at (5,0) (a) {$\mQ_5$};
\node[right =2.5cm of a]  (b){$(\mJ^c_{\w}-\mX^c_4)/Symp_h$};
\node[right =2.5cm of b]  (c){$\mB_4$};
\draw[->,>=stealth] (a) --node[above]{$\alpha$} (b);
\draw[->,>=stealth] (b) --node[above]{$\beta$} (c);
\draw[->,>=stealth] (c) edge[bend right=30]node[below]{$\gamma$}(a);
\end{tikzpicture}
\end{equation}

Note that through out this section, we'll simply write $Symp_h$ for $Symp_h(X_5, \w_{MA})$ since no confusion could occur. Here,

\begin{itemize}
  \item $\mQ_5$ is the moduli space of certain configurations of 5 points on $\CP^2$ given in Definition \ref{q5} below,
  \item $\mB_4=\text{Conf}_4^{ord}(\CC P^1)/PGL_2(\CC)$,
  \item $\mJ^c_{\w}$ is the space of $\w$-compatible almost complex structure and $\mX^c_4$ is the codimension less than 4 part of $\mJ^c_{\w}$ in the prime decomposition of Lemma \ref{prime}.
\end{itemize}

  We will define each morphism $\alpha, \beta$ and $\gamma$ and show that the diagram commutes.  Then Lemma \ref{beta} shows the composition of $\beta\circ \alpha \circ \gamma$ is the identity on $\mB_4$.  This implies the map $\beta$
induces the desired surjective map $\beta^*: \pi_0(Symp_h)$ to $\pi_1(\mB_4)=  PB_4(S^2)/\ZZ_2$.

\begin{rmk}\label{rem:compfree}
    We remark on the almost complex structures.  In section \ref{s:freeact} and section \ref{s:surj} we'll use the space of $\w$-\emph{compatible} almost complex structure $\mJ^c_{\w}$ and corresponding subsets $\mX_{2n}^c$'s, instead of \textit{tamed} almost complex structures.  This way, the framework of \cite{FS88} directly applies to compatible $J$ and we have a proper action for free.  Indeed, the proof of \cite{FS88} carries over to tamed almost complex structures with no essential difficulties, but we choose to be more pedagogical.  Discussions in Section \ref{s:freeact} and \ref{s:surj} will refer to results we proved in earlier sections, but all of them are about properties of $J$-holomorphic curves therefore works equally well in tamed or compatible settings.
\end{rmk}



\subsection{The proper free action of \texorpdfstring{ $Symp_h$}{Symp} on \texorpdfstring{$\mJ^c_{\w}-\mX^c_4$}{Jw-X4}, and its associated fibration}\label{s:freeact}

By Lemma \ref{primeaction},$Symp_h$ naturally acts on $\mJ^c_{\w}-\mX^c_4$.  In this section, we will use the framework in \cite{FS88} to study this group action and establish the associated fibration in Lemma \ref{fiblocsection}.  Note that Theorem 3.3 in \cite{FS88} assures that  this group action is proper.

We start our proof of the freeness of this action by analyzing the configuration of $J$-holomorphic curves.

\begin{figure}[ht]
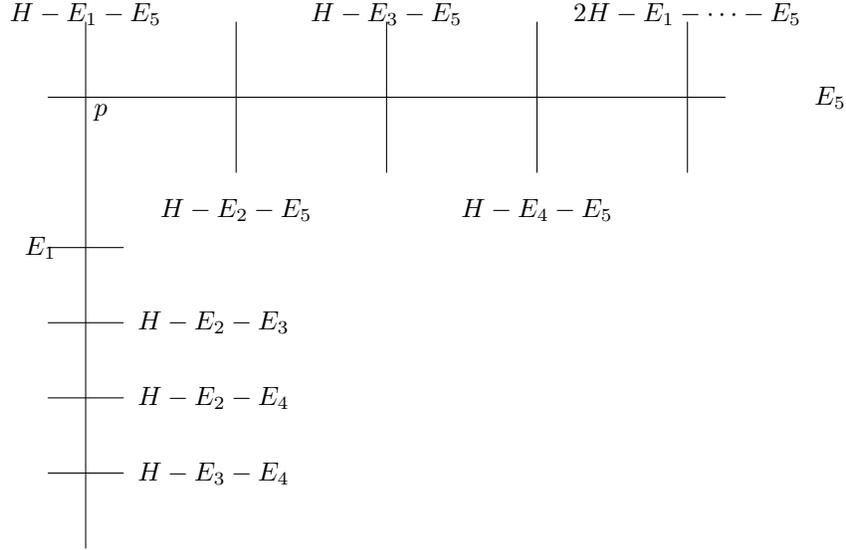

  \centering
\[
\xy
(-5, -10)*{};(85, -10)* {}**\dir{-};
(99, -10)* {E_5};
(0, 0)*{}; (0, -70)*{}**\dir{-};
(0, 1)*{H-E_1-E_5};
(20, 0)*{}; (20, -20)*{}**\dir{-};
(20, -25)*{H-E_2-E_5};
(40, 0)*{}; (40, -20)*{}**\dir{-};
(40, 1)*{H-E_3-E_5};
(60, 0)*{}; (60, -20)*{}**\dir{-};
(60, -25)*{H-E_4-E_5};
(80, 0)*{}; (80, -20)*{}**\dir[red, ultra thick, domain=0:6]{-};
(80, 1)*{2H-E_1-\cdots -E_5};
(-5, -30)*{};(5, -30)* {}**\dir{-};
(-6, -30)* {E_1};
(-5, -40)*{};(5, -40)* {}**\dir{-};
(17, -40)* {H-E_2-E_3};
(-5, -50)*{};(5, -50)* {}**\dir{-};
(17, -50)* {H-E_2-E_4};
(-5, -60)*{};(5, -60)* {}**\dir{-};
(17, -60)* {H-E_3-E_4};
(2, -12)*{p};
\endxy
\]
\caption{Homology classes of configuration with two minimal exceptional classes for $\w \in MA.$}
  \label{cfigd}
\end{figure}

 \begin{lma}\label{inter1}
 Given a reduced form $\w \in MA$ and take a configuration of homology classes as in Figure \ref{cfigd} (each line represents a possibly singular curve with the given homology class).
 If $J\in \mJ^c_{\w}-\mX^c_4$, then the $J$-holomorphic representative of $E_5$ and $H-E_1-E_5$ is smoothly embedded.

 Moreover, each $J-$holomorphic representative of the vertical classes $H-E_i-E_5$ for $i=1,\cdots,4$ and $2H-E_1-\cdots-E_5$ (which are not necessarily smooth), intersect the embedded curve in class $E_5$ exactly once at a single point. And the same holds for $H-E_1-E_5$ intersecting the horizontal classes.
 \end{lma}

 \begin{proof}
Let's recall from Theorem \ref{t:Pin} (Lemma 2.1 in \cite{Pin08}) that an exceptional class with minimal area always has a pseudo-holomorphic embedded representative. This applies to $E_5$ and $H-E_1-E_5$, therefore, the corresponding curves are always embedded.

For the second statement, we'll just prove for $E_5$, and the proof for  $H-E_1-E_5$ is very similar.
We use  $H, E_1,\cdots, E_5,$ as the basis of $H_2(X,\ZZ).$  Denote $e_i=H-E_i-E_5$, $i=2,3,4$ and $e_5= 2H-\sum^5_{i=1} E_i$.  From the homological pairing and positivity of intersections, what we need to show is that $e_i$ does not have a stable representative that contains a $E_5$-component or its multiple covers.

Let $A$ be one of $e_i$.  Assume $A= \sum_k A_k$ is the decomposition of homology classes given by the stable representative.  From our assumption on the almost complex structure, $\{A_k\}$ must consist of rational curves with self-intersection at least $(-2)$.  From Proposition \ref{sphere}, the underlying simple curves in the decomposition could have 4 type of classes: $ B,kF, D_j \in \mS^{-1}, G_k \in S^{-2}.$ Performing a base change \eqref{BH}, the $\{A_k\}$ could have $ H-E_2,k(H-E_1), D_j \in \mS^{-1}, G_k \in S^{-2}$.

Again by Proposition \ref{sphere}, the only class with negative $H$-coefficients are of the form $(k+1)E_1-kH -\sum_j E_j$, while $k\ge1$.  Since these curves have squares less than $-2$, we conclude that each $A_k$  has a non-negative coefficient on $H$. Now we can analyze all possible decomposition $\{A_k\}$ in the configuration in Figure \ref{cfigd} as follows.

\begin{itemize}
\item For $A=H-E_i-E_5,$  there should be exactly one simple component that has the form $A_1=H-\sum_{i_m} E_{i_m}$, $m\le 3$, and other components take the form of $E_1- E_j$, or $E_j$, or their multiple covers from the consideration of $H$-coefficients. Note that $m=2$ cannot hold.  Otherwise, the sum of coefficients of all $E_i$'s across all the components in the decomposition $\ge(-1)$, because $E_1-E_j$ components contribute zero and $E_j$ components contribute positively, a contradiction.  If $m=3$, by comparing the $\w$-area, we have $E_{i_m}\neq E_1$.  To have a component of $kE_5$ in the decomposition for $k\ge1$, one must have at least $k$ copies of $E_1-E_5$ in the rest of components other than $A_1$, counting multiplicities.  However, the total homology class of such a configuration will have positive coefficient in $E_1$, again a contradiction.

\item For $A=2H-\sum^5_{j=1} E_j$, if the decomposition has an $kE_5$-component for $k\ge1$, the sum of $E_i$-coefficients in the rest of the components must be at most $-6$.  Again the simple components that could possibly take positive $H$-coefficients are of the form $H-\sum_{i_m} E_{i_m}$, $m\le3$, and there can be at most two of them, counting multiplicities.  Since $E_1-E_j$ and $E_j$ contributes non-negatively to total $E_i$-coefficients, both $H$-components have to have $m=3$.  But the area consideration prevents any $i_m=1$, hence such a configuration will always have a non-negative total coefficient in $E_1$, which is a contradiction.

    \end{itemize}
\end{proof}



\begin{lma}\label{4free}
For a given form $\w \in MA,$ the action of $Symp_h$ on $\mJ^c_{\w}-\mX^c_4$ is free. And hence $\pi_i(Symp_h)=\pi_{i+1}(\mJ^c_{\w}-\mX^c_4)/Symp_h$ for $i=0,1$.
\end{lma}

\begin{proof}
 For freeness, we will analyze the unique (stable) rational curves of homology classes in Figure \ref{cfigd}.

Suppose $\varphi\in Symp_h(X,\w)$ fixes $J$.  Denote $J(A)$ the $J$-holomorphic representative of the class $A$.  Suppose $J \in \mX^c_0$, all $J(A)$ are irreducible and both $J(E_5)$ and $J(H-E_1-E_5)$ have five distinct geometric intersections with other curves in the picture.  Since $\varphi$ fixes $J$, all these intersections must be fixed points, hence $J(E_5)$ and $J(H-E_1-E_5)$ are pointwise fixed.

Consider $p=J(E_5) \cap J(H-E_1-E_5)$ be the unique intersection, under the metric paired from $\w$ and $J$, the exponential map at $p$ shows that every point is fixed under this action, and hence the action $i$ itself is identity in $Symp_h(X,\w)$. This means the action of $Symp_h(X,\w)$ on $\mJ^c_{\w}$ is free.

In general, pick any $J \in \mJ^c_{\w}-\mX^c_4,$ the bubbles could occur.  Both $J(E_5)$ and $J(H-E_1-E_5)$ have simple representatives because they have minimal area and cannot bubble.  Lemma \ref{inter1} further shows $J(E_5)$ or $J(H-E_1-E_5)$ must intersect other curves in the configuration at finitely many points, since they cannot underlie a bubble.
 While the geometric intersections between $J(E_5)$ and two different vertical curves in Figure \ref{cfigd} can collide (except for $J(E_5)\cap J(H-E_1-E_5)$), we argue there must be at least three distinct intersections, and the same assertion also holds for $J(H-E_1-E_5)$.  All possibilities of numbers of distinct geometric intersections are listed in Table \ref{gip} using the labeling set $\mC\subset S^{\leq -2}$ for the prime submanifold $\mJ^c_{\mC}$ where $J$ belongs to. Indeed, $\mC$ is either empty or has a single square $-2$ class that admits a $J$-representative.

 We spell out one of the entries in the table and the rest can be checked similarly with ease.  If $J(H-E_3-E_4-E_5)$ exists, $J(H-E_3-E_5)$, $J(H-E_4-E_5)$ and $J(2H-E_1-\cdots-E_5)$ must contain it as a component.  We claim the resulting curve configuration must consist of an embedded copy of $J(H-E_3-E_4-E_5)$ with another exceptional curve.  For example, $\w((H-E_3-E_5)-(H-E_3-E_4-E_5))=\w(E_4)$ is the minimal area of all exceptional curves.  Since the stable configuration must contain at least one exceptional curve from Theorem \ref{t:Pin}, the configuration must consists of $J(E_4)$ and $J(H-E_3-E_4-E_5)$, and similarly for $J(H-E_4-E_5)$ and $J(2H-E_1-\cdots-E_5)$.  For $J(H-E_1-E_5)$, if it bubbles, then it has to contain a component which is $J(H-E_3-E_4-E_5)$.  Since the $H$-coefficient has to be non-negative for all components (otherwise, $J$ falls into a strata that allows curve more negative than $(-2)$ by \ref{sphere}), this $J(H-E_3-E_4-E_5)$-component is simple.  Therefore, the rest of the components will have a total class of $E_3+E_4-E_1$.  Again from \ref{sphere}, we see that there must be at least a component of class $E_3-E_1$ or $E_4-E_1$, contradicting that fact that $J$ only allows a single $(-2)$-sphere class.  A similar argument shows $J(H-E_2-E_5)$ must be embedded.  Therefore, there are three geometric intersections between $E_5$ and stable representatives of the vertical classes.

 The rest of the argument follows exactly that of $J\in\mX^c_0$, since a bi-holomorphism with three fixed points on a rational curve must be the identity.  The lemma hence follows.

\begin{table}[ht]
\begin{center}
\begin{tabular}{||c| c | c   ||}
\hline\hline
element in $\mC$& $\#$ of g.i.p. on $E_5$   & $\#$ of g.i.p. on $H-E_1-E_5$  \\ [0.5ex]
\hline\hline
$H-E_2-E_3 -E_5$& 3   & 5 \\
\hline
$H-E_2-E_4 -E_5 $& 3  & 5 \\
\hline
$H-E_3-E_4 -E_5 $& 3 & 5 \\
\hline
$E_1-E_5$& 3  & 5 \\
\hline
$H-E_2-E_3 -E_4 $& 5   & 3\\
\hline
$E_1-E_2$ &  5   & 3 \\
\hline
$E_1-E_3$&  5   & 3  \\
\hline
$E_1-E_4$&  5   & 3 \\
\hline\hline
\end{tabular}

\caption{number of geometric intersection points (g.i.p.) for  $J \in \mJ^c_{\mC}$.}
\label{gip}
\end{center}
\end{table}
\end{proof}


 \begin{lma}
 $\mX^c_4$ is closed in $\mJ^c_{\w}$. Consequently, $(\mJ^c_{\w}-\mX^c_4)$ is a Fr\'echet manifold.
 \end{lma}

\begin{proof}

This follows from Gromov convergence.  Suppose $J_i\in\mX^c_4$ is a convergent sequence to $J\in(\mJ^c-\mX^c_4)$, then one of the following two cases hold:

\begin{enumerate}
  \item[(i)] Infinitely many $J_i$'s admits a $J_i$-rational curve with classes $B_i^2<-2$.
  \item[(ii)] Infinitely many $J_i$ admits at least two $J_i$-rational curves with $B_{i,\delta}^2=-2$ with different homology classes, $\delta=0,1$.
\end{enumerate}

In both cases, since we have only finitely many possible homology classes $B_i$ from Lemma \ref{sphere} and area constraints (say, $\w(B-kF)$ can only be positive for finitely many $k$), we can extract a subsequence from $\{B_i\}$ or $\{B_{i,\delta}\}$.  That is, without loss of generality, we may assume $B_i=B_j$ in the first case, and $B_{i,\delta}=B_{j,\delta}$ for $\delta=0,1$ and all $i,j$ in the second case.

In case (i), the limit of $J_i(B_i)$ must be a stable curve consisting of components with $J(D_{ij})^2\ge-2$ from the assumption on $J$.  But all these components have $c_1(D_{ij})\ge0$ while $c_1(B_i)<0$, a contradiction.

In case (ii), each of $J(B_{i,\delta})$ must converge to a stable curve.  Since $c_1(B_{i,\delta})=0$, and all except one spherical class of $J$, say $D$, has $c_1>0$.  This means the limit of $B_{i,\delta}$ can contain only a multiple cover of $D$, but they must be different.  But $D^2=-2$, so $B_{i,\delta}^=-2$ cannot hold for both $\delta$, again a contradiction.


 \end{proof}

 Based on the slice theorem (Theorem 5.6, Corollary 5.3 in \cite{FS88}), we can prove the following fibration lemma from a standard argument. In Appendix \ref{s:proper}, we'll recall the theorems in \cite{FS88} and give a detailed proof.

\begin{lma}\label{fiblocsection}
The orbit space $(\mJ^c_{\w}-\mJ^c_4)/Symp_h$ is Hausdorff and locally modelled on Fr\'echet spaces. The orbit projection of the free proper action $Symp_h$ on $(\mJ^c_{\w}-\mJ^c_4)$ is a fibration with fiber $Symp_h$.
\end{lma}

\begin{lma}\label{JJ41}
 We have an isomorphism $\pi_1[(\mJ^c-\mX^c_4)/Symp_h]\cong\pi_0(Symp_h).$
\end{lma}
\begin{proof}

From the long exact sequence of the action-orbit fibration
 $$ Symp_h \to (\mJ^c_{\w}-\mJ^c_4)\to (\mJ^c_{\w}-\mJ^c_4)/Symp_h,$$ we have
 $$  \pi_1(Symp_h ) \to   \pi_1(\mJ^c_{\w}-\mJ^c_4) \to  \pi_1((\mJ^c_{\w}-\mJ^c_4)/Symp_h ) \to \pi_0(Symp_h ) \to 0, $$

while $\pi_1(\mJ^c_{\w}-\mJ^c_4)\cong1$ from considering the codimension.
\end{proof}

\subsection{Surjectivity of  \texorpdfstring{ $\pi_0(Symp_h(X,\w_{\DD_4})) \to \pi_0(\Diff^+(S^2,4))$} {PB4} for a rational \texorpdfstring{$\w_{\DD_4}$ }{wD4}}\label{s:surj}

This section is the technical heart of the proof of Theorem \ref{p:8p4}.  We will define the remaining ingredients of equation \eqref{e:key}, including $\mQ_5$ and the three maps $\alpha, \beta, \gamma$.  We also verify the commutativity of the maps and the surjectivity of $\beta^*$.

\subsubsection{The \texorpdfstring{ $\alpha$}{alpha}-map} \label{subsub:the_alpha}


We first address the definition of $\mQ_5$ and the definition of $\alpha$ in \eqref{e:key}.  Consider a universal family of del Pezzo surfaces of degree $4$.  In explicit terms, this is a family $\mY\to \mU:=(\CP^2)^5\backslash \Delta$, where $\Delta$ is an ``extended big diagonal'' where two of the components of $(z_1,\cdots,z_5)\in (\CP^2)^5$ coincide, or when three components lie on the same line.  Each point $u\in\mU$ corresponds to a configuration of five points, the fiber $\mY_u$ is the del Pezzo surface by blowing up $z_1,\cdots, z_5$ on $\CP^2$.  The construction of $\mY$ is straightforward: consider a trivial family over $\mU$ with fiber equal $\CP^2$, there are five canonical sections $s_1,\cdots,s_5$ given by the position of the five components, and $\mY$ is the blow-up of these sections.

Note that it is crucial for the rest of our discussions that the points are ordered, so that we have a well-defined basis of $H_2(\mY_u,\ZZ)$ over each $u$.

$\mQ_5$ is a partial compactification of $\mU$ constructed as follows.  First of all, consider $(\CP^2)^5$ blown up on $\Delta$, giving an exceptional divisor $\ov\Delta$.  This creates many strata in the discriminant locus but we will discard all strata of complex codimension $\ge2$, and the remaining open subset will be denoted as $\ov\mU$.  Again on the trivial family of $\CP^2$ over $\ov\mU$, sections $s_i$ extends to $\ov s_i$ for all $i$, then we choose to first blow-up $\ov s_1$, then the proper transform of the rest of $\ov s_i$.  The resulting family $\ov\mY\to \ov\mU$ can also be described by the fibers.  For example, if $\ov u\in \ov\mU$ is on the discriminant locus where $z_1=z_2$, $\ov u$ also specifies tangent direction where the two points come together.  Then the fiber is a rational surface of five blow-ups obtained by first blowing up a point $z_1$ and $z_3,z_4,z_5$, then blow-up $z_2$ on the exceptional divisor $E_1$, the position of $z_2$ is determined by the tangent direction that was remembered earlier.

Other rational surfaces over the exceptional divisor $\ov\Delta$ are similar.  Besides the collisions of other pairs of points, they also include blowing up three points on the same line in $\CP^2$, etc, which will not give del Pezzo surfaces.  We may label irreducible components of $\ov\Delta$ by negative rational curves.  For example, when $\ov u\in\ov\Delta$ is a point where $z_1=z_2$, then the rational surface on the fiber admits a rational curve of class $E_1-E_2$; while $\ov u$ is a point where $z_1,z_2,z_3$ are on the same line, then the fiber admits a rational curve of class $H-E_1-E_2-E_3$, etc.  An irreducible component of $\ov\Delta$ that admits a unique $(-2)$ rational curve of class $S$ will be denoted as $[S]\subset \ov\Delta$.

In the remainder of our discussions, we will consider $\mU\cup\bigcup_{i=1}^5[E_1-E_i]\cup\bigcup_{2\le i,j,k\le5}[H-E_i-E_j-E_k])$.  We claim that the diagonal action of $PGL(3,\CC)$ acts on this set freely.  The action can be explicitly described as follows.  For a del Pezzo surface $\mY_u$ in $\mU$ or $[H-E_i-E_j-E_k]$, the blow-down of $E_1$ through $E_5$ gives a quintuple of points $\{p_1,\cdots,p_5\}$ where an element $g\in PGL(3,\CC)$ acts on, and $g\mY_u$ is the blow-up of $\{g(p_1),\cdots,g(p_5)\}$.
For $\bar u\in[E_1-E_i]$, the blow-downs yields a quardruple $\{p_1,p_j,p_k,p_l\}$, where $\{i,j,k,l\}=\{2,3,4,5\}$, along with a tangent direction marked by $p_i$.  Note that no triples of $\{p_1,p_j,p_k,p_l\}$ lie on the same line.   $g$ again acts on these four points as well as the tangent direction at $p_1$, then the blow-up is performed first at $p_1,p_i,p_j,p_k$, then the tangent direction specified by $p_i$.

Therefore, the free action follows from the fact that $PGL(3,\CC)$ acts 4-transitively on $\CC P^2$ and there is no stabilizer for points in $\ov\mU$.  The whole construction can be regarded as a partial compactification of the moduli space of del Pezzo surfaces of degree $4$.

\begin{dfn}\label{q5}
We define
\begin{equation}\label{e:q5}
     \mQ_5:=(\mU\cup\bigcup_{i=1}^5[E_1-E_i]\cup\bigcup_{2\le i,j,k\le5}[H-E_i-E_j-E_k])/PGL(3,\CC)\subset \ov \mU/PGL(3,\CC).
\end{equation}

\end{dfn}

\begin{lma}\label{alpha}
For a rational point $\w \in MA,$
there exists a well defined continuous map
$$\alpha: \mQ_5  \quad \mapsto (\mJ^c_{\w}-\mX^c_4)/Symp_h.$$
\end{lma}

\begin{proof}

Up to a rescaling, we can write $PD([l\w]) = aH - b'E_1-bE_2-bE_3-bE_4-bE_5 $ with $a, b'>b \in \ZZ^{>0},$ and $a=b'+2b.$

  Consider the divisor $D$ of $\mY$, which is a linear combination of the universal line class (where over each fiber has class $H$) and the canonical exceptional divisors by the blow-ups of $s_i$, so that over each fiber $\mY_q$ we have $[D_q]=PD([l\w]) = aH - b'E_1-bE_2-bE_3-bE_4-bE_5 $ for $q\in\mQ_5$.  Clearly, $D_q\cdot C_q >0,$ where $C_q$ is any curve in $\mY_q$, and we also have $D_q\cdot D_q>0.$  Hence $D$ is a relative ample divisor which induces a family of embeddings of $\mY_q$ into $\CP^N$.

Equipping $\PP H^0(X;D)$ with a fiberwise Fubini-Study form, one has a fiberwise symplectic structure on $\mY_q$, diffeomorphic through some $\iota_q$ to $\w$ from \cite{McD96}.  For each fiber $\mY_q$, the embedding pulls back the complex structure $J_0$, and pushes to a $J_q$ through $\iota_q$, which gives an integrable almost complex structure $\iota_q(J_q) \in J_{\w}.$  Two different choices of $\iota_q$ (when monodromies are involved) differ by a symplectomorphism in $Symp_h(X,\w)$, and all these almost complex structures does not admit curves of self-intersection $<-2$ or two different $(-2)$ rational curve classes.  Hence this construction yields a well-defined continuous map $$\alpha: \mQ_5 \quad \mapsto (\mJ^c_{\w}-\mX^c_4)/Symp_h.$$
\end{proof}

By Lemma \ref{4free}, $\pi_1 ((\mJ^c_{\w}-\mX^c_4)/Symp_h) = \pi_0(Symp_h)$, and hence $\alpha$ gives the map:
\begin{equation}\label{delta}
\delta: \pi_1(\mQ_5 )   \quad \mapsto \pi_0(Symp_h).
\end{equation}
  We remark that $\delta$ is the monodromy map of the family $\ov\mY$, and those around the meridian of $[E_i-E_j]$ for $2\le i,j\le4$ are precisely the ball-swapping maps, but this observation will not be used in the rest of the proof.


For the remainder of \eqref{e:key}, we consider the  configuration in Figure \ref{conf} of homology classes

\begin{figure}[ht]
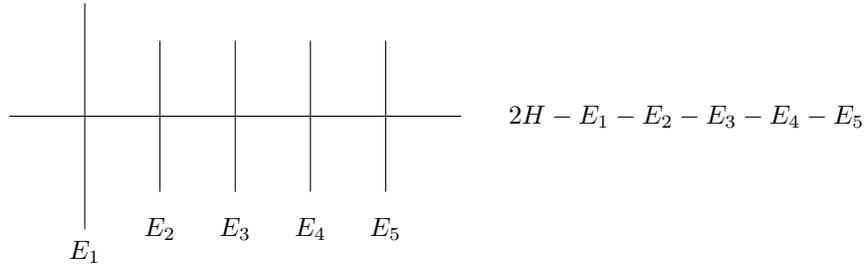

\centering
\[
\xy
(0, -10)*{};(60, -10)* {}**\dir{-};
(90, -10)* {2H-E_1-E_2-E_3-E_4-E_5};
(10, 5)*{}; (10, -25)*{}**\dir{-};
(10, -28)*{E_1};
(20, 0)*{}; (20, -20)*{}**\dir{-};
(20, -25)*{E_2};
(30, 0)*{}; (30, -20)*{}**\dir{-};
(30, -25)*{E_3};
(40, 0)*{}; (40, -20)*{}**\dir{-};
(40, -25)*{E_4};
(50, 0)*{}; (50, -20)*{}**\dir[red, ultra thick, domain=0:6]{-};
(50, -25)*{E_5};
\endxy
\]
  \caption{Configuration of exceptional classes for $\w \in MA$/ Embedded components for $J\in \mJ^c_{open}$.}
\label{conf}
\end{figure}

\subsubsection{The \texorpdfstring{$\beta$}{beta}-map} 
\label{sub:the_}

Recall that for an $[\w]\in MA,$ $E_2, E_3, E_4, E_5$ have the minimal area among exceptional classes.  Therefore, for any almost complex structure $J\in\mJ^c_\w$, the classes $E_2, E_3, E_4, E_5$ always have pseudo-holomorphic simple representatives by Theorem \ref{t:Pin}.

 For a $J \in \mJ^c_{open}=\mJ^c_{\w}-\mX^c_2,$ each class in Figure \ref{conf} has an embedded representative.  The rational curve $J(2H-E_1-\cdots-E_5)$ intersects $J(E_i)$ at a point $p_i$, $2\le i\le5$, which gives a set of four points $\{p_2,p_3,p_4,p_5\}\subset J(2H-E_1-\cdots-E_5)=\CP^1$.  We define this configuration to be $\beta(J)\in\mB_4$.

 The codimension-$2$ strata can be divided into two kinds.  We denote $\mJ^c_{2H-E_1}$ to be the union of  $\mJ^c_{\mC}$  where $\mC$ is
 $\{E_1-E_5$\},
$\{E_1-E_2\}$,
$\{E_1-E_3\}$, or
$\{E_1-E_4$\}.  Take $\mJ^c_{\{E_1-E_2\}}$ as an example, the stable representatives of classes in Figure \ref{conf} is

\begin{figure}[ht]
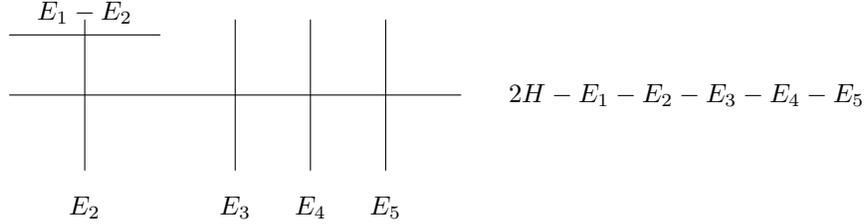

\centering
 \[
\xy
(0, -10)*{};(60, -10)* {}**\dir{-};
(90, -10)* {2H-E_1-E_2-E_3-E_4-E_5};
(10, 0)*{}; (10, -20)*{}**\dir{-};
(10, -25)*{E_2};
(0, -2)*{}; (20, -2)*{}**\dir{-};
(10, 1)*{E_1-E_2};
(30, 0)*{}; (30, -20)*{}**\dir{-};
(30, -25)*{E_3};
(40, 0)*{}; (40, -20)*{}**\dir{-};
(40, -25)*{E_4};
(50, 0)*{}; (50, -20)*{}**\dir[red, ultra thick, domain=0:6]{-};
(50, -25)*{E_5};
\endxy
\]

  \caption{ Embedded components for $J\in \mJ^c_{2H-E_1}$.}
\label{conf2H}
\end{figure}

In this case when $J\in  \mJ^c_{2H-E_1},$ we again take the curve $J(2H-E_1-\cdots-E_5)$, along with its intersection points $p_i:=J(2H-E_1-\cdots-E_5)\cap J(E_i)$ for $i\ge2$, which again yields an element $\beta(J)\in\mB_4$.

The rest of codimension-$2$ strata, denoted $\mJ^c_{H}$, are the union of  $\mJ^c_{\mC}$  where $\mC$ is
$\{H-E_2-E_3 -E_5 \}$,
$\{H-E_2-E_4 -E_5\}$,
$\{H-E_3-E_4 -E_5\}$, or
$\{H-E_2-E_3 -E_4 \}$. Take $\mJ^c_{\{H-E_2-E_3 -E_4 \}}$ as an example, the stable representatives of classes in Figure \ref{conf} is

\begin{figure}[ht]
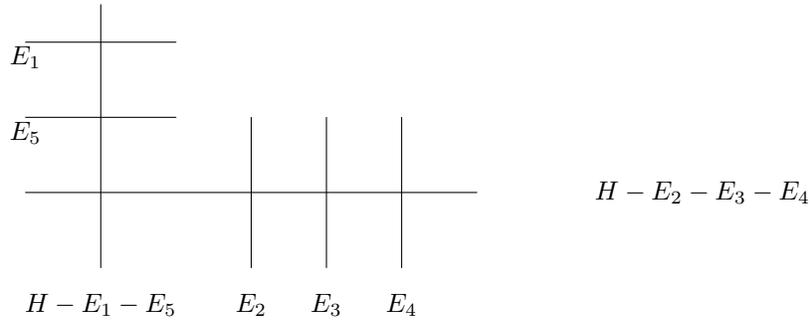

\centering
 \[
\xy
(0, -10)*{};(60, -10)* {}**\dir{-};
(90, -10)* {H-E_2-E_3-E_4};
(10, 15)*{}; (10, -20)*{}**\dir{-};
(10, -25)*{ H-E_1-E_5};
(0, 0)*{}; (20, 0)*{}**\dir{-};
(0, -2)*{E_5};
(0, 10)*{}; (20, 10)*{}**\dir{-};
(0, 8)*{E_1};
(30, 0)*{}; (30, -20)*{}**\dir{-};
(30, -25)*{E_2};
(40, 0)*{}; (40, -20)*{}**\dir{-};
(40, -25)*{E_3};
(50, 0)*{}; (50, -20)*{}**\dir[red, ultra thick, domain=0:6]{-};
(50, -25)*{E_4};
\endxy
\]
  \caption{ Embedded components for $J\in \mJ^c_{H}$.}
\label{confH}
\end{figure}

If $J\in  \mJ^c_{H},$ we consider the unique $(-2)$ curve $J(H-E_i-E_j-E_k)$, and the last stable curve is $H-E_1-E_l$, where $\{i,j,k,l\}=\{2,3,4,5\}$.  Then we define $p_l:=J(H-E_i-E_j-E_k)\cap J(H-E_1-E_l)$, and $p_I:=J(H-E_i-E_j-E_k)\cap J(E_I)$ for $I\neq l$.  Therefore, $\{p_i,p_j,p_k,p_l\}\subset J(H-E_i-E_j-E_k)$ forms an element in $\mB_4$, which is defined to be $\beta(J)$.

An alternative point of view will be useful.  Take the ``base curve'' $J(2H-E_1\cdots-E_5)$ for $\mJ^c_{open}$ and $\mJ^c_{2H-E_1}$, and $J(H-E_i-E_j-E_k)$ for $\mJ^c_H$.  Consider each of them as the image of $J$-holomorphic map $u$ with four marked points, where $u(0)=p_2, u(1)=p_3, u(\infty)=p_4$, and the fourth marked point $u(z)=p_5$.  Then $\beta(J)$ is precisely the domain of $u$, i.e. $(\CP^1, 0,1,\infty,z)$.

The above definition of $\beta$ on various strata are well-defined: given $\phi\in Symp(X,\w)$, the corresponding curve configuration is pushed forward along with $J$.  Therefore, the four-point configuration on the underlying $\CP^1$ remains in the same conjugacy class.




The continuity of $\beta$ on $\mJ^c_{open}$ and $\mJ^c_{2H-E_1}$ is clear because no bubbling of $E_i$ is involved for $i\ge2$.  For $J\in \mJ^c_H$, we focus on the stratum $\mJ^c_{H-E_2-E_3-E_4}$ without loss of generality.  Recall the reparametrization process in Gromov compactness (cf. Theorem 4.7.1 of \cite{MS04} ): take a sequence of $J^i\in \mJ^c_{open}$ that converges to $J$, the curve $J^i(2H-E_1-\cdots-E_5)$ converges to the union of $J(H-E_2-E_3-E_4)$ and $J(H-E_1-E_5)$.  Consider $J(H-E_2-E_3-E_4)$ as the image of a $J$-holomorphic map $u:\CP^1\to\CP^2\#5\ov\CP^2$ as above, and denote the pre-image of the intersection $u^{-1}(J(H-E_1-E_5)\cap J(H-E_2-E_3-E_4)):=z_0\in\CP^1$.  Consider the corresponding $J^i$-holomorphic maps $u^i$ and fix the parametrization of $J^i(2H-E_1-\cdots-E_5)$ by requiring $u^i(0)=p^i_2, u^i(1)=p^i_3, u^i(\infty)=p^i_4$ as above, where $p^i_k=J^i(E_k)\cap J^i(2H-E_1-\cdots-E_5)$ for $k=2,3,4,5$.  Further, denote the pre-image of $p^i_5$ by $z^i$. Then $z^i$ in the domain of $u_i$ lies in a disk $D_{\epsilon_i}(z_0)$ of radius $\epsilon_i\to 0$, by the bubble connect Theorem (Theorem 4.7.1 of \cite{MS04}).  Therefore, one can see that $z^i$ converges to the node when $J^i\to J$, that is, the intersection between $J(H-E_2-E_3-E_4)$ and $J(H-E_1-E_5)$, as desired.



 This gives the continuous map as stated:
 $$\beta: (\mJ^c_{\w}-\mX^c_4)/Symp_h \to \mB_4.$$

\subsubsection{The \texorpdfstring{$\gamma$}{gamma}-map, and the conclusion of Theorem \ref{p:8p4}} 
\label{sub:the_gamma}



For the section map $\gamma$, we want to construct a rational surface with Euler number eight (or rather its integral complex structure) associated to  $[p_2,p_3,p_4,p_5]\in\mB_4= \Conf_4^{ord}(\CC P^1)/PGL_2(\CC)$.

Consider the inclusion of $\mB_4\hookrightarrow\ov\mB_5=\ov{\Conf^{ord}_5(\CP^1)/PGL_2(C)}$ by sending $(z,[1,0],[0,1],[1,1])$ to $(z_0:=[2,1],z,[1,0],[0,1],[1,1])$, where the overline denotes a partial compactification such that $z$ is allowed to collide with $z_0$.  Fix a quadric $u:\CP^1\to\CP^2$, we define a map by sending $(z,[1,0],[0,1],[1,1])$ to the blow-up of $(u(z_0),u(z),u([1,0]),u([0,1]),u([1,1]))$ if $z\neq z_0$.  If $z=z_0$, first blow-up $(u(z),u([1,0]),u([0,1]),u([1,1]))$, then the intersection between the exceptional divisor from $u(z)$ and the proper transform of the quadric.




  One sees that this map is well-defined because the reparametrization on the quadric can be lifted to an element in $PGL_3(\CC)$.
  Recall that, $PGL_3(\CC)$ acts transitively on the strata consisting of irreducible curves in the linear system of conics (see \cite[Corollary 3.12]{Kirwan92}).  A direct computation shows that the stabilizer is $PGL_2(\CC)$.\footnote{Readers who prefer to avoid computations may find the following argument based on Gromov-Witten theory appealing.  Given any element $g\in PGL_2(\CC)$ acting on the domain of the quadric $u$. Since $PGL_3(\CC)$ is $4$-transitive, one may find a complex curve $S\subset PGL_3(\CC)$, so that for $g_1\in C$, we have $g_1(p_i)=u\circ g(p_i)$ and $g_1(p_4)\in u(\CP^1)$.  As we vary $g_1$ by moving $g_1(p_4)$, the evaluation $g_1(p_5)$ swipes out a holomorphic cycle which bounds to intersect $u(\CP^1)$, where one obtains a desired lift $\wt g$.}  In other words, the action of any $g\in PGL_2(\CC)$ on a conic can be extended to an element $\wt g\in PGL_3(\CC)$.

With this understood, we are ready to prove the following key result of the section.

\begin{prp}\label{beta}
For any rational symplectic form $\w \in MA,$  the composition of $ \beta\circ \alpha \circ \gamma $ is an identity map on $\mB_4.$

\begin{equation}
\begin{tikzpicture}
\node at (9,0) (a) {$\mQ_5$};
\node[right =2.5cm of a]  (b){$(\mJ^c_{\w}-\mX^c_4)/Symp_h$};
\node[right =2.5cm of b]  (c){$\mB_4$};
\draw[->,>=stealth] (a) --node[above]{$\alpha$} (b);
\draw[->,>=stealth] (b) --node[above]{$\beta$} (c);
\draw[->,>=stealth] (c) edge[bend right=30]node[below]{$\gamma$}(a);
\end{tikzpicture}
\end{equation}

Consequently,  the map $\pi_0(Symp_h(X))\cong\pi_1((\mJ^c_{\w}-\mX^c_4)/Symp_h(X)) \twoheadrightarrow PB_4(S^2)/\ZZ_2\cong\pi_1(\mB_4)$ is surjective for a rational symplectic form $\w \in MA$.

\end{prp}
\begin{proof}

Let $M=(\CC P^2\# 5\overline{\CC P^2}, \w)$ and $\w\in MA.$  Given a configuration $\bar z:=(z,[1,0],[0,1],[1,1])$, $\gamma(\bar z)$ is a rational surface as defined above.  By construction, it has a unique rational curve in the class $2H-E_1-\cdots-E_5$, whose intersection with $E_2,\cdots,E_5$ represents the configuration $\bar z$.


The $\alpha$ map takes the complex structure of the rational surface, use some symplectomorphism from an appropriately chosen K\"ahler form to our choice of $(\CP^2\#5\ov\CP^2,\w)$ to push forward the integrable complex structure.  Note that $\alpha\circ\gamma(\mB_4)$ only includes almost complex structures in $\mJ^c_{2H-E_1}$ but not $\mJ^c_{H}$ by the definition of $\gamma$.  Therefore, it does not change the fact that, the unique rational curve $J(2H-E_1-\cdots-E_5)$ has a configuration given by the intersections with $J(E_2),\cdots,J(E_5)$, and this configuration is the original four-point configuration $\bar z\in\mB_4$.  By definition, $\beta$ takes this almost complex structure $\alpha\circ\gamma(\bar z)$ back to $\bar z$.


 \end{proof}



\subsection{Conclusion for an arbitrary type \texorpdfstring{$\DD_4$}{D4} form and Main Theorem 1}

We can now complete the proof of Theorem \ref{p:8p4}:

\begin{proof}

Denote  $G=PB_4(S^2)/\ZZ_2$,  and $H=\pi_0(Symp_h(X,\w))$ for a rational $\w$ where $N_{\w}= 8$. By Lemma \ref{beta}, there is a surjective homomorphism $\beta^*: H \twoheadrightarrow G$; and by Lemma \ref{gh}, a  surjective homomorphism $\psi: G \twoheadrightarrow H$.  Then by Lemma \ref{hopfcomp}, $G$ and $H$ are isomorphic,
which means $\pi_0(Symp_h(X))=PB_4(S^2)/\ZZ_2.$
This concludes the first part of Theorem \ref{p:8p4} for rational $\w$. Proposition \ref{nonbalstab} extends this result to any point $\w \in MA$, therefore,  we have $\pi_0(Symp_h(X,\w))=PB_4(S^2/\ZZ_2)$ for all $\w\in MA$.

For the second part, it is clear that
$Im(\phi)$ is a normal subgroup of $PB_5(S^2)/\ZZ_2$.
From \eqref{forgetstrand}, we know $\pi_1(S^2-\{\hbox{4 points}\}))$ is a normal subgroup of $PB_5(S^2)/\ZZ_2$. It's also a subgroup of $Im(\phi)$, by Lemma \ref{gh}, and hence $\pi_1(S^2-\{\hbox{4 points}\}))$ is  normal in $Im(\phi)$. Suppose   $Im(\phi)/ (\pi_1(S^2-\{\hbox{4 points}\}))=K,$ then by the third isomorphism theorem of groups,  $\pi_0(Symp_h(X,\w))= (PB_4(S^2)/\ZZ_2)/K = PB_4(S^2)/\ZZ_2.$
Since $PB_4(S^2)/\ZZ_2$ is Hopfian, it is not isomorphic to its proper quotient. Hence $K$ is trivial and  $ Im(\phi)=\pi_1(S^2-\{\hbox{4 points}\})$.

\end{proof}

\begin{rmk}
    We are not going to address the relation between the generators in $Symp_h(X,\w)$ explicitly here, but one may be convinced that the generators $A_{ij}$ are represented by the squares the Lagrangian Dehn twists.  The relevant Lagrangian spheres are constructed in \cite{Wu13}, and the last author included a sketch of a possible approach to prove the identification of the $A_{ij}$ ball-swapping with the Lagrangian Dehn twist. 
\end{rmk}

Note that so far we have covered the Torelli part ($\pi_0(Symp_h)$) of the Main Theorem \ref{t:main}. And the rest of Theorem \ref{t:main} is about the homological action, which follows from \cite{LW12}. Hence we have completed the proof of Main Theorem \ref{t:main}.

\section{On the fundamental group of \texorpdfstring{$Symp_h(X,\w)$}{Symp}}\label{s:pi1}

In this section, we prove the Main Theorem 2 (Theorem \ref{t:main2}), which we rephrase as the following Proposition:

\begin{prp}\label{51}
For a 5 fold blow-up of $\CC P^2$ with a reduced symplectic form $\w$.

\begin{itemize}
  \item If $\w$ is of type $\aA$, then the rank of $\pi_1(Symp_h(X,\w))$ is equal to $N_{\w}-5$.
  \item If $\w$ is of type $\DD_4$, then the rank of  $\pi_1(Symp_h(X,\w))$ has rank 5.
\end{itemize}

\end{prp}

Note that $\pi_i(Symp(X,\w))=\pi_i(Symp_h(X,\w))$ for any $i\geq 1$. We'll also use  the notation $\pi_1(Symp(X,\w))$ or sometimes simply $\pi_1(Symp).$

\subsection{The upper bound and lower bound of \texorpdfstring{$\pi_1(Symp)$}{pi1}}

In \cite{McD08}, Dusa McDuff gave an approach to obtain the upper bound of
$\pi_1[Symp(X,\omega)]$, where $(X, \omega)$ is a symplectic rational 4 manifold.
 We can follow the route of Proposition 6.4 in \cite{McD08} to give proof of the following result, refining \cite[Proposition 6.4, Corollary 6.9]{McD08}.  See also \cite[Proposition 4.13]{LL16}.

\begin{lma}\label{l:mcduff64}
   Let $(X,\w)$ be a symplectic rational surface with $b_2(X)=r$, and $(\wt X,\wt\w)$ be the blow-up of $X$ for $k$ times.  If all the new exceptional divisor $E_i$ has equal area, and this area is strictly smaller than all exceptional divisors in $X$,
   $$rank[\pi_1(Symp_h(\widetilde{X},\wt \w))]\le rank[\pi_1(Symp_h({X},\w))]+kr.$$


\end{lma}

\begin{proof}
  The proof is a simple combination of proof of McDuff's argument for \cite[Proposition 6.4, Corollary 6.9]{McD08} and Pinsonnault's theorem, Theorem \ref{t:Pin}.  Pinsonnault's theorem removes the assumptions of minimality, and that the symplectic Kodaira dimension $\kappa(X)\ge0$ from McDuff's argument, but the rest of the argument follows word-by-word in \cite{McD08}, so we only offer a sketch below and refer interested readers to \cite{McD08} for details.

  Consider the symplectic bundle $\wt P\to S^2$ with fiber $\wt X$ coming from a $\pi_1(Symp(\wt X, \wt \w))$ element.  Given a family of compatible almost complex structures, the classes $E_i$ each has an embedded representative on each fiber by Theorem \ref{t:Pin}.  The unions over all fibers of these exceptional curves form $k$ symplectic submanifold that can be blown-down, yielding $k$ sections $s_i$ whose classes are spanned by $b_2(X)$.  Since the equivalence class of bundles $\wt P$ only depends on the homotopy classes of $s_i$, $kr$ provides an upper bound on the dimension of their possible values.
\end{proof}

This has the following immediate corollary.

\begin{prp}\label{ubound}
Let $(X, \omega)$ be a symplectic rational surface with a given reduced form, $(\widetilde{X}_k, \widetilde{\omega})$ be the blow up of $X$ at $k$
points, and denote  $r=b_2(X)$.  Assume that the $k$ blowup are smaller than an arbitrary exceptional class of $X$.

If the $k$ blow-up sizes are distinct,
then
  $$rank[\pi_1(Symp(\widetilde{X}_k, \widetilde{\omega})]\leq   rank [\pi_1(Symp(X,\omega))] + rk+ k(k-1)/2;$$

 and if the $k$ blowup sizes are the same, then
$$rank[\pi_1(Symp(\widetilde{X}_k, \widetilde{\omega})]\leq   rank [\pi_1(Symp(X,\omega))] + rk,$$
\end{prp}

\begin{proof}

The second statement is Lemma \ref{l:mcduff64}.  For the first statement, apply Lemma \ref{l:mcduff64} iteratively.
\end{proof}




Now let's try to compute the upper bound for a form $\w\in MA,$  where $c_1>c_2=\cdots =c_5$. There are several methods, for example:

(1) we can start with the two-point blowup of $\CC P^2$ where the blowup sizes are different (whose $\pi_1(Symp)$ has rank 3 by \cite{LP04}), then we can use Lemma \ref{l:mcduff64}, and have

$$ rank[\pi_1(Symp_h(X,\w))] \leq 3+ 3\times 3 =12.$$

(2) starting with a non-monotone one-point blowup of $\CC P^2$ (whose $\pi_1(Symp)$ has rank 1 by \cite{AM00}), then we can use Lemma \ref{l:mcduff64} and have

$$ rank[\pi_1(Symp_h(X,\w))] \leq 1+ 2\times 4 =9.$$

From (2) we have the following corollary.

\begin{cor} \label{maub}
 For a form $\w\in MA,$ the  rank of $\pi_1(Symp_h(X,\w))$ is at most  9.

\end{cor}

We also see that there might be different approaches computing the upper bound and one may wonder how to obtain the most effective upper bound from Proposiiton \ref{ubound}.  In the next section, we'll give an algorithm to compute the optimal upper bound using Lemma \ref{l:mcduff64} and explicitly compute all the type $\aA$ cases.

At this point, we look at the problem from a different perspective, through the associated homotopy sequence from \eqref{e:ES5pt}.

\begin{lma}\label{lbound}
 Suppose $X=(\CC P^2\# 5\overline{\CC P^2},\w)$ where $\w$ is diffeomorphic to a $\RR P^2$ packing form, then we have the exact sequence
\begin{equation}\label{e:ES5pt-Ab}
 \pi_1(Symp_h(X,\w))\rightarrow H^1(\msC_0) \xrightarrow[]{f}  Ab(Im(\phi)) \rightarrow  1.
 \end{equation}


In particular,

\begin{itemize}
  \item if   $Symp_h(X,\w) $ is connected (which is the case when $[\w]$ is not monotone nor in $MA$), ${N_{\omega}}-5\le rank[\pi_1(Symp_h(X,\w))]$;
  \item if the form $[\w]\in MA,$ then $5\le rank[\pi_1(Symp_h(X,\w))]$.
\end{itemize}

\end{lma}

\begin{proof}

The sequence \eqref{e:ES5pt} yields
\begin{equation}\label{e:ES5pt-p4}
     1\rightarrow \pi_1(Symp_h(X,\w))\rightarrow \pi_1(\msC_0) \rightarrow  {Im(\phi)} \rightarrow   1.
\end{equation}
We consider the abelianization of this exact sequence. Since the abelianization functor is right exact and  $ \pi_1(Symp_h(X,\w))$ is abelian,  we have the induced exact sequence \eqref{e:ES5pt-Ab}.
$$ \pi_1(Symp_h(X,\w))\rightarrow H^1(\msC_0) \xrightarrow[]{f}  Ab(Im(\phi)) \rightarrow  1.$$

 If $Symp_h$ is connected,  since Lemma \ref{tran} implies $Symp_h$ acts transitively on homologous $(-2)$-symplectic spheres, the space of $(-2)$-symplectic spheres for a fixed homology class is also connected.
  From Lemma \ref{h1open}, the rank of $H^1(\msC_0)\cong H_1(J_{open})=Ab(\pi_1(\msC_0) )$, which is the number of connected components of the space of symplectic $(-2)$-spheres.  It is now equal to the number $N_{\w}$ of  $(-2)$-symplectic sphere classes.
  And since $Im(\phi)=PB_5(S^2)/\ZZ_2$ from Theorem \ref{t:main},  we know (cf. \cite{GG13} Theorem 5) $Ab(Im(\phi))=\ZZ^5$, therefore, ${N_{\omega}}-5\le rank[\pi_1(Symp_h(X,\w))]$.

For the case of $MA$, we have $ {Im(\phi)}=\pi_1(S^2-\{\hbox{4 points}\}))$ from Theorem \ref{t:main}, whose abelanization is $\ZZ^3$.  Now sequanece \eqref{e:ES5pt-Ab}
  reads
  \begin{equation}\label{e:ES5pt-p4Ab}
      \pi_1(Symp_h(X,\w))\rightarrow H_1(\msC_0) \xrightarrow[]{f}  Ab({Im(\phi)})=\ZZ^3 \rightarrow  1.
\end{equation}

And hence we obtain a lower bound on for a form $\w$ on $MA$, rank of $\pi_1(Symp_h(X, \w))\geq N_\w-3=5$.  Then for any $\w \in MA$ we have  $5\leq rank[\pi_1(Symp_h(X, \w))] \leq 9.$


\end{proof}

The above results already deduce an interesting geometric consequence, and it's useful in Lemma \ref{MA} for the computation of the precise rank of $\pi_1( Symp)$ of a type $\DD$ form.

\begin{cor}\label{-2isotopy}
Homologous $(-2)$ symplectic spheres in $5$ blowups are  symplectically (hence Hamiltonian) isotopic for any symplectic form.
\end{cor}
\begin{proof} We discuss the following cases separately.
\begin{itemize}
\item When $\pi_0{Symp_h(X,\w)}$ is trivial, the conclusion follows from the transitivity of the action of $Symp_h(X,\w)$ on homologeous $(-2)$ symplectic spheres, see Lemma \ref{tran}.

\item When $\w$ is monotone, there is no $(-2)$ symplectic sphere.

\item The only case that is not covered by the previous ones is when $\w \in MA$ and $N_{\w}= 8$.  In this case, we know the homological action acts transitively on the set of symplectic $(-2)$ spherical classes, because the $\w$ area of these sphere classes are the same.  Hence, fix a $(-2)$ spherical class $A$, the number of isotopy classes of symplectic $(-2)$-sphere in class $A$ is a constant $k\in \ZZ^+ \cup\{\infty\}$ independent of $A$.

Note that for each embedded $(-2)$-symplectic sphere $C$, we may take an almost complex structure $J_C\in \mX_2\subset \mJ_w$ so that $C$ is $J_C$-holomorphic.  Since such a $(-2)$-sphere $C$ is unique for each $J_C$, and it varies smoothly with respect to $J_C$.  In other words, if $J_C$ and $J_{C'}$ are in the same connected component of $\mX_2-\mX_4$, $C$ and $C'$ are symplectically (hence Hamiltonian) isotopic.  Note that on all strata of $\mJ$ the path connectedness is equivalent to the connectedness because all prime submanifolds are Fr\'echet.

    We will do a counting argument on the connected component of $\mX_2$.  By Theorem 3.9 in \cite{LL16}, we have that both $\mY=\mJ_{\w}\setminus \mX_4$ and  $\mJ_{open}=\mJ_{\w}\setminus \mX_2$ are submanifolds of the Hausdorff space $\mJ_{\w}$.  Then by the relative Alexander duality in Lemma \ref{relalex} , we have  $H^1(\msC_0) = H^1(\mJ_{open}) = H^0( \amalg_{i=1}^8 \mJ_{A_i}) =8k,$ where $A_i$, $1\leq i \leq 8$ are the 8 symplectic $(-2)$ classes.


    By Lemma \ref{lbound}, together with the fact that $Im(\phi)=\pi_1(S^2-\{p_2, p_3, p_4,p_5\}),$ the rank of $H^1(\msC_0)$ is  no larger than $3+ Rank [ \pi_1(Symp(X,\omega))].$ By Corollary \ref{maub}, $ Rank [\pi_1(Symp(X,\omega))]\leq 9$ and hence $rank[H^1(\msC_0)]\le 12$. If $k>1$, then then rank of $H^1(\msC_0)$ is $8k\geq 16>12$, a contradiction.
This means homologous $(-2)$ symplectic spheres have to be symplectically isotopic.
\end{itemize}
\end{proof}

\subsection{Type A forms}

Now we give explicit computations of the upper bound for a type $\aA$ form.  In order to do this we need to recall some explicit computation of $\pi_1(Symp_h(X_k,\w))$ in \cite{LL16} from Table \ref{2form}, \ref{3form} and \ref{4form} as follows.

\begin{table}[ht]
\begin{center}
 \begin{tabular}{||c c c c c ||}
 \hline
   $k$-Face & $\Gamma_L$ & $N_\w(X_2)$ & $\pi_1(Symp_h(X_2,\w))$ & $\omega$-area \\ [0.5ex]
 \hline\hline
OB &$\aA_1$ &0& $\ZZ^2$  & $c_1 = c_2$\\
\hline
 $\Delta BOA$ & trivial  &1 & $\ZZ^3 $   & $c_1 \ne c_2$ \\
 \hline\hline
\end{tabular}
\caption{ $X_2=\CC P^2 \# 2\overline  {\CC P^2 }$ }
\label{2form}
\end{center}
\end{table}

\begin{table}[ht]
\begin{center}
 \begin{tabular}{||c c c c c ||}
 \hline\hline
   $k$-Face & $\Gamma_L$ & $N_\w(X_3)$ & $\pi_1(Symp_h(X_3,\w))$  & $\w-$area \\ [0.5ex]
 \hline
  Point M& $\aA_1\times \aA_2$&  0& $\ZZ^2$& $(\frac13,\frac13,\frac13)$: monotone\\
 \hline
Edge MO:& $\aA_2$&  1&  $\ZZ^3$&  $\lambda<1; c_1=c_2=c_3$  \\
\hline
 Edge MA:& $\aA_1\times \aA_1$& 2 &  $\ZZ^4$&  $\lambda=1;c_1>c_2=c_3 $\\
\hline
Edge MB:&$\aA_1\times \aA_1$&  2& $\ZZ^4$&  $\lambda=1; c_1=c_2>c_3$ \\
\hline
$\Delta$MOA: & $\aA_1$ &   3& $\ZZ^5$&   $\lambda<1; c_1>c_2=c_3 $ \\
 \hline
 $\Delta$MOB:& $\aA_1$ &  3 &  $\ZZ^5$&  $\lambda<1; c_1=c_2>c_3$\\
  \hline
$\Delta$MAB:& $\aA_1$ &  3&  $\ZZ^5$&  $\lambda=1; c_1>c_2>c_3$ \\
 \hline
  $T_{MOAB}$:& trivial&  4&  $\ZZ^6$& $\lambda<1; c_1>c_2>c_3$\\
 \hline\hline
\end{tabular}
\caption{ $X_3$=${\CC P^2}\# 3\overline{\CC P^2}$, $\lambda=c_1+c_2+c_3$}
\label{3form}
\end{center}
\end{table}

\begin{table}[ht]
\begin{center}
  \begin{tabular}{||c c c c c  ||}
\hline\hline
    $k$-face& $\Gamma_L$  & $N_\w(X_4)$ & $\pi_1(Symp_h(X_4,\w))$ &$\omega$-area  \\ [0.5ex]
\hline\hline
Point M& $\aA_4$ & 0& trivial &  $(\frac13,\frac13,\frac13,\frac13)$: monotone\\
\hline
 MO& $\aA_3$ & 4& $\ZZ^4$& $  \lambda<1; c_1=c_2=c_3=c_4 $\\
\hline
MA& $\aA_3$ & 4& $\ZZ^4$&  $\lambda=1;c_1>c_2=c_3=c_4 $\\
\hline
MB&$\aA_1\times \aA_2$  &6& $\ZZ^6$& $\lambda=1;c_1=c_2>c_3=c_4 $ \\
\hline
MC&$\aA_1\times \aA_2$   &6& $\ZZ^6$& $\lambda=1;c_1=c_2=c_3>c_4 $ \\
\hline
MOA&$\aA_2$  &7& $\ZZ^7$& $ \lambda<1;  c_1>c_2=c_3=c_4$\\
\hline
MOB& $\aA_1\times \aA_1$ &8& $\ZZ^8$& $ \lambda<1; c_1=c_2>c_3=c_4$ \\
\hline
MOC&$\aA_2$   &7& $\ZZ^7$&$ \lambda<1; c_1=c_2=c_3>c_4 $\\
\hline
MAB&$\aA_2$  &7& $\ZZ^7$&  $ \lambda=1;c_1>c_2>c_3=c_4$\\
\hline
MAC&$\aA_1\times \aA_1$ &8& $\ZZ^8$&   $ \lambda=1;c_1>c_2=c_3>c_4  $\\
\hline
MBC&$\aA_1\times \aA_1$ &8& $\ZZ^8$&  $ \lambda=1;c_1=c_2>c_3>c_4 $  \\
\hline
MOAB& $\aA_1$ &9& $\ZZ^9$& $\lambda<1; c_1>c_2>c_3=c_4$\\
\hline
MOAC& $\aA_1$    &9& $\ZZ^9$& $ \lambda<1;   c_1>c_2=c_3>c_4  $\\
\hline
MOBC& $\aA_1$   &9& $\ZZ^9$&  $ \lambda<1; c_1=c_2>c_3>c_4 $\\
\hline
MABC& $\aA_1$   &9& $\ZZ^9$&  $ \lambda=1;   c_1>c_2>c_3>c_4  $\\
\hline
MOABC& trivial   &10& $\ZZ^{10}$&  $ \lambda<1;c_1>c_2>c_3>c_4 $\\
\hline\hline
\end{tabular}
\caption{$X_4={\CC P^2} \# 4\overline{\CC P^2}$, $\lambda = c_1+c_2+c_3$.}
\label{4form}
\end{center}
\end{table}

\begin{lma}\label{ubounda}

  If $\w$ is a type $\aA$ reduced form on $\CP^2\#5\ov\CP^2$, then $rank[\pi_1(Symp_h(X_5,\w))]=N_\w-5$.
\end{lma}

\begin{proof} In this proof, we denote the $R_\w:=rank[\pi_1(Symp_h(X_i,\w))]$, $i=2,3,4,5$ for convenience.
  We will show that the inequalities appeared in Lemma \ref{lbound} for type $\aA$ forms are indeed equalites, or rather, $N_\w-5\ge R_\w$.

  From Proposition \ref{ubound}, we may reduce the problem to computations on rational surfaces $\CP^2\#k\ov\CP^2$ for $k\le 4$ \cite{LL16}, where the rank of $\pi_1(Symp_h(X,\w))$  is explicitly given in Tables \ref{2form},\ref{3form}, \ref{4form}.  We'll simply blow down all $E_i$'s with the minimal symplectic area, and then then apply Lemma \ref{l:mcduff64}, setting $X$ to be the blowdown and $\wt X$ to be $X_5$.   In the arguments below, we only specify the exceptional divisors we'll blow down for each face.  Since we know $\pi_1[Symp(X_k,w)]$ depends only on the face $\w$ belongs to when $k\le4$, the same holds for $X_5$. In what follows we study case-by-case explicitly.  We invite the reader to refer to Table \ref{5form} for checking the numerical conditions for each face.  What one should observe from Table \ref{2form}, \ref{3form} and \ref{4form} (indeed this is Theorem 4.8 in \cite{LL16}) is that,

  \begin{itemize}
     \item in $X_4$, $N_\w=rank[\pi_1(Symp_h(X_4,\w))]$,
     \item in $X_i$, $N_\w+2=rank[\pi_1(Symp_h(X_i,\w))]$, $i=2,3$.
   \end{itemize}



\begin{enumerate}
\item[ 1)] For $\w$ in any $k$-face with vertex $D$, we have $c_4>c_5;$ then we blow down $E_5$ and obtain $X_4=\CC P^2  \# 4{\overline {\CC P^2}}$
  with some form $\ov\w$.
    By Proposition \ref{ubound}, $R_\w\le R_{\ov\w}+5$.  Because $E_5$ is the only smallest area exceptional sphere,  there are 10 symplectic $(-2)$-sphere classes ($H-E_i-E_j-E_k, 1\leq i\ne j \ne k=5$; and $E_i-E_5$, $1\le i\le4$) pairing $E_5$ nonzero. Hence $N_{\w}=N_{\ov\w}+10=R_{\ov\w}+10$, which implies $N_\w-5=R_{\ov\w}+5\ge R_\w$.

\item  [2)] For $\w$ in any $k$-face without vertex $D$ but with $C$, we have $c_3>c_4=c_5$, then we blow-down both $E_4$ and $E_5$ and consider $  X_3=\CC P^2  \# 3{\overline {\CC P^2}}$
with some $\ov\w$.  From Proposition \ref{ubound}, $R_\w\le R_{\ov\w}+4+4=R_{\ov\w}+8$.

On the other hand, $E_4, E_5$ are the only two smallest area exceptional spheres, there are 15 symplectic $(-2)$-spherical classes intersecting one of $E_4$ and $E_5$ (6 intersecting $E_4$ only, 6 intersecting $E_5$ only, 3 intersecting both).
Hence $N_{\w}=N_{\ov\w}+15=R_{\ov\w}+13$.  Therefore, $N_\w-5=R_{\ov\w}+8\ge R_\w$


\item[ 3)]For any $k$-face without vertex $D$ or $C$ but with $B$, there are 4 cases: $MOAB,$ $MOB, $ $MAB,$ and $MB.$
For $ MOAB$ or $ MOB,$ we have  $c_2>c_3=c_4=c_5$. Blow down $E_3, E_4, E_5$ and consider
$X_2$  with $\ov\w$.  Similarly, we have $R_\w\le R_{\ov\w}+9=N_{\ov\w}+11$ from Proposition \ref{ubound}.

    Since $N_{\ov\w}$ is the number of symplectic $(-2)$-spheres which do not intersect any of $E_3,E_4, E_5$ in $X_5$, we count the symplectic $(-2)$-spheres which intersect $E_3,E_4, \text{or } E_5$.  There are 16 such classes: $H-E_1-E_i-E_j, H-E_2-E_i-E_j, E_1-E_i, E_2-E_i, H-E_1-E_2-E_j $ or $H-E_3-E_4-E_5$. And hence $R_{\ov\w}+9=N_{\w}-5\ge R_\w$.

    For $MAB$, perform a base change \eqref{BH},  $\w(B)=1-c_2\ge \w(F)=1-c_1; E_1=\cdots=E_4=c_3$ and blow it down to a non-monotone $S^2\times S^2$.  From \cite{AM00}, $\pi_1(Symp(S^2\times S^2,\w'))=1$ when $\w'$ is non-monotone.  Proposition \ref{ubound} hence implies  $R_\w \leq 1+2+2+2+2=9$, which coincides with the lower bound $N_{\w}-5$.

For case $MB$, again perform base change \eqref{BH}, $B=F=1-c_1; E_1=\cdots=E_4=c_3$ and blow-down to a monotone $S^2\times S^2$.  Then from Gromov's result $ R_\w \leq 0+2+2+2+2=8$, which coincides with the lower bound $N_{\w}-5$..

\item  [4)] $MOA$ is the only type $\aA$ face without vertex $B$, $C$, $D$, but with $A$, for which we have  $c_1>c_2=c_3=c_4=c_5.$  We blow down $E_2$ through $E_5$ and consider
a non-monotone $ \CC P^2  \# {\overline {\CC P^2}}$. Then from \cite{AM00} similar to case 3), we have an upper-bound $R_\w\le 1+2+2+2+2=9$ for both $MOA$ and $MA$.
    Note that $MOA$ is of type $\aA$ and $MA$ is of type $\DD_4$. For this theorem, we only need to show that $N_{\w}-5=9$ on $MOA$, which is true from Table \ref{5form}.

\item  [5)] For $MO$, we simply blow down $E_1$ through $E_5$, which yields an upper-bound $5$ from Proposition \ref{ubound}.


\end{enumerate}

\end{proof}

Note that the method in the proof above does not give the precise rank for the face $MA$, which numerically falls into the same situation as in $4)$, which yields $R_\w\le 9$ (cf. Corollary \ref{maub} and Lemma \ref{lbound}).

 \begin{rmk}

   Note that the proof of Lemma \ref{ubounda} gives the method how to apply Lemma \ref{ubound} to obtain an optimal upper bound.

   (1) if we start with $(X,\omega)=(\CC P^2,\omega_{FS}),$ let $(\widetilde{X}_k, \widetilde{\omega}_\epsilon) $ be the blow  up  of $(X,\omega)$ $k$ times with area of $E_i$ being $\epsilon_i$ and $\widetilde{\omega}_\epsilon$ being a reduced form, then
$$Rank[\pi_1(Symp(\widetilde{X}_k, \widetilde{\omega}_\epsilon)]\leq  k+ N_E,$$
where $ N_E$ is the number class of the form $E_i-E_j$ which pair positively with $\wt\w_\epsilon$.

   (2) We can also start with $X$ being blow-up of several points of $\CC P^2$, instead of $\CC P^2$ itself, one get finer results on the upper-bound of rank $\pi_1(Symp_h(X,\w))$. This is why we need the second case of Proposition \ref{ubound}.

   An example is case $MABCD$ in form \ref{5form}, using  method (1), one have 15 as the upper-bound; while using method (2) starting with $\CC P^2 \# 3{\overline {\CC P^2}}$ of sizes $c_1,c_2,c_3,$ one have $5+4+5 =14 $ as the upper-bound.


\end{rmk}

\begin{rmk}\label{small5}
  In the work of Anjos and Eden \cite{AE17}, they also
deduce the rank of other homotopy groups for some special cases of 5 blow-ups of the projective plane, in particular, a generic blow-up with very small size. Their result agrees with what we got on the fundamental group.
\end{rmk}

\subsection{Type \texorpdfstring{$\DD_4$}{D4} forms}

So far, we have a bound $5\le\pi_1(Symp_h(X,\w))\le9$ for a type $\DD_4$ forms $\w$. Our last lemma prove that this rank is actually $5$.

\begin{lma}\label{MA}

  Choose a rational point $[\w]$ on $MA$.  Assume $c_1>\frac12$, the rank of $\pi_1(Symp_h(\CC P^2  \# 5{\overline {\CC P^2}},\w))$ is 5.

\end{lma}

\begin{proof}

Consider the following configuration $C$ consisting of 7 exceptional spheres in the  5-point blowup,
\[
\xy
(0, -10)*{};(80, -10)* {}**\dir{-};
(100, -13)* {D=H-E_3-E_4};
(10, 0)*{}; (10, -50)*{}**\dir{-};
(10, 1)*{H-E_1-E_5};
(31, 0)*{}; (31, -20)*{}**\dir{-};
(31, -25)*{E_3};
(40, 0)*{}; (40, -20)*{}**\dir{-};
(40, -25)*{E_4};
(0, -30)*{}; (20, -30)*{}**\dir{-};
(22, -30)*{E_5};
(60, 0)*{}; (60, -50)*{}**\dir{-};
(60, 1)*{H-E_1-E_2};
(50, -30)*{}; (70, -30)*{}**\dir{-};
(72, -30)*{E_2};
\endxy
\]

We claim that its complement has a symplectic completion symplectomorphic to $\CC\times \CC^*$, iff $\w(E_1)>\w(E_2)+\w(E_5)$.  Since $[\w]\in MA$,  this is equivalent to $c_1>\frac12.$
To see this, take a generic compatible integrable complex structure $J_0$ on $X$.  Since $[\omega] \in H^2(X;\QQ)$, we can write  $PD([l\omega])=aH-b_1E_1-b_2E_2-b_3E_3-b_4E_4-b_5E_5$ with  $ a, b_i \in \ZZ^{\geq0}$. Further, we assume $ b_1\geq \cdots \geq b_5$. Then we can represent $PD([l\omega])$  as a
positive integral combination of all elements in the set $\{H-E_1-E_2, H-E_3-E_4, H-E_1-E_5,E_2,E_3, E_4, E_5\}$, 
  which is the homology type of $C$.  And the proof is a direct computation by checking when the form is a positive combination of the divisor classes:

\begin{align}
PD([l\omega])=aH-b_1E_1-b_2E_2-b_3E_3-b_4E_4-b_5E_5 \nonumber\\
=d_7( H-E_3-E_4 )   \nonumber\\
+ d_6 (H-E_1-E_5) \\
+d_1 (H-E_1-E_2)  \nonumber   \\
+ d_2 E_2 +\cdots d_5 E_5 \nonumber
\end{align}

When $l$ is large enough so that $a$ and $b_i$ are all integers, we may further assume from $c_1>\frac{1}{2}$ and $c_1+c_2+c_3=1$ that $b_1\ge b_2+b_3+1$.  Then the above equation can be solved inductively by first noticing $d_7=a-b_1$ and setting $d_5=1$, and the solution will be positively integral.
  Hence this complement is a psuedoconvex domain. Also, because the form $\w$ has rational period, by the argument in \cite{LLW15} Proposition 3.3, the complement is therefore Stein. Note that it is a complex line bundle over $\CC^*$(which is automatically Stein), and any line bundle over a Stein base is a trivial bundle, by the main Theorem in \cite{Gra58}. The underlying complex manifold is indeed $\CC\times\CC^*$.




Consider diagram \ref{summary} for the space of configurations $\mS$ specified above, we have

 \begin{equation}\label{e:pi1seq}
\begin{CD}
Symp_c(U)=Stab^1(C) @>>> Stab^0(C) @>>> Stab(C) @>>> Symp_h(X, \omega) \\
@. @VVV @VVV @VVV \\
@. \ZZ^5 @. (S^1)^6\times PB_4(S^2)/\ZZ_2 @.  \mS
\end{CD}
\end{equation}

We use a ball-swapping argument similar to that in Proposition \ref{fib5} to show the fibrancy of \eqref{e:pi1seq} and deduce that $Stab(C)$ weakly homotopic equivalent to $\Diff(S^2,4)\times S^1.$  To see this, we need to prove the restriction map $Stab(C)\to Symp(C)$ is surjective on the factor $Symp(D,4)$.
 In other words, for any given $ h^{(2)}\in Symp(D,4)$ we need a lift $h^{(4)}\in Stab(C) $  which fixes the whole configuration  $C$ as a set, whose restriction on $D$ is $ h^{(2)}.$

   To achieve this, we
 first perform a base change on $H_2(X_5, \ZZ)$ with  $e_1=H-E_1-E_5, e_2= H-E_1-E_2, e_3=E_3, e_4= E_4, e_5=H-E_2-E_5, h=2H-E_1-E_2-E_5$ and  $D=2h-e_1-\cdots -e_5.$  Blowing down  $e_1 \cdots e_4$, we obtain
 a $(\CC P^2 \# {\overline {\CC P^2}}, rel \coprod_{i=1}^4 B(i))$ with a rational curve in homology class $2H-E_1$ and four disjoint balls  $\coprod_{i=1}^4 B(i)$ each centered on
 this sphere and the intersections are 4 disjoint disks on this $S^2$.   By the identification in Lemma \ref{relflux}, this blow down process sends  $ h^{(2)}$ in $Symp(D,4)$ to a unique
 $\overline {h^{(2)}} $ in $Symp(S^2,\coprod_{i=1}^4 D_i).$

 Note that we need to construct a symplectomorphism $f^{(4)}$ on $(X_5,\w)$ so that it preserve the configuration and restricted to  $ h^{(2)}$ in $Symp(D,4)$.  Let's temporarily forget about the divisors in the classes $E_2$ and $E_5$. Then after blowdown, we can use the same method in \ref{relflux} to construct a ball-swapping symplectomorphism $g^{(4)}$ in $(\CC P^2 \# {\overline {\CC P^2}}, rel \coprod_{i=1}^4 B(i))$ so that the restriction is $\overline {h^{(2)}} $ in $Symp(S^2,\coprod_{i=1}^4 D_i).$ Then after blowup, we have a  $\tilde{g}^{(4)}$ so that it restricts to   $ h^{(2)}$ in $Symp(D,4)$. Now we need to take care of the two divisors in the classes $E_2$ and $E_5$, because  $\tilde{g}^{(4)}$ may move them to a different position. Note they are exceptional divisors and they do not intersect the divisor $D$ in class $H-E_3-E_4$. We can always find symplectomorphism $\phi^{(4)}$ supported away from $D$ and moving the two divisors back to their original position.    Now  $f^{(4)}:= \phi^{(4)} \circ \tilde{g}^{(4)} $ is the desired symplectomorphism  on $(X_5,\w)$ preserving the configuration and restricting to  $ h^{(2)}$ in $Symp(D,4)$.

Then Theorem \ref{fib3} shows that $Stab(C) \rightarrow  Symp(C)$ is a  fibration.

Now use Lemma \ref{symgau}, in particular equation \eqref{sympk} and \eqref{mG}, we know $Stab^0(C)\sim \mG(C)\sim \ZZ^5$ and $Symp(C)\sim (S^1)^6\times PB_4(S^2)/\ZZ^2$. Hence completes the diagram \eqref{e:pi1seq}.  Further, the connecting map from $\pi_0(Symp(C))$ to  $\pi_1(Stab^0(C))$ is surjective. Given the above discussions, we know that
$$Stab(C)\sim S^1\times PB_4(S^2)/\ZZ_2.$$
On $MA$,   assume $c_1>\frac12,$ we have  the Stein complement and hence the following LES
\begin{equation}\label{MA1}
   \ZZ \to  \pi_1(Symp_h(X,\w))\rightarrow   \pi_1(\mS) \xrightarrow[]{f}  \pi_0(Stab(C)) \xrightarrow[]{g}   \pi_0(Symp_h) \to 1.
\end{equation}

Firstly note that $\pi_0(Stab(C)) \cong  \pi_0(Symp_h)\cong PB_4(S^2)/\ZZ_2$ is Hopfian. Since the map $g$ is a self-epimorphism of $ PB_4(S^2)/\ZZ_2$, it must be an isomorphism. Then we know that the image of the map $f$ is the trivial group.  This means that  $\pi_1(Symp_h(X,\w))$ is an extension of $\pi_1(\mS)$ by a subgroup of $\ZZ.$

 We consider  $H^1(\mS)$ for the moment.  Note that the only curve that could be non-smooth is the one in class $D=H-E_3-E_4$ by Theorem \ref{t:Pin}.  Denote the space of almost complex structures which allows an embedded $D$-representative as $J_\mS$, then $J_\mS\sim \mS$.  Let $\mX^{-}_2$ be the complement of $J_\mS\subset J_\w$ and $\mX_4$ denote the strata of codimension $4$ or more as before.  Note that $\mX_4$ intersects both $J_\mS$ and $\mX_2^-$.  Explicitly, the codimension $2$ strata in $\mX^-_2$ contain those almost complex structures which allow embbedded holomorphic representatives for $H-E_3-E_4-E_i$ for $i=2$ or $5$; or $E_1-E_i$ for $j=3$ or $4$.
 Also, note that $H^1(J_\mS)= H^1(J_\mS\setminus \mX_4).$

 Now we compute   $H^1[J_\mS\setminus \mX_4]$ by the relative Alexander duality in Lemma \ref{relalex}. Let
\begin{itemize}
 \item $\mX=\mJ_{\w},$\hspace{5mm} \item $\mY= \mX^{-}_2,$ \hspace{5mm} \item $\mZ= \mX_4.$\end{itemize}

By Theorem \ref{rational} $\mX_4$ is closed in $\mX$. We also know that $\mY$ is closed in $ \mX$: $\mY-\mZ=\mX_2^--\mX_4$   is the union of the 4 prime submanifolds characterized by  the 4 classes $H-E_2-E_3-E_5,$ $H-E_2-E_4-E_5,$ $E_1-E_2$ and $E_1-E_5$, whose closures do not intersect $\mX- \mY=J_\mS$. Then $\mX-\mZ$ and $\mX-\mY$ are clearly submanifolds of $\mX.$  Also, $\mY-\mZ$ is a closed submanifold of codim 2 in $\mX-\mZ= \mJ_{\w}-\mX_4.$   We can now appeal to Lemma \ref{relalex} and deduce that $H^1(\mS)=H^1(J_\mS)= H^0(\mY-\mZ).$  And by Corollary \ref{-2isotopy}, each of the four prime submanifolds in $J_\mS$ is connected, so we know that $Ab[\pi_1(\mS)]= H^1(\mS) =\ZZ^4$.


Consider the abelianization of sequence \eqref{MA1}.
By the right exactness of the abelianization functor,  we know $\pi_1(Symp_h(X,\w))$ at most has rank 5. This together with the Lemma \ref{lbound} give us the exact rank of $\pi_1(Symp_h(X,\w))$ being 5, for any form $\w \in MA$ with $c_1>\frac12.$

\end{proof}

\begin{proof}[Proof of Theorem \ref{t:main2}]

The only missing part is when $\w$ is a type $\DD$ irrational form or $c_1\le\frac{1}{2}$.  Using  Proposition \ref{nonbalstab}, we can extend the case of $c_1>\frac{1}{2}$ to any form on $MA$ and conclude Theorem \ref{t:main2}

\end{proof}

\appendix

\section{Proof of McDuff-Salamon's problem} 
\label{sec:proof_of_mcduff_salamon_problem}

\begin{thm}\label{t:MS}
  Any symplectomorphism of $(X_k:=\CP^2\#k\ov\CP^2,\w_k)$ is smoothly isotopic to identity if it acts trivially on homology.
\end{thm}

\begin{proof}

 Since the space of homologous exceptional curves are connected via Hamiltonian diffeomorphisms \cite[Proposition 3.2]{LW12}, we may assume $f|_{N(C_i)}=id|_{N(C_i)}$ for some $[C_i]=E_i$.  Also denote $c_i=\w(C_i)$.

    We know that $f$ is a ball-swapping with respect to some embeddings of $B(c_i)$ in \cite{LWnote}, which we recall a sketch here.  Consider the action of $Symp(\CP^2)$ on the space of (unparametrized, ordered) ball-packing $Emb(B(c_i))$, then the fiber of the resulting fibration is the symplectomorphism group that fixes $B(c_i)$, denoted as $Ham(\CP^2,c_i)$.  Note that upon blow-up, symplectomorphisms in $Ham(\CP^2,c_i)$ is homotopy equivalent to $Symp_h(X_k,C_i)$, the symplectomorphisms of $X_k$ which fix $C_i$, in which $f$ lives.  By examining the associated homotopy sequence, we find that
    \begin{equation}\label{e:connecting}
         \pi_1(Emb(B(c_i)))\to \pi_0(Ham(\CP^2,c_i))\cong \pi_0(Symp_h(X_k,C_i))\to \pi_0(Ham(\CP^2))=1
    \end{equation}
     is exact, hence the first map is surjective.  From our construction, the image of the first map consists of ball-swapping symplectomorphisms, as desired.

     We will take an arbitrary lift of this connecting homomorphism from $f$ to a loop of ball-packing, denoted as
     \begin{equation}\label{e:}
          \iota_f(t):\coprod_i B(c_i)\hookrightarrow \CP^2.
     \end{equation}
For any ${\delta}:=\{\delta_1,\cdots,\delta_k\}$ such that $\delta_i<c_i$ sufficiently small, we argue that there is an associated ball-swapping $f_{\delta}$ of $(X_k,\w_{\delta})$, where $[\w_{\delta}](C_i)=\delta_i$.  Moreover, if $f_{\delta}$ is smoothly isotopic to identity, so is $f$.

To see this, blow-down each $C_i$ and obtain a ball embedding of $B(c_i)$.  Assume $N(C_i)$ now becomes a slightly larger ball $B(c_i+\delta_i)\supset B(c_i)$ after the blow-down.  Take a diffeomorphism $\wt\varphi:\wt X-B(c_i)\to \wt X-B(\delta_i/2)$ so that $\wt\varphi=id$ outside $B(c_i+\delta_i)$, and that it is $S^1$-equivariant with respect to the circle actions on  $B(c_i)$ and $B(\delta_i/2)$.  The last condition enables one to descend $\wt\varphi$ to a diffeomorphism between $(X_k,\w)$ to $(X_k,\w_\delta)$ for a new symplectic form $\w_\delta$, which we denote as $\varphi$.  Note that $\varphi$ is not a symplectomorphism.  However, we see that $f_\delta:=\varphi\circ f\circ\varphi^{-1}: (X_k,\w_{\delta})\to (X_k,\w_{\delta})$ is still a symplectomorphism: outside $B(c_i+\delta_i)$, $\varphi$ is a symplectomorphism, while inside we have $f=id$.  Therefore, if $\varphi\circ f\circ\varphi^{-1}\sim id$ smoothly, we have $f\sim id$ in $\Diff^0(X)$.

We note the reader that, in contrast, even if $\varphi\circ f\circ\varphi^{-1}\sim id$ in $Symp_h(X_k,\w_\delta)$, it does not imply $f\sim id$ in $Symp_h(X_k,\w)$, because the property that $\varphi\circ f\circ\varphi^{-1}$ being a symplectomorphism depends crucially on the interaction between $f$ and $\varphi$, and cannot be perturbed arbitrarily.
The rest of the proof will show that there exists a sequence $\delta_i>0$, such that the resulting $f_\delta$ is \textit{symplectically} isotopic to identity.



Note that $\iota_f(t)$ restricts to the smaller balls $B(\delta_i)$, which yields a loop of packing of $\coprod_iB(\delta_i)$ in $\CP^2$.  This loop is, by definition, a lift $\iota_{f_\delta}(t)$ of $f_\delta$ through the connecting map \eqref{e:connecting}.  Therefore, we reduce our problem to showing that $\iota_{f_\delta}(t)$ is trivial in $\pi_1(Emb(B(c_i)))$.

Let $x_i$ be image of the center of $B(c_i)$ under $\iota_{f_\delta}(0)$.  $\iota_{f_\delta}(t)(x_i)$ are $k$ disjoint smooth loops which can be isotoped through a family of Hamiltonian diffeomorphism $\gamma(s,t)$ to constant loops at $x_i$ simultaneously.  In other words, we have
\begin{equation} \label{e:homotopy}
     \left\{
        \begin{aligned}
          &\gamma(0,t)=\gamma(s,0)=id, \\
          &\gamma(1,t)(x_i)=\iota_{f_\delta}(t)(x_i),\\
          &\gamma(s,1)(x_i)=x_i
        \end{aligned}
     \right.
\end{equation}
  For convenience, we may also require $\gamma(s,t)(x_i)\cap \gamma(s,t)(x_j)=\emptyset$ if $i\neq j$.

Take an arbitrary metric $g$ on $\CP^2$  and assume $\frac{1}{K}<|\gamma(s,t)(x_i)|_{C^2}<K$ under $g$.
One may choose a sequence of small $\lambda_i>0$, so that there is a disk centered at $x_i$ of radius $r_i$ (measured by $g$), denoted as $D_g(x_i,r_i)\subset \CP^2$, which satisfies
$$\iota_{f_\delta}(0)(B(\lambda_i))\subset D_g(x_i,r_i)\subset \iota_{f_\delta}(0)(B(c_i)).$$

Choose $\delta_i<\lambda_i$ such that $diam_g(\iota_{f_\delta}(t)(B(\delta_i)))<\frac{r_i}{K}$ for all $t$.  Then we have $\gamma(1,t)^{-1}(\iota_{f_\delta}(t)(B(\delta_i)))\subset D(x_i,r_i)\subset \iota_{f_\delta}(0)(B(c_i))$.

Note that $\gamma(s,t)^{-1}(\iota_{f_\delta}(t)(B(\delta_i)))$ does \textit{not} yield a homotopy of loops of embeddings, unfortunately.  What breaks down is that, we may not require $\gamma(s,1)^{-1}(\iota_{f_\delta}(t)(B(\delta_i)))$ to be  independent of $s$.  However, notice that
\begin{equation}\label{e:contain}
     \gamma(s,1)^{-1}(\iota_{f_\delta}(1)(B(\delta_i)))=\gamma(s,1)^{-1}(\iota_{f_\delta}(0)(B(\delta_i)))\subset \gamma(s,1)^{-1}(D_g(x_i,\frac{r_i}{K}))\subset D_g(x_i,r_i)\subset \iota_{f_\delta}(0)(B(c_i)).
\end{equation}

Therefore, we have a loop of symplectic packing (of $B(\delta_i)$, parametrized by $s$)
  $$\beta(s):=\gamma(s,1)^{-1}(\iota_{f_\delta}(1)(B(\delta_i)))\subset \iota_{f_\delta}(0)(B(c_i)).$$

  Here we recall a Lemma that can be proved by the strategy in \cite{LP04}:

\begin{lma}\label{ballcontr}
The space $Emb(B(c),\delta)$, consisting of symplectic embedded images of an open ball $B(\delta)$ into $B(c)$, is weakly contractible, if $\delta\ll c$.
\end{lma}
\begin{proof}

  Indeed we prove a stronger statement: The space of symplectic open ball $B^4(\delta)$ in   $B^4(1)$ is weakly contractible provided that $\delta\leqslant \frac12$.

Apply Theorem 2.5 in \cite{LP04} for $\CC P^2$ with the Fubini-Study form on $S^2$ with size $1$ on the line class.  Choose a line and denote by $H$, then let's consider the following action:

  $$Symp(\CC P^2\#\ov{\CC P^2}, E; H )\to  Symp(\CC P^2; H)\to Emb(B(1), \delta)= Emb(\CC P^2-H;B(\delta)).$$

  Here $Symp(\CC P^2; H)$ is the group of symplectomorphism of $\CC P^2$ supported away from the fixed symplectic curve in class $H$ (by abuse of notation, we will also denote such a curve as $H$), $Symp(\CC P^2\#\ov{\CC P^2}, E; H)$ is the symplectomorphism of $\CC P^2\#\ov{\CC P^2}$ supported away from the fixed line $H$ and fix the exceptional divisor $E$ as a set (Lemma 2.3 in \cite{LP04}).


 Note that the term $Symp(\CC P^2; H)$ is weakly homotopy equivalent to the compactly supported symplectomorphism of $B(1)$, therefore, we only need to show that  $Symp(\CC P^2\#\ov{\CC P^2}, E; H)$ is weakly contractible when $\w(E)\le\frac12$, which by the homotopy fibration yields that
  $ Emb(B(c),\delta)$ is  weakly contractible, provided that  $\delta\leqslant \frac12$.

  In order to do this, we consider the action of $Symp(\CC P^2\#\ov{\CC P^2},\w)$ on the space of holomorphic curves in the class $H$. By \cite{AGK09} Theorem 1.1, the space of integrable complex structures $\mJ_{int}$ compatible with a given $\w$ is weakly contractible.  Note that all integrable $J$ compatible with $
  \w$ with $\w(E)\le\frac12$ on $\CC P^2\#\ov{\CC P^2}$ has no negative rational curves other than the exceptional curve (see \cite{Zha17,LL16}), hence is always a blowup complex structure.  Therefore, we can perform a complex blow-down of the exceptional divisor.  Therefore, the space of irreducible holomorphic curves in the class $H$ is the same as the space of lines in $\CC P^2$ that do not pass the blowup point.

  Now by 9.5.11 in \cite{MS04}, for any fixed $J$, the space of lines in $\CC P^2$ that does not pass through a given point is diffeomorphic to $\RR^4$. This means that for a given integrable $J'$ on $\CC P^2\#\ov{\CC P^2}$, the moduli space of holomorphic curves in the class $H$ is also diffeomorphic to $\RR^4$.  Therefore, the fibration $\mM(J;H)\to \mA\to \mJ^{int}_\w$ has both contractible space and fiber, where $\mM(J;H)$ is the moduli space of $J$-curves in class $H$, and $\mA$ consists of pairs $(C,J)$, where $C$ is an irreducible $J$-holomorphic sphere in class $H$.  $\mA$ is contractible. On the other hand, $\mA$ fibers over the space of symplectic curves in class $H$ with contractible fibers.  This follows from the fact that there is a symplectomorphism connecting two symplectic rational curves in class $H$, therefore, the space of integrable complex structure which makes $C$ $J$-holomorphic is non-empty and contractible.  Then we know that the space of symplectic curves in the class $H$ on  $\CC P^2\#\ov{\CC P^2}$ is weakly contractible.

  Then the conclusion for $(\CC P^2\#\ov{\CC P^2},\w)$ comes from the following fibration, since the total space $Symp(\CC P^2 \# \overline{\CC P^2 })$ is weakly homotopic to $U(2)$ and the fiber  $Symp(\CC P^2\#\ov{\CC P^2}, E; H)$ in the leftmost fibration has to be weakly contractible.

\[
\begin{CD}
Symp(\CC P^2\#\ov{\CC P^2}, E; H)@>>> Stab^0(H) @>>> Stab(H) @>>> Symp(\CC P^2 \# \overline{\CC P^2 }) \\
@. @VVV @VVV @VVV \\
@. \mG\simeq S^1  @. Symp(H)\simeq SO_3    @. \mH\simeq \star
\end{CD}
\]

\end{proof}

From Lemma \ref{ballcontr}, we know in particular the space of single small ball-packing in a ball is simply connected.  Therefore, $\beta(s)$ is homotopic to the constant loop, which is the embedding $\iota_{f_\delta}(0)(B(\delta_i))$ itself.  One may concatenate this isotopy (in $t$-direction) to $\gamma(s,t)(\iota_{f_\delta}(t))$, and the end result is a family of embedding $\ov\gamma(s,t):B(\delta_i)\hookrightarrow \CP^2$, which is a homotopy of loops of packings from the constant loop $\ov\gamma(0,t)\equiv \iota_{f_\delta}(0)$ to $\ov\gamma(1,t)=\iota_{f_\delta}(t)$ with both endpoints fixed.

To sum up, we have an isotopy of loop of ball-packings $\ov\gamma(s,t)^{-1}\iota_{f_\delta}(t)(B(\delta_i))$ from $\iota_{f_\delta}(t)(B(\delta_i))$ to a loop $\gamma(1,t)^{-1}(\iota_{f_\delta}(t)(B(\delta_i)))$ inside a fixed symplectic ball $\iota_{f_\delta}(0)(B(c_i))$.  We again use the fact that the space of a single ball-packing inside $\iota_{f_\delta}(0)(B(c_i))$ is simply connected, hence a concatenation of a further homotopy yields an isotopy from $\iota_{f_\delta}(t)(B(\delta_i))$ to the identity loop of ball-packing, as desired.

\end{proof}










\section{An infinite dimensional slice theorem of Fujiki-Schumacher and group action fibration} \label{s:proper}

This is a recap of the slice theorem of  \cite{FS88} and a detailed proof of \ref{fiblocsection}.

\begin{lma}
The orbit space $(\mJ^c_{\w}-\mJ^c_4)/Symp_h$ is Hausdorff and locally modelled on Fr\'echet spaces. The orbit projection of the free proper action $Symp_h$ on $(\mJ_{\w}-\mJ_4)$ is a fibration with fiber $Symp_h$.
\end{lma}

This Lemma is a corollary of the following results from \cite{FS88}.  Let's first recall the set up and notations: Let $\mJ^c_{\w}$ be the space of $C^{\infty}$ $\w$-compatible almost complex structures and $Symp_h$ the symplectomorphism group preserving homology classes. Denote $\mJ^{c,k}_{\w}$ and $Symp^k_h$ their $H^k$ completion, which are Hilbert manifolds. Note that the Frechet manifold $\mJ^c_{\w}$ ($Symp_h$ respectively) is the Inverse Hilbert limit (ILH-V-space) of the completion, i.e.   $\mJ^c_{\w}= \underset{k\to \infty}{\lim} \mJ^{c,k}_{\w}.$ And let $\mM^k$ be the quotient of $\mJ^{c,k}_{\w}$ by $Symp^k_h$.

\begin{proof}
The Hausdorff property follows from Proposition 6.1 in \cite{FS88}. And the local Fr\'echet-ness follows from Corollary 5.7 in \cite{FS88}, since the slice is a submanifold of the Fr\'echet manifold  $\mJ_{\w},$ and it's homeomorphic to the open set of $(\mJ_{\w}-\mJ_4)/Symp_h$.

  Note that the statements are for $Symp$ (or its $H^k$-completion)  and we stated the $Symp_h$ version since they only differ by a finite group for the rational 4-manifold.

\begin{thm}[ \cite{FS88} Theorem 5.6]
Let $J\in \mJ^{c,k}_{\w}$ and  $k>2m + 3.$
Then there exists
always a slice $\mJ^k$ through $J$ for the action of $Symp^k_h$  on
$\mJ^{c,k}_{\w}$.

\end{thm}

\begin{cor} [ \cite{FS88} Corollary 5.7]
  The natural map of the quotient spaces $\mJ^k / Symp_h(J) \to \mM^k$ are homeomorphisms onto open subsets for $J\in\mJ^{c,k}_{\w} $, where $ Symp_h(J)$ is the isotropy group at $J$.

\end{cor}

\begin{thm}[ \cite{FS88} Theorem 6.9]
The moduli space of almost Kahler structures on a real
symplectic manifold $(M, \w)$ that allow no holomorphic vector fields other
  than zero is a Hausdorff ILH-V-space of class $C^{\infty,0}.$

\end{thm}

  \begin{cor} [ \cite{FS88}   Corollary 5.3]
    Let $J$ be a fixed almost complex structure. Then $\mu(-,J): Symp^{k+1}_h \to  \mJ^{c,k}_{\w}$
is of class $C^k$, its image, the orbit $\mO^k_J$, is closed in $\mJ^{c,k}_{\w}$. Further, there is a $Symp^{k+1}_h$-invariant neighborhood $W^k$ of the
zero-section of $\nu^k$ which is mapped diffeomorphically by the exponential map
onto a neighborhood $N^k$ of $\mO^k_J$ in $\mJ^{c,k}_{\w}$.
The map is $Symp^{k+1}_h$-equivariant.

\end{cor}

  The orbit projection $\pi: (\mJ_{\w}-\mJ_4) \to \mB:= (\mJ_{\w}-\mJ_4)/Symp_h $ is clearly surjective. We see it is also a submersion as follows. For any given point $p$ in $\mB,$ consider a tangent vector $\vec{v}\in T_p\mB$ represented by a path $\gamma_t$. On the local chart of $(\mJ_{\w}-\mJ_4)/Symp_h$ containing $p$, by the slice theorem (Inverse Hilbert space Limit (IHL)  version of Theorem 5.6 in \cite{FS88}), we can lift $\gamma_t$ into a path $\Gamma_t$ s.t. $ P= \Gamma_0 \in \pi^{-1}(p)$  in the slice $S_{P}$ which is a subset of $(\mJ_{\w}-\mJ_4).$ Denote the tangent vector of  $\Gamma_t$  at $t=0$ by $\vec{u}$ Then $d \pi_(P)(\vec{u})=\vec{v}.$ This means the projection is a submersion in the differential geometric sense. It is  a homotopic submersion.

The proof of Theorem 6.9 of \cite{FS88} confirms that Corollary 5.3 in \cite{FS88} holds in the IHL(inverse Hilbert limit) setting. Then we have an invariant neighborhood at any given point in $(\mJ_{\w}-\mJ_4).$ This means that for any fiberwise continuous map $[0,1] \times S^n  \to (\mJ_{\w}-\mJ_4),$ we can make the interval $[0,1]$ in the normal direction of the fiber. Then consider the invariant neighborhood of any preimage of the point $0$ in $(\mJ_{\w}-\mJ_4)/Symp_h$ in $(\mJ_{\w}-\mJ_4),$ we know that along the path $[0,1]$ the fiber are identified homeomorphically. Then each $ S^n $ must be homotopic in each fiber. This means that all vanishing cycles of all dimensions are trivial, and all emerging cycles are trivial.

  Then by Theorem A in \cite{Mei02}( statement see Theorem \ref{fib3}  ), the orbit projection $(\mJ_{\w}-\mJ_4)\to (\mJ_{\w}-\mJ_4)/Symp_h$ is a fibration with fiber $Symp_h$.

\end{proof}

\printbibliography

\end{document}